\documentclass[12 pt]{amsart}
\usepackage{amsmath,amsfonts,amssymb}
\usepackage{graphics,epsf,psfrag,epsfig}
\usepackage[extra]{tipa}
\usepackage[colorlinks=true, linkcolor=blue, citecolor=red]{hyperref}

\newtheorem{lemma}{Lemma}[section]
\newtheorem{theorem}[lemma]{Theorem}

\newtheorem{proposition}[lemma]{Proposition}

\theoremstyle{definition}
\newtheorem{remark}[lemma]{Remark}

\numberwithin{equation}{section}

\DeclareMathOperator{\Supp}{Supp}

\newcommand{\sumtwo}[2]{\sum_{\substack{#1 \\ #2}}} 

\makeatletter
\def\captionfont@{\footnotesize}
\def\captionheadfont@{\scshape}

\long\def\@makecaption#1#2{%
  \vspace{2mm}
  \setbox\@tempboxa\vbox{\color@setgroup
    \advance\hsize-6pc\noindent
    \captionfont@\captionheadfont@#1\@xp\@ifnotempty\@xp
        {\@cdr#2\@nil}{.\captionfont@\upshape\enspace#2}%
    \unskip\kern-6pc\par
    \global\setbox\@ne\lastbox\color@endgroup}%
  \ifhbox\@ne 
    \setbox\@ne\hbox{\unhbox\@ne\unskip\unskip\unpenalty\unkern}%
  \fi
  \ifdim\wd\@tempboxa=\z@ 
    \setbox\@ne\hbox to\columnwidth{\hss\kern-6pc\box\@ne\hss}%
  \else 
    \setbox\@ne\vbox{\unvbox\@tempboxa\parskip\z@skip
        \noindent\unhbox\@ne\advance\hsize-6pc\par}%
\fi
  \ifnum\@tempcnta<64 
    \addvspace\abovecaptionskip
    \moveright 3pc\box\@ne
  \else 
    \moveright 3pc\box\@ne
    \nobreak
    \vskip\belowcaptionskip
  \fi
\relax
}
\makeatother
\def\writefig#1 #2 #3 {\rlap{\kern #1 truecm
\raise #2 truecm \hbox{#3}}}

\newcommand{\ga}{\alpha}
\newcommand{\gb}{\beta}
\newcommand{\gd}{\delta}
\newcommand{\gep}{\varepsilon}       
\newcommand{\gp}{\varphi}

\newcommand{\gk}{\kappa}
\newcommand{\go}{\omega}

\newcommand{\gl}{\lambda}
\newcommand{\gL}{\Lambda}
\newcommand{\gs}{\sigma}

\newcommand{\NN}{\mathbb{N}}
\newcommand{\ZZ}{\mathbb{Z}}

\newcommand{\RR}{\mathbb{R}}

\newcommand{\PP}{\mathbb{P}}

\newcommand{\EE}{\mathbb{E}}

\newcommand{\mA}{\mathcal{A}}
\newcommand{\mB}{\mathcal{B}}
\newcommand{\mC}{\mathcal{C}}

\newcommand{\mE}{\mathcal{E}}
\newcommand{\mF}{\mathcal{F}}
\newcommand{\mG}{\mathcal{G}}
\newcommand{\mH}{\mathcal{H}}
\newcommand{\mI}{\mathcal{I}}

\newcommand{\mP}{\mathcal{P}}
\newcommand{\mQ}{\mathcal{Q}}

\newcommand{\mZ}{\mathcal{Z}}

\newcommand{\half}{\frac{1}{2}}

\newcommand{\tK}{\tilde{K}}

\newcommand{\tN}{\tilde{N}}

\newcommand{\td}{\tilde{\delta}}
\newcommand{\tga}{\tilde{\ga}}

\newcommand{\tps}{\tilde{\psi}}

\newcommand{\htau}{\hat{\tau}}

\newcommand{\hZ}{\hat{Z}}

\newcommand{\bP}{{\ensuremath{\mathbf P}} }
\newcommand{\bPs}{\bP_{\sigma}}
\newcommand{\bEs}{\bE_{\sigma}}
\newcommand{\bEt}{\bE_{\tau}}
\newcommand{\bPt}{\bP_{\tau} }
\newcommand{\bPst}{\bP_{\sigma,\tau} }
\newcommand{\bEst}{\bE_{\sigma,\tau} }

\newcommand{\bE}{{\ensuremath{\mathbf E}} }

\newcommand{\olm}{\overline{m}}

\newcommand{\tolm}{\tilde{\overline{m}}}
\newcommand{\tphi}{\tilde{\varphi}}

\newcommand{\ind}{\mathbf{1}}
\renewcommand{\a}{\mathrm{ann}}
\newcommand{\dd}{\mathrm{d}}
\newcommand{\Var}{\mathrm{Var}}

\newcommand{\tp}{\tilde{p}}

\renewcommand{\tilde}{\widetilde}
\newcommand{\que}{\rm que}
\newcommand{\ann}{\rm ann}

\title{Pinning of a renewal on a quenched renewal}

\author[K. S. Alexander]{Kenneth S. Alexander}
\address{Department of Mathematics, KAP 104\\
University of Southern California\\
Los Angeles, CA  90089-2532 USA}
\email{alexandr@usc.edu}

\author[Q. Berger]{Quentin Berger}
\address{LPMA, Universit\'e Pierre et Marie Curie\\
Campus Jussieu, case 188\\
4 place Jussieu, 75252 Paris Cedex 5, France}
\email{quentin.berger@upmc.fr}

\begin{document}

\begin{abstract}
We introduce the pinning model on a quenched renewal, which is an instance of a (strongly correlated) disordered pinning model. The potential takes value 1 at the renewal times of a quenched realization of a renewal process $\sigma$, and $0$ elsewhere, so nonzero potential values become sparse if the gaps in $\sigma$ have infinite mean.
The ``polymer'' -- of length $\sigma_N$ -- is given by another renewal $\tau$, whose law is modified by the Boltzmann weight $\exp(\beta\sum_{n=1}^N \mathbf{1}_{\{\sigma_n\in\tau\}})$. Our assumption is that $\tau$ and $\sigma$ have gap distributions with power-law-decay exponents $1+\alpha$ and $1+\tilde \alpha$ respectively, with $\alpha\geq 0,\tilde \alpha>0$. There is a localization phase transition: above a critical value $\beta_c$ the free energy is positive, meaning that $\tau$ is \emph{pinned} on the quenched renewal $\sigma$. We consider the question of relevance of the disorder, that is to know when $\beta_c$ differs from its annealed counterpart $\beta_c^{\rm ann}$. We show that $\beta_c=\beta_c^{\rm ann}$ whenever $ \alpha+\tilde \alpha \geq 1$, and $\beta_c=0$ if and only if the renewal $\tau\cap\sigma$ is recurrent. 
On the other hand, we show $\beta_c>\beta_c^{\rm ann}$ when $ \alpha+\frac32\, \tilde \alpha <1$. We give evidence that this should in fact be true whenever $ \alpha+\tilde \alpha<1$, providing examples for all such $ \alpha,\tilde \alpha$ of distributions of $\tau,\sigma$ for which $\beta_c>\beta_c^{\rm ann}$. We additionally consider two natural variants of the model: one in which the polymer and disorder are constrained to have equal numbers of renewals ($\sigma_N=\tau_N$), and one in which the polymer length is $\tau_N$ rather than $\sigma_N$.  In both cases we show the critical point is the same as in the original model, at least when $ \alpha>0$.
   \\[10pt]
  \\[5pt]
\end{abstract}

\maketitle


\section{Introduction}

\subsection{Motivations}
A variety of polymer pinning models have been studied in theoretical and mathematical physics in the past decades, see \cite{denHoll,GBbook, SFLN} for reviews. We introduce in this paper a new type of disordered pinning model that we call \emph{pinning on a  quenched renewal}. Before giving its definition, we recall two well-studied related models that motivate the introduction and the study of this new model.

\subsubsection*{Pinning on an inhomogeneous defect line}
The disordered pinning model was introduced by Poland and Scheraga \cite{PolanScheg} to model DNA denaturation, and it has recently been the subject of extensive rigorous mathematical studies, cf. \cite{GBbook,SFLN}. We recall its definition.

Let $\tau=\{\tau_0=0,\tau_1,\ldots\}$ be a discrete recurrent renewal process of law $\bP$, and let $\go=(\go_n)_{n\in\NN}$ be a sequence of IID random variables of law denoted $\PP$, with finite exponential moment $M(\gl)=\EE[e^{\gl \go}]<+\infty$, for $\lambda>0$ small enough. Then, for $\gb>0$ and $h\in\RR$, and a fixed realization of $\go$ (quenched disorder), the Gibbs measure is defined by
\begin{equation} \label{eq:RN}
\frac{\dd \bP_{N,h}^{\gb,\go}}{\dd \bP} (\tau) := \frac{1}{Z_{N,h}^{\gb,\go}}\ e^{\sum_{n=1}^N (\gb \go_n+h)\ \ind_{\{n\in\tau\}}}\, ,
\end{equation}
where $Z_{N,h}^{\gb,\go}$ is the \emph{partition function}, which normalizes $\bP_{N,h}^{\gb,\go}$ to a probability measure. This Gibbs measure corresponds to giving a (possibly negative) reward $\gb \go_n +h$ to the occurrence of a renewal point at $n$.  In this context it is natural to think of the polymer configuration as the space-time trajectory of some Markov chain, with the polymer interacting with a potential whenever it returns to some particular site, for example the origin; then $\tau$ represents the times of these returns.

This system is known to undergo a localization phase transition when $h$ varies, and much attention has been given to the question of disorder relevance, that is, whether the \emph{quenched} and \emph{annealed} systems (with respective partition functions $Z_{N,h}^{\gb,\go}$ and $\EE Z_{N,h}^{\gb,\go}$) have different critical behaviors. The so-called Harris criterion \cite{Harris} is used to predict disorder relevance, and its characterization in terms of the distribution of the renewal time $\tau_1$ has now been mathematically settled completely \cite{A06,AZ08,BL15,CdH10,DGLT07,GLT09,L10,T06}, the complete necessary and sufficient condition for critical point shift being given only recently, in \cite{BL15}.

An interesting approach to this problem is based on a large deviation principle for cutting words out of a sequence of letters \cite{BGdH08}, see \cite{CdH10}: quantities of the model, such as critical points, are expressed as variational formulas involving (quenched and annealed) rate functions $I^{\que}$ and $I^{\ann}$. In \cite{CdH10}, the authors consider a version of $\tau$ truncated at a level $T$, and it is implicitly shown that a lower bound on the critical point shift is, for all $\gb>0$, 
\begin{align}
\label{eq:CdH}
\lim_{T\to\infty} \frac{1}{m_T}\lim_{N\to\infty} \frac{1}{N} \bE_{\sigma} \EE  \log \bE\bigg[ \exp\bigg( \sum_{n=1}^{\sigma_N} 
   \Big[ \big( \gb \go_n - \log M(\gb) \big)\ind_{\{n\in\tau\}}
+\frac12 \log M(2\gb) \ind_{\{n\in \tau\cap\sigma\}} \Big] \bigg) \bigg] 
\end{align}
where $m_T$ is the mean inter-arrival time of the truncated renewal, and $\sigma$ is a quenched trajectory of a renewal with the same (truncated) law $\bP$; thus $\sigma_N \sim m_TN$.

Note that the first term of the summand on the right side of \eqref{eq:CdH} gives the quenched Hamiltonian from \eqref{eq:RN} calculated at the annealed critical point $u=-\log M(\beta)$.  The second term of the summand corresponds to a quenched system in which the disorder is sparse (at least without the truncation), in the sense that it is $0$ except at the sites of the quenched renewal $\sigma$, and these sites have a limiting density of $0$ when $\sigma_1$ has infinite mean, that is
\begin{equation} \label{sparse}
  \lim_{N\to\infty} \frac{ |\sigma \cap \{1,\dots,N\}| }{N} = 0 \quad \bPs-\text{a.s.}
\end{equation}
Since the quenched system is not pinned at the annealed critical point, the limit \eqref{eq:CdH} would be $0$, meaning no pinning, without that second term. The question is then roughly whether the additional presence of the very sparse disorder $\sigma$ is enough to create pinning  (this oversimplifies slightly, as we are ignoring the limit $T\to \infty$).  It therefore suggests the question, of independent interest, of whether such very sparse disorder creates pinning for arbitrarily small $\beta$, in the simplified context where the first term in the summand is absent and there is no truncation.  Our purpose here is to answer that question positively: under an appropriate condition on the tail exponents of the return times $\tau_1$ and $\sigma_1$, there is pinning for arbitrarily small $\beta$.

It should also be noted that the sum in \eqref{eq:CdH} is up to $\sigma_N$, not $N$ as in \eqref{eq:RN}.  To create positive free energy, $\tau$ must be able to hit an order-$N$ number of sites $\sigma_i$, which is impossible for a length-$N$ polymer with extremely sparse disorder satisfying \eqref{sparse}, see \cite{JRV05}.

\subsubsection*{Pinning on a Random Walk}
The Random Walk Pinning Model considers a random walk $X$ (whose law is denoted $\bP_X$), which is pinned to another independent (quenched) random walk $Y$ with distribution $\bP_Y$ identical to $\bP_X$, see \cite{BGdH11}. For any $\gb> 0$, the Gibbs measure is given by
\begin{equation}
\label{eq:RWPM}
\frac{\dd \bP_{N,\gb}^Y}{\dd \bP_X}(X) := \frac{1}{Z_{N,\gb}^Y} \, e^{\gb \sum_{n=1}^N \ind_{\{X_n=Y_n\}}}.
\end{equation}
This system undergoes a localization phase transition: above some critical value $\gb_c$, the random walk $X$ sticks to the quenched trajectory of $Y$, $\bP_Y$-a.s. Here $X$ and $Y$ are assumed irreducible and symmetric, with
$\log \bP_X(X_{2n}=0) \stackrel{n\to\infty}{\sim} -\rho \log n$: for example, with $\rho=d/2$ in the case of symmetric simple random walks on $\ZZ^d$.

One can compare this model to its annealed counterpart, in which the partition function is averaged over the possible trajectories of $Y$, $\bE_Y[Z_{N,\gb}^Y]$, with corresponding critical value $\gb_c^{\a}$. Non-coincidence of quenched and annealed critical points implies the existence of an intermediate phase for the long-time behavior of several systems (such as coupled branching processes, parabolic Anderson model with a single catalyst, directed polymer in random environment): we refer to \cite{BGdH11,BS08} for more details on the relations between these models and the random walk  pinning model.

The question of non-coincidence of critical points for pinning on a random walk has only been partially answered: it is known that $\gb_c>\gb_c^{\a}$ if $\rho>2$, see \cite{BGdH11}, and in the special cases of $d$-dimensional simple random walks, $\gb_c>\gb_c^{\a}$ if and only if $d\geq 3$  \cite{BT09,BS08,BS09} (note that the case $d\geq 5$ was already dealt with the case $\rho>2$). It is however believed that one has $\gb_c>\gb_c^{\a}$ whenever $\rho>1$, see Conjecture 1.8 in \cite{BGdH11}. 

\smallskip
The model we introduce now is related to this one, in the sense that we replace the random walks by renewals: we study a renewal $\tau$, which interacts with an object of the same nature, that is a quenched renewal $\sigma$.

\subsection{Pinning on a quenched renewal}
We consider two recurrent renewal processes $\tau$ and $\sigma$, with (possibly different) laws denoted respectively by $\bP_{\tau}$ and $\bP_{\sigma}$.
We assume that there exist $\ga\geq 0,\tilde\ga > 0$, and slowly varying functions $\varphi(\cdot),\tilde{\varphi}(\cdot)$ (see \cite{Bingham} for a definition) such that for all $n\geq 1$
\begin{equation}
\label{eq:alphas}
\begin{split}
\bP_{\tau}(\tau_1=n) & = \varphi(n) \, n^{-(1+\ga)}\quad \text{and}\quad
 \bP_{\sigma}(\sigma_1=n) = \tilde\varphi(n)\, n^{-(1+\tilde\ga)} \, . \end{split}
\end{equation}
Let
\begin{equation}\label{ddef}
  d_\tau = \min\Supp\tau_1, \qquad d_\sigma = \min\Supp\sigma_1,
\end{equation}
where $\Supp\chi$ denotes the support of the distribution of a random variable $\chi$.
At times we will add the assumption
\begin{equation}\label{nozero}
  d_\tau \leq d_\sigma,
\end{equation}
which ensures that certain partition functions carrying the restriction $\sigma_N=\tau_N$ cannot be 0, for sufficiently large $N$.


\smallskip
Let us write $|\tau|_n$ for $|\tau\cap \{1,\dots,n\}|$.  We consider the question of the pinning of $\tau$ by the renewal $\sigma$: the Hamiltonian, up to $N$ $\sigma$-renewals, is
$\mH_{N,\sigma}(\tau) := |\tau\cap\sigma|_{\sigma_N}$.
For ${\gb\geq 0}$ (the inverse temperature) and for a quenched trajectory of $\sigma$, we introduce the Gibbs transformation $\mP_{N,\gb}^{\sigma}$ of $\bP_{\tau}$ by
\begin{equation}
\label{defmu}
\frac{ {\rm d}\mP_{N,\gb}^{\sigma} }{{\rm d} \bP_{\tau}}:= \frac{1}{Z_{N,\gb}^{\sigma}} e^{\gb |\tau\cap\sigma|_{\sigma_N}} \ind_{\{\sigma_N\in\tau\}} \, ,
\end{equation}
where $Z_{N,\gb}^{\sigma}:=\bE_{\tau}\big[\exp(\gb |\tau\cap\sigma|_{\sigma_N}) \ind_{\{\sigma_N\in\tau\}} \big]$ is the partition function. 
Note that the resulting polymer is \emph{constrained}, meaning $\tau$ is required to have a renewal at the endpoint $\sigma_N$ of the polymer.

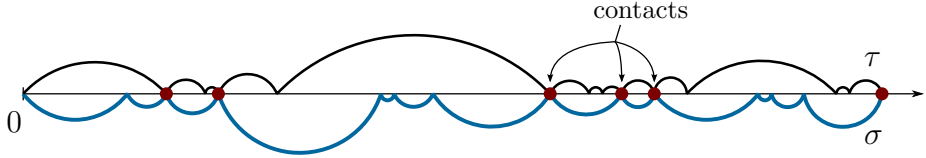
\begin{figure}[htbp]
\begin{center}
\setlength{\unitlength}{0.4cm}
\begin{picture}(30,4.5)(0,0)
\put(0,0){\includegraphics[width=12cm]{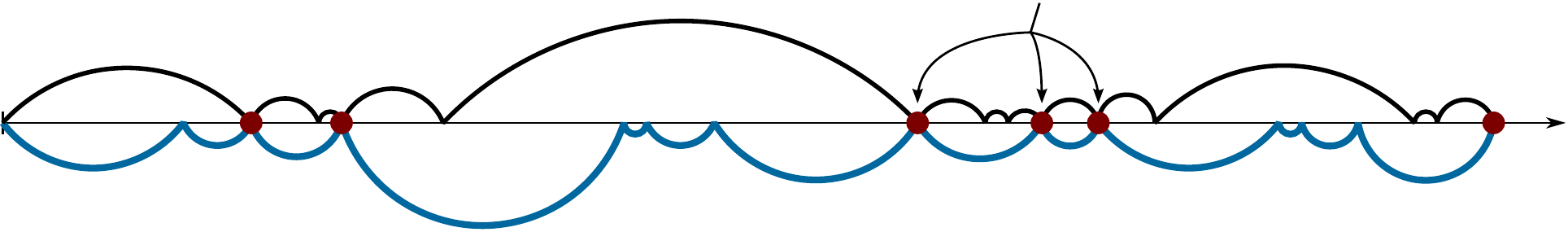} }
\put(-.5,0.8){$0$}
\put(28,3){$\tau$}
\put(28,0.4){$\sigma$}
\put(19,4.6){\footnotesize contacts}
\end{picture}
    \caption{\label{fig:model} \footnotesize The renewal $\sigma$ is quenched (here, we have $N=13$ $\sigma$-renewals), and the renewal $\tau$ collects a reward $\gb$ for each contact with a $\sigma$-renewal (here, $|\tau\cap\sigma|_{\sigma_N} = 6$).}
  \end{center}
  \vskip-.2in
\end{figure}

\begin{proposition}\label{defF}
 The quenched free energy, defined by
 \begin{equation}
 \mF(\gb):= \lim_{N\to\infty} \frac1N \log Z_{N,\gb}^{\sigma} = \lim_{N\to\infty} \frac1N \bE_{\sigma}\log Z_{N,\gb}^{\sigma} 
\end{equation}
exists and is constant $\bP_{\sigma}$-a.s.\ $\mF$ is non-negative, non-decreasing and convex. There exists $\gb_c =\gb_c(\bP_{\sigma})\geq 0$ such that  $\mF(\gb)=0$ if $\gb\leq \gb_c$ and $\mF(\gb)>0$ if $\gb> \gb_c$.
\end{proposition}

\begin{proof}
We have
$$Z_{N+M,\gb}^{\sigma} \geq \bE_{\tau}\big[\exp(\gb |\tau\cap\sigma|_{\sigma_N}) \ind_{\{\sigma_{N+M}\in\tau\}}\ind_{\{\sigma_N\in\tau\}} \big] = Z_{N,\gb}^{\sigma} Z_{M,\gb}^{\theta^{N} \sigma},$$
where $\theta$ is the ``shift operator applied to increments'':
\[
  (\theta\sigma)_i = \sigma_{i+1} - \sigma_1.
\]
Therefore the sequence $(\log Z_{N,\gb}^{\sigma})_{n\in\NN}$ is superadditive, and using Kingman's subadditive  Theorem \cite{King}, one gets the $\bP_{\sigma}$-a.s.\ existence of the limit in \eqref{defF}, and moreover
\begin{equation}
\label{subadd}
 \mF(\gb)=\mF(\gb,\bP_{\sigma})= \sup_{n\in \NN} \frac1N \bE_{\sigma}\log Z_{N,\gb}^{\sigma}.
\end{equation}
The non-negativity trivially follows from the fact that $\gb\geq 0$, and the convexity is classical and comes from a straightforward computation.
\end{proof}

A standard computation gives that, when the derivative of $\mF(\gb)$ exists, 
\begin{equation}
\frac{\partial}{\partial \gb} \mF(\gb) = \lim_{N\to\infty} \mP_{N,\gb}^{\sigma} \bigg[  \frac{1}{N}\ |\tau\cap\sigma|_{\sigma_N} \bigg],
\end{equation}
where $\mP_{N,\gb}^{\sigma}$ also denotes the expectation with respect to the measure defined in \eqref{defmu}.
Therefore, when $\mF(\gb)>0$, a positive proportion of $\sigma$-renewal points are visited by $\tau$, and $\gb_c$ marks the transition between a delocalized and localized phase.

\begin{remark}\rm
\label{rem:constrainedvsfree}
We have chosen to work with the constrained version of the model, by including $\ind_{\{\sigma_N\in\tau \}}$ in the partition function.
We stress that the free energy is the same as for the corresponding ``free'' partition function,
\begin{equation}
Z_{N,\gb}^{\sigma,{\rm free}}:= \bE_{\tau} \left[  e^{\gb \sum_{n=1}^N \ind_{\{\sigma_n\in\tau\}}}\right].
\end{equation}
Indeed, let $K$ be such that $\bP_\tau(\tau_1=k)>0$ for all $k\geq K$.  Consider a trajectory $\tau$ for which the last renewal in $[0,\sigma_N-K]$ is at some point $\sigma_N-k$.  We can map $\tau$ to a constrained trajectory $\hat{\tau}$ by removing all renewals of $\tau$ in $(\sigma_N-K,\sigma_N)$ and adding a renewal at $\sigma_N$.  This is a $2^K$-to-1 map, and the Boltzmann weight of $\hat{\tau}$ is smaller than that of $\tau$ by a factor no smaller than $e^{-\beta K} \bP_\tau(\tau_1=k) \geq e^{-\beta K} \min_{K\leq j\leq\sigma_N}\bP_\tau(\tau_1=j) \geq c\sigma_N^{-(2+\ga)}$ for some constant $c>0$.  Hence
there exists $C>0$ such that for all large $N$,
\begin{equation}
\label{eq:freevsconstrained}
Z_{N,\gb}^{\sigma}\leq Z_{N,\gb}^{\sigma,{\rm free}} \leq C (\sigma_N)^{2+\ga} Z_{N,\gb}^{\sigma}.
\end{equation}
It is straightforward from \eqref{eq:alphas} (and because we impose $\tga>0$) that $\lim_{N\to\infty} N^{-1} \log \sigma_N =0$ $\bP_{\sigma}$-a.s., which with \eqref{eq:freevsconstrained} gives $\lim_{N\to\infty} N^{-1} \log Z_{N,\gb}^{\sigma, {\rm free}} = \mF(\gb)$, $\bP_{\sigma}$-a.s.
\end{remark}

\subsection{The annealed model}
\label{annealedsec}
We will compare the quenched-renewal pinning model to its annealed counterpart, with partition function $\bE_{\sigma} Z_{N,\gb}^{\sigma}  = \bE_{\tau,\sigma}\big[ e^{\gb |\tau\cap \sigma|_{\sigma_N}} \ind_{\{\sigma_N\in\tau\}}\big]$.  
The annealed free energy is defined by
\begin{equation}
\mF^{\a}(\gb):= \lim_{N\to\infty}\frac1N \log \bE_{\sigma} Z_{N,\gb}^{\sigma}\, .
\end{equation}
The existence of the annealed free energy is straightforward, using the superadditivity of $\log \bE_{\sigma} Z_{N,\gb}^{\sigma}$. 

Note that this model does not treat $\tau$ and $\sigma$ symmetrically. Closely related is what we will call the \emph{homogeneous doubled pinning model}, in which the length of the polymer is fixed, rather than the number of renewals in $\sigma$: the partition function and free energy are 
\[
  Z_{n,\gb}^{\rm hom}:=\bE_{\tau,\sigma}\big[ e^{\gb |\tau\cap \sigma|_n} \ind_{\{n\in\tau\cap\sigma\}}\big], \qquad \mF^{\hom}(\gb) = \lim_{n\to\infty} \frac1n \log Z_{n,\gb}^{\rm hom}\, ,
  \]
and the corresponding critical point is denoted $\beta_c^{\rm hom}$. This model is exactly solvable, see \cite[Ch. 2]{GBbook}, and in particular, its critical point is $\beta_c^{\rm hom}=-\log \bP_{\tau,\sigma}\big( (\tau\cap\sigma)_1<\infty) = \log\big( 1+\bE_{\sigma,\tau}[|\tau\cap\sigma|]^{-1}\big)$, with the convention that $\frac{1}{\infty}=0$.

\begin{proposition}
\label{prop:annealed}
The annealed system is solvable, in the sense that for $\gb>0$, $\mF^{\a}(\gb)$ is the solution $\mathtt{F}$ of
\begin{equation}
\label{eq:fann}
\sum_{n\ge 1} e^{- \mathtt{F} n} \bPst (\sigma_n \in \tau) = \frac{1}{e^{\gb} -1}
\end{equation}
if a solution exists, and $0$ otherwise. There exists $\gb_c^{\a} \ge 0$ such that $\mF^{\a} (\gb)=0$ if $\gb\leq \gb_c^{\a} $ and $\mF^{\a} (\gb)>0$ if $\gb> \gb_c^{\a} $. In fact, we have $\gb_c^{\a} =\gb_c^{\rm hom}$, satisfying $e^{\gb_c^{\a}}-1 = \big( \sum_{n\ge 1} \bPst (\sigma_n \in \tau)\big)^{-1}$, and thus $\gb_c^{\a}=0$ if and only if $\tau\cap \sigma$ is recurrent.

Additionally, under \eqref{eq:alphas} with $\tga>0$, there exists a slowly varying function $L(\cdot)$ such that
\begin{equation}
\label{fanncritic}
\mF^{\a} \big( \gb\big) \stackrel{ \gb \downarrow \gb_c^{\a}}{\sim} L\left( \frac{1}{\gb-\gb_c^{\a}} \right)\, \big(\gb-\gb_c^{\a} \big)^{ 1\vee \frac{1}{|\ga^*|} }, \text{  with } \ga^* := \frac{1- \ga\wedge 1 - \tilde\ga\wedge 1}{ \tilde\ga \wedge 1}\in [-1,+\infty) . 
\end{equation}
In the case $\ga^* =0$, one interprets \eqref{fanncritic} as saying $\mF(\gb_c^\a+u )$ vanishes faster than any power of $u$.
\end{proposition}

The proof of this proposition uses standard techniques, and is postponed to Appendix~\ref{app:K*}. It relies on a rewriting of the partition function, together with an estimate on the probability $\bPst(\sigma_n\in\tau)$ that we now present (its proof is also postponed to Appendix \ref{app:K*}).
\begin{lemma}
\label{lem:Kstar}
Under \eqref{eq:alphas}, with $\tga>0$, there exists a slowly varying function $\gp^*(\cdot)$ such that, for $\ga^*$ from \eqref{fanncritic},
\begin{equation}
\bPst(\sigma_n\in\tau)  = \gp^*(n) \, n^{-(1+\ga^*)} \, .
\end{equation}
\end{lemma}

\smallskip
From  Jensen's inequality we have $\mF(\gb)\leq \mF^{\a}(\gb)$,
so that $\gb_c\geq \gb_c^{\a}$.
When $\tau\cap\sigma$ is recurrent, $\gb_c^{\a}=0$, and our main theorem will say that $\gb_c=0$ as well. In the transient case we have $\gb_c^{\a}>0$ and we can ask whether $\gb_c= \gb_c^{\a}$ or not.

\subsection{On the recurrence/transience of $\tau\cap\sigma$}
\label{sec:recur}

The criterion for the recurrence of the renewal $\tau\cap \sigma$ is that
\begin{equation} \label{recur}
  \bE_{\sigma}\bE_{\tau}[|\tau\cap\sigma|] = \sum_{n\in\NN} \bP_{\sigma}(n\in\sigma) \bP_{\tau}(n\in\tau) = +\infty .
\end{equation}
Under the assumption \eqref{eq:alphas}, the exact asymptotic behavior of $\bPt(n\in\tau)$ is known. We write $a\wedge b$ for $\min(a,b)$.
Defining the truncated mean and tail
\[
\olm(n) = \bE_\tau\big( \tau_1 \wedge n \big), \quad r_n := \bPt(\tau_1>n)\, ,
\]
we have
\begin{equation}
\label{Doney}
\bPt(n\in\tau)  \stackrel{n\to\infty}{ \sim} 
\begin{cases}
 r_n^{-2}n^{-1}\gp(n) &\text{if } \ga=0\, ,\\
 \frac{\ga \sin(\pi \ga)}{ \pi} n^{-(1-\ga)} \gp(n)^{-1} &\text{ if } 0<\ga<1  \, ,\\
 \olm(n)^{-1} &\text{ if } \ga=1, \bEt[\tau_1]=+\infty \, ,\\
\bEt[\tau_1]^{-1} &\text{ if } \bEt[\tau_1]<+\infty\, .
\end{cases}
\end{equation}
The first is from \cite{Nagaev}, the second is due to Doney \cite[Thm. B]{Doney}, the third is due to Erickson \cite[Eq. (2.4)]{Eric}, and the fourth is the classical Renewal Theorem. 

Applying \eqref{Doney} to $\tau\cap\sigma$ in \eqref{recur}, we see that $\ga+\tilde\ga>1$ implies that $\tau\cap \sigma$ is recurrent, and also that $\ga+\tilde\ga<1$ implies that $\tau\cap\sigma$ is transient. The case $\ga+\tilde\ga=1$ is marginal, depending on the slowly varying functions $\varphi(\cdot)$ and $\tilde\varphi(\cdot)$: if also $\alpha,\tilde\alpha>0$ then $\tau\cap\sigma$ is recurrent if and only if  
\begin{equation}
\label{recur2}
 \sum_{n\in\NN} \frac{1}{n\varphi(n)\tilde\varphi(n)} =+\infty.
\end{equation}
 For the case $\alpha=0,\tilde\alpha=1$, we have recurrence if and only if 
\begin{equation}
\label{recur3}
   \sum_{n\in\NN} \frac{\varphi(n)}{n\, \tilde{\overline{m}}(n)\bPt(\tau_1\geq n)^2} = +\infty,
\end{equation}
 where $\tilde{ \overline{m}}(n) = \bEs(\sigma_1 \wedge n)$, which is slowly varying since $\tilde\alpha=1$.

\smallskip
Since $\bEst[|\tau\cap\sigma|] =\sum_{n\ge1} \bPst(\sigma_n\in\tau)$, for $\ga^*$ from Lemma \ref{lem:Kstar} we have that $\tau\cap \sigma$ is recurrent if $\ga^*<0$ and transient if $\ga^*>0$. In the case $\ga^*=0$, either can hold depending on the slowly varying function $\gp^*$.  

\section{Main results}

\subsection{Results and comments}
Our main result says that in the case where $\tau\cap\sigma$ is recurrent, or transient with $\ga+\tga=1$, the quenched and annealed critical points are equal (so both are $0$ in the recurrent case). 
When $\ga+\tga <1$, a sufficient condition for unequal critical points involves the exponent $\ga^*$ appearing in Lemma \ref{lem:Kstar}.

\begin{theorem}
\label{thm:main}
Assume \eqref{eq:alphas}.
\begin{gather}
\text{If } \tilde\ga>0 \text{ and } \ga+\tilde\ga\geq 1\ (i.e.\ \ga^* \le 0), \quad\text{ then } \gb_c=\gb_c^{\a} ;\label{eq:main1}\\
\text{If } \tilde\ga>0, \, \ga+\tga<1 \ (i.e.\ \ga^* > 0), \text{ and } \ga^*>\half, \quad\text{ then } \gb_c>\gb_c^{\a} \label{eq:main2}.
\end{gather}
In the case $\tga>0,\ga+\tilde\ga\geq 1$, in addition to having equal critical points, the quenched and annealed systems have the same critical behavior:
\begin{equation}
\label{eq:main3}
\lim_{\gb\downarrow \gb_c } \frac{\log \mF(\gb) }{ \log(\gb-\gb_c)} = \frac{1}{|\ga^*|}, \qquad \text{if } \tga>0, \tga+\ga\ge 1 \ (i.e.\ \ga^* \le 0)\, .
\end{equation}
\end{theorem}

Note that $\ga^*>1/2$ is equivalent to $\ga+\frac32\,  \tilde\ga < 1$.  We suspect that this condition can be removed in \eqref{eq:main2}.
In that direction, our next result gives examples of distributions for $\tau,\sigma$ with any prescribed asymptotic behavior with $\ga+\tga<1$ (\textit{i.e.}\ any $\gp,\tilde\gp$ and $\ga,\tga>0$ with $\ga+\tga<1$) for which $\gb_c>\gb_c^{\a}$.
It shows that, if the equality of critical points is determined entirely by the exponents $\ga,\tilde\ga$ and the asymptotics of $\gp,\tilde\gp$, then for $\ga,\tilde\ga>0$ \eqref{eq:main2} is true without the condition $\ga^*>1/2$.

\begin{theorem}
\label{thm:exist}
For any $\alpha,\tilde\alpha >0$ with $\ga+\tga<1$ and any slowly varying functions $\gp(\cdot),\tilde\gp(\cdot)$, there exist distributions for $\tau$ and $\sigma$ satisfying  $\bPt(\tau_1=n) \stackrel{n\to\infty}{\sim} \gp(n) n^{-(1+\ga)}$ and $\bPs(\sigma_1=n) \stackrel{n\to\infty}{\sim} \tilde\gp(n) n^{-(1+\tga)}$,
 with $\gb_c > \gb_c^{\a}$. 
\end{theorem}

We expect that, following the same scheme of proof, one can extend Theorem \ref{thm:exist} to the case $\ga=0$, $\tga\in(0,1)$. However, in the interest of brevity we do not prove it here, since it would require separate estimates from \cite{ABzero}.

\medskip
Let us now make some comments about these results.

\smallskip
{\bf 1.} Since the renewal $\tau\cap \sigma$ can be recurrent only if $\tilde\ga+\ga\geq 1$, one has from \eqref{eq:main1} that $\gb_c=0$ if and only if $\gb_c^{\a}=0$, that is if and only if $\tau\cap \sigma$ is recurrent. This is notable in that $|\tau\cap\sigma| = +\infty\ \bPst-$a.s.~is enough for $\tau$ to be pinned on a quenched trajectory of $\sigma$ for all $\gb>0$, even though a typical $\sigma$-trajectory is very sparse.

\smallskip
{\bf 2.} When $\ga+\tilde\ga=1$, $\tau\cap\sigma$ can be either recurrent or transient depending on \eqref{recur2}-\eqref{recur3}. Thus we have examples with transience where the critical points are positive and equal.

\smallskip
{\bf 3.} If $\bE_{\sigma}[\sigma_1]<+\infty$, we have a system of length approximately $\bE_{\sigma}[\sigma_1] N$. Here we already know from \cite{JRV05} that the renewal $\tau$ is pinned for all $\gb>0$, since there is a positive density of $\sigma$-renewal points in the system.

\smallskip
{\bf 4.} We note in Remark \ref{rem:monotone} that if  $n\mapsto \bP(n\in\tau)$ is non-increasing, then some calculations are simplified, and we can go slightly beyond the condition $\ga^*>1/2$, to allow $\ga^*=1/2$ with $\lim_{k\to\infty} \tilde \varphi(k)^{(1-\ga)/\tilde{\ga}} \bar\varphi(k) =+\infty$. Here $\bar\gp(k)=\gp(k)^{-1}$ if $\ga\in(0,1)$ and $\bar\gp(k)=\gp(k)/\bPt(\tau_1>k)^2$ if $\ga=0$. The monotonicity of $n\mapsto \bP(n\in\tau)$ follows easily if the renewal points $j$ of $\tau$ correspond to times $2j$ of return to $0$ of a nearest-neighbor Markov chain (birth-death process.) Indeed, one can then couple two trajectories (one non-delayed, the other one delayed by $2k$ for some  $k\geq 1$) the first time they meet, and show that  $\bP(k+n\in\tau)\leq \bP(n\in\tau)$ for any $n\geq 1$.

\smallskip
{\bf 5.}
The random walk pinning model \eqref{eq:RWPM} is quite analogous to our model, replacing only $\ind_{\{\sigma_n\in\tau\}}$ by $\ind_{\{X_n=Y_n\}}$.
In our model, the decay exponent $1+\ga^*$ of the probability $\bPst(\sigma_n\in\tau)$ corresponds to the parameter $\rho$ in \cite{BGdH11} (the decay exponent of $\bP_{X,Y}(X_{n}=Y_n)$.)
Theorem~\ref{thm:main} is in that sense stronger than its counterpart in the random walk pinning model, since $\ga^*>1/2$ translates into $\rho>3/2$, compared to $\rho>2$ in \cite{BGdH11}.
Moreover, Theorem \ref{thm:exist} supports the conjecture that the quenched and annealed critical points differ whenever $\rho>1$, see Conjecture 1.8 in \cite{BGdH11}.

\smallskip
{\bf 6.} One may argue that, in view of the annealed critical exponent $1/\ga^*$, the condition $\ga^*>1/2$ is somehow reminiscent of the Harris criterion for disorder relevance/irrelevance.  
Applying this criterion without further precaution tells that disorder should be irrelevant if $1/\ga^* >2 $ and relevant if $1/\ga^* < 2$.
However, in the Harris criterion formulation, disorder is irrelevant if it does not change the critical behavior \emph{provided that the disorder strength is sufficiently small}. 
Compared to the pinning model (where Harris' criterion has been confirmed, as mentioned in Introduction), the difference here -- as well as for the Random Walk Pinning Model -- is that one cannot tune the stength of the disorder: a quenched realization of $\sigma$ is given, and there is no extra parameter to play with to get an arbitrarily small disorder strength.  This is for example what makes it easier for us to prove a critical point shift in the case $\ga^*>1/2$, since according to Harris' prediction, disorder should be relevant whatever the disorder strength is -- which is here though to be fixed, positive. The case $\ga^*\in(0,1/2]$ is therefore more subtle, however our Theorem \ref{thm:exist} provides examples where $0<\ga^* \le 1/2$ but where a shift in the critical points still occur (so to speak, disorder is strong enough to shift the critical point).

\subsection{Variations of the main model: balanced and elastic polymers}\label{variants}

Pinning, in the quenched or annealed model, means that $\tau$ hits a positive fraction of the sites in the quenched renewal $\sigma$, up to some $\sigma_N$. The number of renewals in $\tau$ is unlimited.  We can alternatively consider the \emph{balanced polymer} with $\tau$ constrained to satisfy $\tau_N=\sigma_N$.

A second alternative pinning notion asks whether a positive fraction of $\tau$ renewals hit sites of $\sigma$, by considering a polymer of length $\tau_N$ instead of $\sigma_N$.  Physically this may be viewed as an \emph{elastic polymer}, since $\tau$ has a fixed number $N$ of monomers and needs to stretch or contract to place them on renewals of $\sigma$.

\subsubsection*{\bf First variation: balanced polymer, $\tau_N=\sigma_N$}
We first consider $\tau$ constrained to have $\tau_N=\sigma_N$, since in that case, both pinning notions are equivalent.


We introduce
\begin{equation}
\hZ_{N,\gb}^{\sigma} = \bE_{\tau}\big[\exp(\gb |\tau\cap\sigma|_{\sigma_N}) \ind_{\{\tau_N=\sigma_N\}} \big],
\end{equation}
and 
\begin{equation}
\label{defFtilde}
 \hat\mF(\gb):= \lim_{N\to\infty} \frac1N \log \hZ_{N,\gb}^{\sigma} = \lim_{N\to\infty} \frac1N  \bEs[\log \hZ_{N,\gb}^{\sigma}] \quad \bPs-a.s. 
\end{equation}
The proof of Proposition \ref{defF} applies here, establishing the existence and non-randomness of this limit.

For the balanced polymer we need the condition \eqref{nozero}, for otherwise the partition function is 0 with positive probability (whenever $\sigma_N=Nd_\sigma$, see \eqref{ddef}), and therefore $N^{-1} \bEs[\log \hZ_{N,\gb}^{\sigma}]  = -\infty$ for all $N\in\NN$.

Further, suppose $\bEs[\sigma_1]<\bEt[\tau_1]$. Then $\sigma_N\sim \bEs[\sigma_1]N$ a.s., so $\bPt(\tau_N = \sigma_N)$ decays exponentially in $N$.  It is easy to see that therefore $\hat\mF(\gb)<0$ for small $\beta$. In this sense the constraint $\tau_N = \sigma_N$ dominates the partition function, which is unphysical, so we assume $\bEs[\sigma_1]\geq\bEt[\tau_1]$, which ensures $\hat\mF(\gb)\geq 0$ for all $\beta>0$.
There then exists a critical point $\hat \gb_c :=\inf\{\gb:\hat\mF(\gb)>0\}$ such that $\hat\mF(\gb)>0$ if $\gb>\hat\gb_c$ and $\hat\mF(\gb)=0$ if $\gb\leq \hat\gb_c$.
The positivity of $\hat\mF(\gb)$ implies that $\tau$ visits a positive proportion of $\sigma$-renewal points, and also that a positive proportion of $\tau$-renewals sit on a $\sigma$-renewal point.
Clearly, one has that $Z_{N,\gb}^{\sigma}\geq \hZ_{N,\gb}^{\sigma}$, so that $\hat \gb_c\geq \gb_c$. The following proposition establishes equality, and is proved in Section \ref{sec:othermodels}.

\begin{proposition}
\label{prop:compareFtilde} 
Assume \eqref{eq:alphas} and \eqref{nozero}. If $\bE_{\sigma}[\sigma_1]\geq \bE_{\tau}[\tau_1]$, and in particular if $\bE_{\sigma}[\sigma_1]=+\infty$, then $\hat\gb_c = \gb_c$. 
\end{proposition}

Hence, here again, $\hat\gb_c=0$ if and only if $\tau\cap\sigma$ is recurrent.

\subsubsection*{\bf Second variation: elastic polymer of length $\tau_N$}
In the elastic polymer, pinning essentially means that a positive fraction of $\tau$-renewals fall on $\sigma$-renewals.  The partition function is
\begin{equation}
\bar Z_{N,\gb}^{\sigma} = \bE_{\tau}\big[\exp(\gb |\tau\cap\sigma|_{\tau_N}) \big].
\end{equation}
Standard subadditivity methods to establish existence of $\lim_{N\to\infty} N^{-1} \log \bar Z_{N,\gb}^{\sigma}$ do not work here, but we can consider the $\liminf$ instead, since its positivity implies pinning in the above sense.
We therefore define 
\begin{equation}
\label{defFbar}
 \bar\mF(\gb):= \liminf_{N\to\infty} \frac1N  \bEs \log \bar Z_{N,\gb}^{\sigma} ,
\end{equation}
which is non-decreasing in $\gb$, and the critical point $\bar \gb_c := \inf \{\gb : \bar \mF(\gb)>0 \}$. 
Compared to the balanced polymer, here it is much less apparent \emph{a priori} that we should have $\bar \gb_c \geq \gb_c $.  It could be favorable for $\tau$ to jump far beyond $\sigma_N$ to reach an exceptionally favorable stretch of $\sigma$, before $\tau_N$.  The original model \eqref{defmu} does not permit this type of strategy, so the question is whether this allows pinning in the elastic polymer with $\beta<\beta_c$.  The answer is no, at least if $\ga>0$, or if $\ga=0$ and $\tga\geq 1$, as the following shows; the proof is in Section \ref{sec:othermodels}.  We do not know whether $\bar \gb_c=\gb_c $ when $\ga=0$ and $\tga<1$.

\begin{proposition}
\label{prop:compareFbar}
Assume \eqref{eq:alphas} with $\tga>0$:

(i) $\bar \gb_c \leq \gb_c$ ;

(ii) $\bar \gb_c = \gb_c$ in both cases $\ga>0$ and $\ga=0,\, \tga\geq 1$.
\end{proposition}
 In particular, $\bar \gb_c=0$ whenever $\tau\cap\sigma$ is recurrent. 

\subsection{Organization of the rest of the paper}
We now present a brief outline of the proofs, and how they are organized.

\smallskip
In Section \ref{sec:proof1}, we prove the first part \eqref{eq:main1} of Theorem \ref{thm:main}, to establish pinning of the quenched system when $\beta>\gb_c^{\a}$, with $\ga+\tga\geq 1$. We use a rare-stretch strategy to get a lower bound on the partition function:
the idea is to consider trajectories $\tau$ which visit exponentially rare stretches where the renewals of $\sigma$ have some (not-well-specified) special structure which reduces the entropy cost for $\tau$ to hit a large number of $\sigma$ sites. In Section \ref{sec:preliminaries}, we classify trajectories of $\sigma$ according to the size of this entropy cost reduction, then select a class which makes a large contribution to the annealed free energy $\mF^{\a}(\gb)>0$ (we will fix $\gb>\gb_c^{\a}$). We call trajectories in this class \emph{accepting}, and the localization strategy of Section \ref{sec:localization} consists in visiting all accepting segments of $\sigma$. More detailed heuristics are presented in Section \ref{sec:sketch1}. 
We also prove \eqref{eq:main3} in Section \ref{sec:finitevol}, thanks to a finite volume criterion.

\smallskip
In Section \ref{sec:proof2}, we prove the second part \eqref{eq:main2} of Theorem \ref{thm:main}. First, we rewrite the partition function as that of another type of pinning model --- see Section \ref{sec:chgpartitionfunction}; this reduces the problem to a system at the annealed critical point $\gb_c^\a>0$.
We then employ a fractional moment method, combined with coarse-graining, similar to one developed for the disordered pinning model \cite{BL15,DGLT07,GLT08,GLT09}, and later used for the random walk pinning model \cite{BT09,BS08,BS09}.
In Sections \ref{sec:fracmoment}-\ref{sec:coarsegraining}, we show how one can reduce to proving a finite-size estimate of the fractional moment of the partition function.
Then to show that the fractional moment of the partition function is small, we describe a general change of measure argument in Section~\ref{sec:chgmeasure}. The change of measure is based on a set $J$ of trajectories $\sigma$, defined in Section \ref{sec:eventJ}, which has high probability, but becomes rare under certain conditioning which, roughly speaking, makes the partition function $Z_{N,\gb}^{\sigma}$ large. Identification of such an event is the key step; our choice of $J$ only works in full generality for $\ga^*>1/2$, necessitating that hypothesis.

\smallskip
In Section \ref{sec:example}, we give another sufficient condition for differing critical points, see \eqref{less1}. Then, we construct examples of $\tau,\sigma$ for which this condition holds, to prove Theorem \ref{thm:exist}.

\smallskip
In Section \ref{sec:othermodels}, we study the variants of the model introduced in Section \ref{variants}, and prove Propositions \ref{prop:compareFtilde}-\ref{prop:compareFtilde}. We mainly use some particular localization strategies to obtain lower bounds on the free energy of these two models.

\smallskip
The Appendix is devoted to the proof of technical lemmas on renewal processes.

\subsection{Notations}
\label{sec:notations}

Throughout the paper, $c_i$ are constants which depend only on the distributions of $\tau$ and $\sigma$, unless otherwise specified.  

Recall the truncated means
\[
  \olm(n) = \bEt(\tau_1\wedge n), \quad \tolm(n) = \bEs(\sigma_1\wedge n).
\]
When $\ga,\tga>0$ we denote by $a_n$ the ``typical size'' of $\tau_n$ (in the sense that $(\tau_n/a_n)_{n\ge 1}$ is tight), $b_n$ the typical size of $\sigma_n$, $A_n$ the typical size of $\max\{\tau_1,\tau_2-\tau_1,\dots,\tau_n-\tau_{n-1}\}$, and $B_n$ the typical size of $\max\{\sigma_1,\sigma_2-\sigma_1,\dots,\sigma_n-\sigma_{n-1}\}$.  More precisely, $B_n$ is defined (up to asymptotic equivalence) by 
\begin{align}\label{defAnBn}
  B_1=d_\sigma, \qquad B_n\geq d_\sigma, \qquad B_n^{\tga} \tphi(B_n)^{-1} \stackrel{n\to\infty}{ \sim} n,
\end{align}
(see \eqref{ddef}) and $b_n$ by
\begin{align}\label{defanbn}
  b_n=B_n \text{  if } 0<&\tga<1, \quad 
    b_n\tolm(b_n)^{-1} \stackrel{n\to\infty}{ \sim} n \text{  if $\tga=1$ and } \bEs[\sigma_1]=\infty, \quad
    \notag\\
  &b_n = \bEs[\sigma_1] n \text{  if } \bEs[\sigma_1]<\infty, 
\end{align}
and $A_n,a_n$ are defined analogously for the distribution of $\tau_1$.
The following can be found in~\cite{Bingham}.
These sequences are unique up to asymptotic equivalence, and in the infinite-mean case with $0<\ga,\tga\leq 1$,
there exist slowly varying functions $\psi$ and $\tilde\psi$, whose asymptotics are explicit in terms of $\ga,\tga$ and $\gp,\tilde\gp$, such that
\begin{equation}
\label{defanbn2}
a_n = \psi(n) \, n^{1/\ga} \, , \qquad b_n=\tilde \psi(n) \, n^{1/\tga}.
\end{equation}
The sequences $A_n^{-1}\max\{\tau_1,\tau_2-\tau_1,\dots,\tau_n-\tau_{n-1}\}$ and $B_n^{-1}\max\{\sigma_1,\sigma_2-\sigma_1,\dots,\sigma_n-\sigma_{n-1}\}$ converge in distribution, and the limits have strictly positive densities on $(0,\infty)$.
When the corresponding exponents $\ga,\tga $ are in $ (0,1)$, the ratios $\tau_n /a_n$ and $\sigma_n/b_n$ converge to stable distributions, which have strictly positive densities on $(0,\infty)$.
When the exponent $\ga,\tga \in [1,\infty)$, the same ratios converge in probability to $1$. 
This follows from the law of large numbers when the mean is finite, and  from Theorem 8.8.1 in \cite{Bingham} when $\ga=1$. 

Define
\begin{equation}\label{dbardef}
  \bar d_\tau = \min\{n\geq 1: \bPt(k\in\tau)>0 \quad\text{for all } k\geq n\},
  \end{equation}
and $\bar d_\sigma$ similarly.

\section{Rare-stretch strategy: proof of the first part of Theorem \ref{thm:main}}
\label{sec:proof1}

\subsection{Sketch of the proof}
\label{sec:sketch1}

In the recurrent case, we may use super-additivity argument to obtain directly a lower bound on the quenched free energy via a finite-volume criterion, as presented in Section \ref{sec:finitevol}. This gives us that $\gb_c=0$ when $\ga^*<0$, together with a lower bound on the order of the quenched phase transition. However, it is not helpful when $\ga^*=0$ and in particular in the transient case, when we try to prove that $\gb_c=\gb_c^{\a}>0$, and we need a novel idea.

We use a rare-stretch localization strategy to show that, when $\ga+\tga\geq 1$, we have $\mF(\gb)>0$ for all $\gb>\gb_c^{\a}$.
The idea is to identify ``accepting'' stretches of $\sigma$ that we then require $\tau$ to visit; the contribution to the free energy from such trajectories $\tau$ is shown sufficient to make the free energy positive. Here, the definition of an accepting stretch is made without describing these stretches explicitly, in the following sense: for fixed large $\ell$, different configurations $\{\sigma_1,\dots,\sigma_\ell\}$ are conducive to large values of $|\tau\cap\sigma|_{\sigma_\ell}$ to different degrees, as measured by the probability $\bP_{\tau}\big(|\tau\cap\sigma|_{\sigma_\ell} \geq \td \ell )$ for some appropriate $\td$.  We divide (most of) the possible values of this probability into  a finite number of intervals, which divides configurations $\sigma$ into corresponding classes; we select the class responsible for the largest contribution to the annealed partition function, and call the corresponding trajectories $\sigma$ accepting. This is done in Section \ref{sec:preliminaries}. We then get in Section \ref{sec:localization} a lower bound for the quenched partition function by considering trajectories of length $N\gg \ell$ which visit all accepting stretches of $\sigma$ of length~$\ell$.

\subsection{Finite-volume criterion and proof of \eqref{eq:main3}}
\label{sec:finitevol}

The finite-volume criterion comes directly from \eqref{subadd}: if there exists some $N=N_{\gb}$ such that $\bEs[\log Z_{N_{\gb},\gb}^{\sigma}] \ge 1$, then we have that $\mF(\gb) \ge N_{\gb}^{-1}>0$. 
We will use this simple observation to obtain a lower bound on the free energy in the case $\ga^*<0$, showing both $\gb_c=0$ and  \eqref{eq:main3}.

We write $Z_{N,\gb}^{\sigma} = \bPt(\sigma_N\in\tau) \bEt\big[e^{\gb |\tau\cap\sigma|_{\sigma_N}} |\sigma_N\in\tau \big]$, so by Jensen's inequality we get that
\begin{equation}
\label{eq:finitevol}
\bEs \log Z_{N,\gb}^{\sigma} \ge  \gb \bEs\Big[ \bEt\big[ |\tau\cap\sigma|_{\sigma_N} \, |\,  \sigma_N\in\tau\big]  \Big] +\bEs\log \bPt(\sigma_N\in \tau).
\end{equation}
Let us now estimate the terms on the right. For the second term,  we get thanks to \eqref{Doney} that for large $N$ (hence large $\sigma_N$), $\log \bPt(\sigma_N\in \tau) \ge  - 2 \log \sigma_N$. Then, we have $\bEs[\log \sigma_N ] \le 2 \tga^{-1} \log N$ for $N$ large enough, as established below in \eqref{Elogsigma} -- in fact the factor $2$ can be replaced by a constant arbitrarily close to $1$.

It remains to get a lower bound on the first term. For this, it is enough to get a lower bound on $\bEs\big[ \bPt(\sigma_k \in\tau | \sigma_n\in\tau) \big]$ for $k\le N/2$. For $b_n$ from \eqref{defanbn}, using the fact that $\bPt(x\in\tau)$ is regularly varying (see \eqref{Doney}) we get that there exists a constant $c_1$ such that for $k\le N/2$ large enough
\begin{align*}
\bEs\big[ \bPt(\sigma_k \in\tau | &\sigma_N\in\tau) \big] = \bEs\left[\frac{\bPt(\sigma_k \in \tau) \bPt(\sigma_N-\sigma_k \in\tau)}{ \bPt(\sigma_N \in\tau)} \right]\\
& \ge \bPs\Big( \sigma_k \le b_k, \sigma_{N}- \sigma_k \ge b_{N}, \sigma_N \le 4 b_{N} \Big) \times c_1 \frac{\bPt( b_k \in\tau) \bPt( b_N \in\tau) }{\bPt(4 b_N \in\tau)} \, .
\end{align*}
Since $\sigma_n/b_n$ converges in distribution, the first probability is bounded below by a constant. Hence $\bEs\big[ \bPt(\sigma_k \in\tau | \sigma_n\in\tau) \big]$ exceeds a constant times $\bPt(b_k \in\tau)$ which is regularly varying with exponent $- (1-\ga\wedge 1) / \tga\wedge 1 = -(1+\ga^*)$, thanks to \eqref{defanbn}-\eqref{defanbn2} and \eqref{Doney}. There is therefore a slowly varying function $\gp_1^*(\cdot)$ such that
\[\bEs\big[ \bPt(\sigma_k \in\tau | \sigma_N\in\tau) \big] \ge \gp_1^*(k) k^{-(1+\ga^*)}  \quad \text{for all } k\le N/2\, .\]
 We finally obtain that for $\ga^* <0$, for some $c_2$,
\[ 
\bEs\Big[ \bEt\big[ |\tau\cap\sigma|_{\sigma_N} |\sigma_N\in\tau \big]  \Big]
\ge \sum_{k =1}^{N/2}\bEs\big[ \bPt(\sigma_k \in\tau | \sigma_n\in\tau) \big] \ge c_2\gp_1^*(N) N^{-\ga^*} \, ,
\]
so by \eqref{eq:finitevol}, 
\begin{equation}
\label{eq:finitevol2}
\bEs \log Z_{N,\gb}^{\sigma} \ge  c_2\gb \gp_1^*(N) N^{-\ga^*} - 2 \tga^{-1} \log N.
\end{equation}
Fix $\gep\in(0,1)$, and for $\gb>0$ set $N_{\gb}^{\gep} := \gb^{-(1+\gep)/|\ga^*| }$, so that $N_{\gb}^{\gep}\to +\infty$ as $\gb\downarrow 0$. In view of \eqref{eq:finitevol2} there exists some $\gb_0 =\gb_0(\gep)$ such that for any $0<\gb\le \gb_0$, we have  $\bEs \log Z_{N_{\gb}^{\gep},\gb}^{\sigma} \ge 1$.
This shows that for every $\gep>0$ and every $\gb \le \gb_0=\gb_0(\gep)$ we have  
\[ \mF(\gb) \ge \frac{1}{N_{\gb}^{\gep}} = \gb^{(1+\gep)/|\ga^*| } >0 \, . \]
Since $\gep$ is arbitrary, this proves that when $\ga^*<0$ we have $\gb_c=0$ and 
\[\limsup_{\gb\downarrow 0} \frac{ \log \mF(\gb)}{\log \gb}  \le \frac{1}{|\ga^*|} \, .\]

\medskip
On the other hand, the upper bound $\mF(\gb)\le \mF^{\a}(\gb)$ together with Proposition \ref{prop:annealed} gives, in the case $\ga^*<0$ (so that in particular $\gb_c^{\a}=0$)
\[\liminf_{\gb\downarrow 0} \frac{ \log \mF(\gb)}{\log \gb}  \ge \liminf_{\gb\downarrow 0} \frac{ \log \mF^{\a}(\gb)}{\log \gb} =  \frac{1}{|\ga^*|} \, ,\] 
and this ends the proof of \eqref{eq:main3} when $\ga^*<0$.

\smallskip
For the case $\ga^*=0$, it is more delicate since $\bEs\big[ \bEt\big[ |\tau\cap\sigma|_{\sigma_N} |\sigma_N\in\tau \big]  \big]$ grows only as a slowly varying function, and may not be enough to compensate the logarithmic term in~\eqref{eq:finitevol}. We develop another approach in the following Sections, to prove that $\gb_c = \gb_c^{\a}$ also when $\ga^*=0$.
Once we have proven that $\gb_c=\gb_c^{\a}$, we get directly from $\mF(\gb)\le \mF^{\a}(\gb)$ and Proposition \ref{prop:annealed}  that
\[\liminf_{\gb\downarrow \gb_c} \frac{ \log \mF(\gb)}{\log \gb}  \ge \liminf_{\gb\downarrow \gb_c^{\a}} \frac{ \log \mF^{\a}(\gb)}{\log \gb} =+\infty \,, \]
thus proving \eqref{eq:main3} also when $\ga^* =0$.

\subsection{Preliminaries: definition of \emph{accepting} stretches}
\label{sec:preliminaries}
We first define the accepting property for $\sigma$.
Let $I_\tau$ and $I$ be the rate functions, for contact fractions $\delta$, for the renewals $\tau$ and $\tau\cap\sigma$, respectively: 
\begin{equation}\label{ratefunc}
  I_\tau(\gd):= -\lim_{n\to\infty} \frac{1}{n} \log \bP_{\tau} (|\tau|_n \geq \delta n), \quad
  I(\gd):= -\lim_{n\to\infty} \frac{1}{n} \log \bP_{\sigma,\tau} (|\tau\cap\sigma|_n \geq \delta n).
 \end{equation}
It is straightforward that $I_\tau(\delta)=\delta J_\tau(\delta^{-1})$, where $J_\tau$ is the usual rate function for i.i.d.~sums distributed as the gap between renewals in $\tau$, and similarly for $I(\delta)$. Let $\delta_{\max} = \sup\{\delta\in [0,1]:I(\delta)<\infty\}$ and $0<\delta<\delta_{\max}$. Then
\begin{equation}
\label{eq:ratefun}
\bP_{\sigma,\tau}(|\tau\cap\sigma|_n \geq \delta n) = e^{-I(\delta)\, n\, (1+o(1))}.
\end{equation}

For technical reasons, we need the following lemma: it says that the rate function $I(\gd)$ is unchanged when also imposing $\sigma_n\in\tau$, and that $\sigma_n$ grows only polynomially.

\begin{lemma}
\label{lem:LDP}
Suppose $\tga>0$. For all $\gd<\gd_{\max}$, for $b_n$ from \eqref{defanbn},
\begin{equation}\label{zerorate}
\lim_{n\to\infty} \frac{1}{n} \log \bP_{\sigma,\tau}\Big( \sigma_n \in \tau, \sigma_n \leq 2b_n\ \big|\ |\tau\cap\sigma|_n \geq \delta n \Big) = 0,
\end{equation}
so that
\begin{equation}
\bP_{\sigma,\tau}\Big( |\tau\cap\sigma|_n \geq \delta n \, , \, \sigma_n \in \tau, \sigma_n \leq 2b_n\Big) = e^{-I(\delta)\, n\, (1+o(1))}.
\end{equation}
\end{lemma}

\begin{proof}
Let $\gep>0$ and let  $m$ be the least integer in $[\delta n,\infty)$. Since $\gd<\gd_{\max}$, there exists $\eta>0$ with 
\begin{align}\label{smallrate}
\lim_{n\to\infty} &\frac{1}{n} \log \bP_{\sigma,\tau}\Big( (\tau\cap\sigma)_m \leq (1-\eta) n\ \big|\ |\tau\cap\sigma|_n \geq \delta n \Big) \notag\\
&\geq \lim_{n\to\infty} \frac{1}{n} \log \bP_{\sigma,\tau}\Big( |\tau\cap\sigma|_{(1-\eta)n} \geq \gd n\ \big|\ |\tau\cap\sigma|_n \geq \delta n \Big) > -\gep.
\end{align}
Since $b_n > n$ and $\sigma_n/b_n$ converges either to a stable limit or to 1, there exist $n_0\geq 1$ and $\theta>0$ as follows. For $n$ large and $\gd n \leq j\leq (1-\eta)n$, $m\leq k\leq j$,
\begin{align}
\bP_{\sigma,\tau}&\Big( \sigma_n \in \tau, \sigma_n \leq 2b_n\ \big|\ (\tau\cap\sigma)_m=j ; \sigma_k = j\Big) \notag\\
&\geq  \bPs\big(\sigma_{n-k} \leq  b_{n-k}\big) \min_{ k \leq n+b_n} \bPt(k\in\tau) \notag\\
&\geq \theta \min_{ k \leq 2b_n} \bPt(k\in\tau),
\end{align}
which is bounded below by an inverse power of $n$, uniformly in such $j$.  Therefore 
\[
  \lim_{n\to\infty} \frac{1}{n} \log \bP_{\sigma,\tau}\Big( \sigma_n \in \tau, \sigma_n \leq 2b_n\ \big|\ (\tau\cap\sigma)_m \leq (1-\eta) n \Big) = 0,
\]
which with \eqref{smallrate} proves \eqref{zerorate}, since $\gep$ is arbitrary.
\end{proof}

\medskip 
Since we take $\gb>\gb_c^{\a}=\gb_c^{\hom}$, the associated annealed and homogeneous double pinning free energies are positive:
\begin{equation}
\label{eq:Fannsup}
\mF^{\a}(\gb) \geq \mF^{\hom}(\beta)  = \sup_{\delta>0}(\beta\delta - I(\delta)) > 0.
\end{equation}
Let $\td<\delta_{\max}$ satisfy
\begin{equation} \label{eq:dtilde}
  \beta\td - I(\td) > \frac 12 \mF^{\hom}(\beta), 
\end{equation}
and let
\[
  f_n(\sigma) :=  \bP_{\tau}\big(|\tau\cap\sigma|_n \geq \td n \, , \, \sigma_n \in \tau \big) \,\ind_{\{\sigma_n \leq 2 b_n \}}.
\]
Note that $f_n$ is actually a function of the finite renewal sequence
$ B := \{0,\sigma_1,\dots,\sigma_n\},$
so we may write it as $f_n(B)$.
We decompose the probability that appears in $f_n$ according to what portion of the cost $I(\td)n$ is borne by $\sigma$ versus by $\tau$: we fix $\gep>0$ and write
\begin{multline}\label{splitks}
 \bP_{\sigma,\tau}\big(|\tau\cap\sigma|_n \geq \td n, \sigma_n \in \tau, \sigma_n \leq 2 b_n \big) = \bE_{\sigma}[f_n(\sigma)] \\
    \leq \sum_{0 \leq k \leq 1/\gep} \bP_{\sigma}\left(f_n(\sigma) \in \big( e^{-(k+1)\gep I(\td)n},e^{-k\gep I(\td)n} \big] \right) e^{-k\gep I(\td)n}
      + e^{-(1+\gep)I(\td)n} \, .
\end{multline}
Then by Lemma \ref{lem:LDP}, there exists some $n_\gep$ large enough so that for $n\geq n_{\gep}$, 
\[
e^{-(1+\gep)I(\td)n} \leq \tfrac12  \bP_{\sigma,\tau}\big(|\tau\cap\sigma|_n \geq \td n, \sigma_n \in \tau, \sigma_n \leq 2 b_n \big)\, .
\]
Hence choosing $k_0=k_0(n)\in[0,1/\gep]$  to be the index of the largest term in the sum in \eqref{splitks}, we have
\begin{multline}
\label{prod}
  \bP_{\sigma}\left(f_n(\sigma) \in \big( e^{-(k_0+1)\gep I(\td)n},e^{-k_0\gep I(\td)n} \big] \right) e^{-k_0\gep I(\td)n}\\
   \geq \frac12 \, \frac{1}{1+1/\gep}\  \bPst\big(|\tau\cap\sigma|_n \geq \td n,\sigma_n\in\tau, \sigma_n \leq 2 b_n \big).
\end{multline}

We let $a:=a(n):=k_0\gep$, so $a\in [0,1]$ represents, roughly speaking, the proportion of the cost $I(\td)$ borne by $\tau$ in the most likely cost split, i.e.~in the $k=k_0$ term. Lemma \ref{lem:LDP} says that if $n\geq n_{\gep}$ for some $n_\gep$ large enough, one has that
\[
\bPst\big(|\tau\cap\sigma|_n \geq \td n,\sigma_n\in\tau, \sigma_n \leq 2 b_n \big) \geq e^{-(1+\gep/2) I(\td) n}\, .
\]
Then, \eqref{prod} gives that
\begin{equation}
\label{eq:probabaccepting}
  \bP_{\sigma}\left(f_n(\sigma) \in \big( e^{-(a+\gep )I(\td)n},e^{-a I(\td)n} \big]  \right) 
    \geq \half\, \frac{\gep}{1+\gep}\ e^{a I(\td)n} e^{-(1+\gep/2) I(\td) n} \geq e^{-(1-a+\gep) I(\td)n},
\end{equation}
where the last inequality holds provided that $n$ is large enough ($n\geq n_{\gep}$).

This leads us to define, for $n\geq n_\gep$ the event $\mA_n$ for $\sigma$ (or more precisely for $\{0,\sigma_1,\ldots,\sigma_n\}$) of being \emph{accepting}, by
\begin{equation}
\label{def:accepting}
\sigma\in\mA_n \qquad \text{ if } \qquad  f _n(\sigma)\in ( e^{-(a+\gep )I(\td)n},e^{-a I(\td)n} ] \, .
\end{equation}
Then \eqref{eq:probabaccepting} gives the lower bound
\[
\bPs(\sigma \in\mA_n) \geq e^{-(1-a+\gep) I(\td)n} \quad\text{if}\quad n\geq n_\gep\, .
\]
Moreover we have for $n\geq n_{\gep}$
\begin{align}
\label{eq:onaccepting}
\bE_{\tau}\big[ e^{\gb |\tau\cap \sigma|_{\sigma_{n} } } \ind_{\{\sigma_n \in \tau\}} \big] \ind_{\{ \sigma \in\mA_n \}} &\geq e^{\gb \td n } \ \bP_{\tau}\Big( |\tau\cap\sigma|_{\sigma_n}\geq \td n, \sigma_n \in\tau,  \sigma_n\leq 2 b_n\Big) \ \ind_{\{ \sigma \in\mA_n \}} \notag\\
&\geq  e^{( \gb \td n-(a+\gep) I(\td) )n}\ \ind_{\{ \sigma \in\mA_n \}} .
\end{align}

\subsection{Localization strategy}
\label{sec:localization}

Let us fix $\gb>\gb_c^{\a}=\gb_c^{\hom}$, $\td$ as in \eqref{eq:dtilde}
and $\ell$ large (in particular $\ell\geq n_\gep$, so that \eqref{eq:onaccepting} holds for $n=\ell$.)
We divide $\sigma$ into segments of $\ell$ jumps which we denote $Q_1, Q_2,\ldots$, that is, $Q_i:=\{\sigma_{(i-1)\ell +1} - \sigma_{(i-1)\ell} ,\ldots, \sigma_{i \ell} - \sigma_{(i-1)\ell}  \}$.
We say that $Q_i$ is an \emph{accepting segment} if $Q_i\in \mA_\ell$. We now write $\mA$ for $\mA_\ell$, and $a$ for $a(\ell)$.
Let 
\begin{equation} \label{pAdef}
  p_\mA:=\bP_{\sigma}(Q_i \in \mA) \geq e^{-(1-a+\gep) I(\td)\ell},
\end{equation}
where the inequality is from \eqref{eq:probabaccepting}.
Informally, the strategy for $\tau$ is then to visit all accepting segments, with no other restrictions on other regions. On each such segment one has an energetic gain of at least  $e^{( \gb \td n-(a+\gep) I(\td) )n}$, thanks to \eqref{eq:onaccepting}.

We define $M_0:=0$, and iteratively $M_{i}:=\inf\{j>M_{i-1} \, ;\, Q_j\in \mA\}$; then $M_i-M_{i-1}$ are independent geometric random variables of parameter $p_{\mA}$.
Imposing visits to all the accepting segments $Q_{M_{i}}$ (that is imposing $\sigma_{M_i\ell},\sigma_{(M_i+1)\ell}\in\tau$), one has
\begin{equation} \label{allaccept0}
  Z_{M_k\ell,\gb}^{\sigma} \geq \bE_\tau\left( \prod_{i=1}^k \exp\bigg( \beta \big|\sigma\cap\tau\cap (\sigma_{(M_i-1)\ell},\sigma_{M_i\ell}]\big| \bigg)
    \ind_{\{\sigma_{M_i\ell} \in \tau\} } \ind_{\{ \sigma_{(M_i-1)\ell} \in\tau \} }\right) \, .
\end{equation}
So, using \eqref{eq:onaccepting}, and with the convention that $\bPt(0\in\tau)=1$, we have
\begin{equation}
\label{allaccept}
\log Z_{M_k\ell,\gb}^{\sigma} \geq \sum_{i=1}^{k} \bigg(\log \bP_{\tau}\big( \sigma_{(M_{i}-1)\ell} - \sigma_{M_{i-1} \ell}\in \tau \big) + \gb \td \ell-(a+\gep) I(\td) \ell \bigg).
\end{equation}
Letting $k$ go to infinity, and using the Law of Large Numbers twice, we get
\begin{align}
\label{eq:boundF}
\mF(\gb) &\geq \liminf_{k\to\infty} \frac{1}{M_k \ell} \log Z_{M_k\ell,\gb}^{\sigma} \notag\\
&\geq  \frac{1}{\bE_{\sigma} [M_1]} \left( \frac1\ell \bE_{\sigma}\log \bP_{\tau}\big( \sigma_{(M_{1}-1) \ell} \in \tau \big) 
  + \gb \td -(a+\gep) I(\td) \right),
\end{align}
with $\bE_{\sigma}[M_1]=1/p_\mA$.

\medskip
We are left with estimating $\frac{1}{\ell} \bE_{\sigma}\log \bP_{\tau}\big( \sigma_{(M_{1}-1) \ell} \in \tau \big)$. 
For any $\ga\geq 0$ and non-random times $n$, \eqref{Doney} ensures that, for any $\eta>0$, one can find some $n_{\eta}$ such that, for all $n\geq n_{\eta}$,
\begin{equation}
\log \bPt(n\in\tau) \geq  - (1-\ga\wedge 1 +\eta) \log n\, .
\end{equation}

Therefore if $\ell_{\eta}$ is chosen large enough, and $\ell\geq \ell_{\eta}$, one has
\begin{equation}
\label{eq:Doney1}
\bE_{\sigma}\log \bP_{\tau}\big( \sigma_{(M_{1}-1) \ell} \in \tau \big)\geq -(1-\ga\wedge 1+\eta) \bE_{\sigma}[\ind_{\{M_1>1\}}\log \sigma_{(M_1-1)\ell}].
\end{equation}
The following is proved below.

\begin{lemma} \label{logsigma}
Let $\eta>0$ and suppose $\tilde\ga>0$ in \eqref{eq:alphas}.  Provided $\ell$ is sufficiently large we have
\[
  \bE_{\sigma}[\ind_{\{M_1>1\}}\log \sigma_{(M_1-1)\ell}] \leq \frac{1+\eta}{\tilde\ga\wedge 1} \left( \log \ell +\log\frac{1}{p_\mA} \right).
\]
\end{lemma}

Combining \eqref{eq:Doney1} with Lemma \ref{logsigma}, one gets for $\ell\geq \ell_\eta$, and provided that $\ell_{\eta}$ is large so that $\frac{1}{\ell}\log \ell \leq \eta $,
\begin{equation}
\label{eq:Doney3}
\frac{1}{\ell} \bE_{\sigma}\log \bP_{\tau}\big( \sigma_{(M_{1}-1) \ell} \in \tau \big)\geq -(1-\ga\wedge 1 +\eta)  \frac{1+\eta}{\tilde\ga\wedge 1} \Big( \frac{1}{\ell} \log\frac{1}{p_\mA} +\eta \Big).
\end{equation}
From  \eqref{eq:boundF} and \eqref{eq:Doney3}, and assuming that $\eta$ is sufficiently small (depending on $\gep$), we get
\begin{equation} \label{freelower}
  \mF(\gb) \geq p_\mA\left( \beta\tilde\gd - (a+\gep)I(\tilde\gd)  - \frac{1-\ga\wedge 1 }{\tilde\ga\wedge 1}\, (1+\gep) \, \frac1\ell \log\frac{1}{p_\mA}  -\gep\right).
\end{equation}
Now, the crucial point is that $\ga+\tga\geq 1$, so that $(1-\ga\wedge 1)/\tilde\ga\wedge 1 \leq 1$, and one has
\begin{align}
 \mF(\gb)   &\geq  p_\mA\left( \beta\tilde\gd - (a+\gep)I(\tilde\gd) - (1+\gep)  \frac1\ell \log\frac{1}{p_\mA} -\gep \right) \notag \\
  &\geq p_\mA\left( \beta\tilde\gd - (a+\gep)I(\tilde\gd) - (1+\gep)  (1-a +\gep) I(\td) -\gep \right)  \notag\\
  &\geq  p_\mA\left( \frac12 \mF^{\hom}(\gb) - (3-a+\gep)\gep I(\tilde\gd) -\gep \right)\, ,
\end{align}
where we used \eqref{pAdef} in the second inequality, and \eqref{eq:dtilde} in the third one. If we choose $\gep$ small enough, we therefore have that $\mF(\gb)\geq \frac14 p_\mA \mF^{\hom}(\gb)>0$, as soon as $\gb>\gb_c^\a=\gb_c^{\hom}$. This completes the proof of \eqref{eq:main2} Theorem \ref{thm:main}.\qed

\begin{proof}[Proof of Lemma \ref{logsigma}]
Let us first bound $\bE_\sigma \log \sigma_n$ for deterministic large $n$.  Consider first $\tga\in(0,1)$.  Let $\eta\in (0,1/4)$.
From \cite{Doney}, whenever $k/b_n \to\infty$, $\bPs(\sigma_n =k) \sim n \bPs(\sigma_1 =k)$. Hence using \eqref{eq:alphas}, uniformly in $t>1+\eta$, as $n\to\infty$,
\begin{equation}\label{tail}
\bPs(\sigma_n \geq n^{t/\tga}) =  (1+o(1)) n \bPs(\sigma_1 \geq n^{t/\tga} ) \leq n^{1-t +\eta}\, .
\end{equation}

Therefore there exists some $n_\eta$ such that if $n\geq n_\eta$,
\begin{align} \label{integral}
  \bE_\sigma \frac{\tilde\ga \log \sigma_n}{\log n} &\leq 1 + \eta + \int_{1+\eta}^\infty \bPs\left(\log \sigma_n \geq \frac{t}{\tilde\ga}\log n\right)\ dt\notag\\
  &\leq 1 + \eta + n^{1+\eta} \int_{1+\eta}^\infty e^{-t\log n}\ \dd t \ \leq\ 1 + \eta + \frac{1}{\log n},
\end{align}
so that, provided that $n$ is large enough,
\begin{equation} \label{sigman}
  \bE_\sigma \log \sigma_n \leq \frac{1+2\eta}{\tilde\ga} \log n.
\end{equation}

For $\tga\geq 1$, we can multiply the probabilities $\bPs(\sigma_1=n)$ by the increasing function $cn^\gamma$ for appropriate $c,\gamma>0$, and thereby obtain a distribution with tail exponent in $(1-\eta,1)$ which stochastically dominates the distribution of $\sigma_1$.  This shows that
\[
  \bE_\sigma \log \sigma_n \leq \frac{1+2\eta}{1-\eta} \log n.
\]
Thus for all $\tga>0$ and $\eta\in (0,1/4)$, for $n$ large,
\begin{equation} \label{Elogsigma}
  \bE_\sigma \log \sigma_n \leq \frac{1+4\eta}{\tilde\ga\wedge 1} \log n.
\end{equation}

\medskip
Now let $n_\mA=K/p_\mA$, with $K$ (large) to be specified.
First, provided that $\ell$ is large enough,
\begin{equation}
\label{M2}
\bE_{\sigma} \big[ \ind_{\{M_1>1\}} \ind_{\{M_1 \leq 2 n_\mA\}}\log \sigma_{(M_1-1)\ell} \big] \leq \bE_{\sigma} \big[  \log \sigma_{2 n_\mA\ell} \big]\leq \frac{1+4\eta}{\tga \wedge 1} \log (2 n_\mA \ell)\, ,
\end{equation}
where we used \eqref{Elogsigma} in the last inequality.
Second, fixing $m\geq 3$, we have for any $j\in \big((m-1)n_{\mA}, m n_\mA\big]$
\begin{align} \label{condl}
  \bE_{\sigma}[\ind_{\{M_1>1\}}& \log \sigma_{(M_1-1)\ell} \mid M_1=j ]
  = \bE_\sigma\left[ \log \sigma_{(j-1)\ell} \mid Q_1 \notin \mA,\dots,Q_{j-1} \notin \mA \right]\notag \\
  &\leq \bE_\sigma\left[ \log \sigma_{m n_\mA\ell} \mid Q_1 \notin \mA,\dots,Q_{m n_\mA} \notin \mA \right]\notag \\
  &\leq \sum_{r=1}^m \bE_\sigma\bigg[ \log\big(\sigma_{rn_\mA\ell} - \sigma_{(r-1)n_\mA\ell} \big)\ \big|\ Q_i \notin \mA\ \forall i \in \big( (r-1) n_\mA , r n_\mA\big]  \bigg] \notag\\
  &= m\, \bE_\sigma( \log \sigma_{n_\mA\ell} \mid Q_i \notin \mA\ \forall i \leq n_\mA) \notag\\
  &\leq m\,  \frac{ \bE_\sigma[ \log \sigma_{n_\mA\ell}] }{ (1-p_\mA)^{n_\mA} } \, .
\end{align}
Hence, using \eqref{Elogsigma}, one has that if $\ell$ is large enough,
\begin{align} \label{sumj}
  \bE_{\sigma}\Big[ \ind_{\{M_1 \in  ((m-1)n_{\mA}, m n_\mA ] \}}\log \sigma_{(M_1-1)\ell} \Big]
    &\leq \bPs\big(M_1>  (m-1)n_{\mA} \big) \, m\, \frac{1+4\eta}{\tilde\ga \wedge 1} \frac{ \log(n_\mA\ell) }{ (1-p_\mA)^{n_\mA} } \notag\\
    &\leq m (1-p_\mA)^{(m-2) n_\mA} \frac{1+4\eta}{\tilde\ga \wedge 1} \log(n_\mA\ell) \notag\\
    &\leq 2 m e^{- (m-2) K} \frac{1}{\tilde\ga \wedge 1} \log(n_\mA\ell) \, .
\end{align}
Combining \eqref{M2}, and \eqref{sumj}, we obtain, provided that $K$ is large enough (depending on $\eta$)
\begin{align} \label{combine}
  \bE_{\sigma}\big[ \ind_{\{M_1>1\}}\log \sigma_{(M_1-1)\ell} \big]&\leq 
   \frac{1+4\eta}{\tga \wedge 1} \log (2 n_\mA \ell) + \sum_{m=3}^\infty 2 m e^{- (m-2) K} \frac{1}{\tilde\ga \wedge 1} \log(n_\mA\ell)  \notag\\
 & \leq  \frac{1+5\eta}{\tilde\ga \wedge 1} \log(2n_\mA\ell) \notag\\
 &\leq \frac{1+6\eta}{\tilde\ga \wedge 1} \left( \log \ell + \log \frac{1}{p_\mA} \right),
\end{align}
where the last inequality is valid provided that $\ell$ is large.
\end{proof}

\section{Coarse-graining and change of measure: proof of the second part of Theorem \ref{thm:main}}
\label{sec:proof2}

In this section and the following ones, we deal with the case $\ga+\tga < 1$. In particular, one has $0\leq \ga<1,\ 0<\tga<1$ and $\ga^*>0$. Moreover, the renewal $\tau\cap \sigma$ is transient, so $\bEst[|\tau\cap\sigma|]<+\infty$, and the annealed critical point is $\gb_c^\a = -\log\big( 1-\bE_{\sigma,\tau}[|\tau\cap\sigma|]^{-1}\big) >0$.

\subsection{Alternative representation of the partition function}
\label{sec:chgpartitionfunction}

We use a standard alternative representation of the partition function, used for example in the Random Walk Pinning Model, see \cite{BGdH11} and \cite{BT09,BS08,BS09}, and in other various context: it is the so-called \textit{polynomial chaos expansion}, which is the cornerstone of \cite{CSZ14}.

We write $e^{\gb}=1+z$ (called Mayer expansion), and expand $(1+z)^{\sum_{n=1}^{N} \ind_{\{\sigma_n\in\tau\}}}$: we get
\begin{equation}
\label{expand1}
Z_{N,\gb}^{\sigma} = \sum_{m=1}^{N} z^m \hspace{-.3cm}\sum_{1\leq i_1<\cdots<i_m=N} \hspace{-.3cm}  \bE_{\tau}\bigg[\prod_{k=1}^m \ind_{\{\sigma_{i_k}\in\tau\}} \bigg]
 = \sum_{m=1}^{N} z^m  \hspace{-.3cm} \sum_{1\leq i_1<\cdots<i_m=N} \prod_{k=1}^m \bE_{\tau}[\ind_{\{\sigma_{i_k}-\sigma_{i_{k-1}} \in\tau \}}],
\end{equation}
where $i_0:=0$. Now, let us  define
\begin{equation}
\label{eq:defK*}
K^*(n):= \frac{1}{\bE_{\tau,\sigma}[|\tau\cap\sigma|]}\bE_{\tau,\sigma}[\ind_{\{\sigma_{n} \in\tau \}}],
\end{equation}
so that $\sum_{n\in\NN} K^*(n) =1$, and we denote $\nu$ a  renewal process with law $\bP_\nu$ and inter-arrival distribution $\bP_{\nu}(\nu_1=n)=K^*(n)$. In particular, $\nu$ is recurrent.
Lemma \ref{lem:Kstar} (proven in Appendix \ref{app:K*}) gives the asymptotic behavior of $K^*(n)$:
\begin{equation}
K^*(n) \stackrel{n\to\infty}{\sim} \frac{\varphi^*(n) \, n^{-(1+\alpha^*)}}{\bE_{\tau,\sigma}[|\tau\cap\sigma|]}   \qquad \text{ with } \alpha^*= \frac{1-\ga-\tilde\ga}{\tilde\ga}>0\, .
\end{equation}

We want to interpret the expansion \eqref{expand1} as the partition function of some pinning model, with underlying renewal $\nu$.
We define $z_c^{\a}:=e^{\gb_c^a} -1$, so that $z_c^{\a}=\bE_{\tau,\sigma}[|\tau\cap\sigma|]^{-1}$. Then we write $z=z_c^{\a} e^{u}$: thanks to \eqref{eq:defK*}, we have
\begin{align}\label{zcheck}
Z_{N,\gb}^{\sigma}  =: \check Z_{N,u}^{\sigma} &= \sum_{m=1}^{N} e^{um} \sum_{i_1<\cdots<i_m=N} \prod_{k=1}^m K^*(i_k-i_{k-1})   
  w(\sigma,i_{k-1},i_k) \notag\\
&=\bE_{\nu}\Big[ e^{u |\nu\cap(0,N]|} W(\nu,\sigma,[0,N]) \ind_{\{N\in\nu\}} \Big],
\end{align}
with
\begin{equation}
w(\sigma,a,b):= \frac{\bE_{\tau}[\ind_{\{\sigma_{b}-\sigma_{a} \in\tau \}}]}{\bE_{\tau,\sigma}[\ind_{\{\sigma_{b}-\sigma_{a} \in\tau \}}]} , \quad \quad \bE_{\sigma}[w(\sigma,a,b)]=1,
\end{equation}
and
\begin{equation}
\label{def:W}
W(\nu,\sigma,[a,b]):= \frac{\bE_{\tau}[\prod_{k: \nu_k\in [a,b]} \ind_{\{\sigma_{\nu_k} \in \tau \}}]}{\bE_{\tau,\sigma}[\prod_{k: \nu_k\in [a,b]} \ind_{\{\sigma_{\nu_k} \in \tau \}}]} , \qquad \bE_{\sigma}[W(\nu,\sigma,[a,b])]=1.
\end{equation}

Thus $\check Z_{N,u}^{\sigma}$ corresponds to a pinning model: an excursion $(\nu_i,\nu_{i+1}]$ of $\nu$ is weighted by $e^u w(\sigma,\nu_i,\nu_{i+1})$.
Note that, since $\bE_{\sigma}[w(\sigma,a,b)]=1$, the annealed partition function,
\[\bEs \check Z_{N,u}^{\sigma} = \bE_{\nu}\big[ e^{u |\nu\cap(0,N]|} \ind_{\{N\in\nu\}}\big] ,\]
is the partition function of a homogeneous pinning model, with an underlying recurrent renewal $\nu$: the annealed critical point is therefore $u_c^{\a}=0$.


\subsection{Fractional moment method}
\label{sec:fracmoment}

We are left with studying the new representation of the partition function, $\check Z_{N,u}^{\sigma} $, and in particular, we want to show that its (quenched) critical point is positive.
For that purpose, it is enough to show that there exists some $u>0$ and some $\zeta>0$ such that
\begin{equation}
\label{fracmomentfinite}
\liminf_{n\in\NN} \bE_{\sigma} \bigg[ (\check Z_{N,u}^{\sigma})^{\zeta}\bigg]<+\infty.
\end{equation}
Indeed, using Jensen's inequality, one has that
\begin{equation}
\mF(\gb) = \lim_{N\to\infty} \frac1N \bE_{\sigma} \log Z_{N,\gb}^{\sigma} =   \lim_{N\to\infty} \frac{1}{\zeta N} \bE_{\sigma} \log (\check Z_{N,u}^{\sigma})^{\zeta}  \leq
\liminf_{N\to\infty} \frac{1}{\zeta N} \log\bE_{\sigma} \bigg[ (\check Z_{N,u}^{\sigma})^{\zeta} \bigg],
\end{equation}
so that \eqref{fracmomentfinite} implies $\mF(\gb)=0$.

\subsection{The coarse-graining procedure}
\label{sec:coarsegraining}

Let us fix the coarse-graining length $L=1/u$, and decompose the system into blocks of size $L$: $\mB_i:= \{(i-1)L+1,\ldots,iL\}, i\geq 1$.
Considering a system of length $nL$, for $n\in\NN$, we have
\begin{equation}
\check Z_{nL,u}^{\sigma} = 
\sum_{ \mI=\{1\leq i_1<\dots<i_m=n\} }
\sumtwo{ d_{i_1}\leq f_{i_1} }{ d_{i_1},f_{i_1}\in \mB_{i_1} } \cdots \sumtwo{d_{i_m}\leq f_{i_m}=nL}{d_{i_m}\in\mB_{i_m}}  \prod_{k=1}^m \check Z_{[d_{i_k},f_{i_k}],u}^{\sigma} K^*(d_{i_k}-f_{i_{k-1}}) w(\sigma,f_{i_{k-1}},d_{i_k}),
\end{equation}
where $f_0=i_0=0$ and
\begin{align*}
\check Z_{[a,b],u}^{\sigma} & := \bE_{\nu}\left[  e^{u |\nu\cap(a,b]|} W(\nu,\sigma,[a,b]) \ind_{\{b\in\nu\}}] | a\in\nu\right]\\
&\leq e \bE_{\nu}\left[  W(\nu,\sigma,[a,b]) \ind_{\{b\in\nu\}}] | a\in\nu\right] = e\check Z_{[a,b],0}^{\sigma}.
\end{align*}

Then, we denote by $\mZ_{\mI}$ the partition function with $u=0$, and where $\nu$ is restricted to visit only blocks $\mB_i$ for $i\in\mI$: if $\mI=\{1\leq i_1<\dots<i_m=n-1\}$,
\begin{equation}
\label{eq:ZI}
\mZ_{\mI}:= \sumtwo{d_{i_1}\leq f_{i_1}}{d_{i_1},f_{i_1}\in \mB_{i_1}} \cdots \sumtwo{d_{i_m}\leq f_{i_m}=nL}{d_{i_m}\in\mB_{i_m}}  \prod_{k=1}^m \check Z_{[d_{i_k},f_{i_k}],0}^{\sigma} K^*(d_{i_k}-f_{i_{k-1}}) w(\sigma,f_{i_{k-1}},d_{i_k}).
\end{equation}
For $\zeta<1$ we have
\begin{equation}
\label{eq:checkZ}
(\check Z_{nL,u}^{\sigma})^{\zeta} \leq \sum_{\mI\subset \{1,\ldots,n\}} e^{\zeta | \mI |}\, (\mZ_{\mI})^{\zeta},
\end{equation}
and we are left with estimating $\bE_{\sigma}[(\mZ_{\mI})^{\zeta}]$.

We choose $\zeta$ given by $(1+\ga^*/2)\zeta = 1+\ga^*/4$. We will show in the next two sections that for every $\gd>0$, there exists $L_0$ such that for any $L\geq L_0$
\begin{equation}
\label{coarsepartition}
\bE_{\sigma}[(\mZ_{\mI})^{\zeta}] \leq  \gd^{|\mI|}\prod_{k=1}^{|\mI|} (i_{k}-i_{k-1})^{-(1+\alpha^*/4)} \, .
\end{equation}
With \eqref{eq:checkZ}, this shows that
\begin{equation}
\bE_{\sigma}\left[ (\check Z_{nL,u}^{\sigma})^{\zeta} \right]\leq  \sum_{\mI\subset \{1,\ldots,n\},n\in\mI} \prod_{k=1}^{|\mI|} \frac{e^{\zeta} \gd}{ (i_{k}-i_{k-1})^{1+\alpha^*/4} }.
\end{equation}
We choose $\gd$ so that $\hat K(n) = e^{\zeta}\gd n^{-(1+\ga^*/4)}$ sums to 1, making it the inter-arrival probability of a recurrent renewal process $\hat \tau$. We then have $\bE_{\sigma}\left[ (\check Z_{nL,u}^{\sigma})^{\zeta} \right]\leq \bP_{\hat\tau}(n\in\hat \tau)\leq 1$, which yields \eqref{fracmomentfinite} and concludes the proof of the second part of Theorem~\ref{thm:main}.\qed

\subsection{Change of measure argument}
\label{sec:chgmeasure}

To estimate $\mZ_{\mI}$ and prove \eqref{coarsepartition}, we use a change of measure, but only on the blocks $\mB_i$, for $i\in\mI$.
The idea is to choose an event $J_L$ depending on $\{0,\sigma_1,\ldots,\sigma_L\}$ which has a small probability under $\bP_{\sigma}$, but large probability under the modified measure, see Lemma \ref{lem:eventJ} or Lemma \ref{lem:eventI}.

With the event $J_L$ to be specified, we define, for some (small) $\eta>0$,
\begin{equation}
\label{defg}
g(\sigma):= \eta \ind_{J_L} + \ind_{J_L^c}\, , \qquad
g_{\mI} := \prod_{i\in\mI} g(\sigma_{\mB_i}),
\end{equation}
where $\sigma_{\mB_i}=\{0,\sigma_{iL+1}-\sigma_{iL},\ldots,\sigma_{(i+1)L}-\sigma_{iL}\}$ is the translation of $\sigma$ by $\sigma_{iL}$.

Using H\"older's inequality, we have
\begin{equation}
\label{eq:holder}
\bE_{\sigma}[(\mZ_{\mI})^{\zeta}] = \bE_{\sigma}[g_{\mI}^{-\zeta} (g_{\mI} \mZ_{\mI})^{\zeta}] \leq \bE_{\sigma}\big[ (g_{\mI})^{-\zeta/(1-\zeta)} \big]^{1-\zeta} \bE_{\sigma} \left[ g_{\mI} \mZ_{\mI}\right]^{\zeta}. 
\end{equation}
The first term on the right in \eqref{eq:holder} is easily computed: assuming we choose $J_L$ with $\bPs(J_L)\leq \eta^{\zeta/(1-\zeta)}$, we have
\begin{align}
 \bE_{\sigma}\bigg[ (g_{\mI})^{-\zeta/(1-\zeta)} \bigg]^{1-\zeta} &=
 \bE_{\sigma}\bigg[ (g(\sigma))^{-\zeta/(1-\zeta)} \bigg]^{|I|(1-\zeta)} \notag \\
 &= \left( \eta^{-\zeta/(1-\zeta)} \bP_{\sigma}(J_L) + \bPs(J_L^c)\right)^{|I|(1-\zeta)} \notag\\
 &\leq 2^{|\mI |}\, .
\label{firstholder} 
\end{align}

We are left to estimate $\bE_{\sigma}\big[  g_{\mI} \mZ_{\mI} \big]$.
For this it is useful to control $\bEs[g(\sigma) \check Z_{[a,b]}^{\gs}]$ for $0\leq a<b\leq L$.
From the definition of $g$,
\begin{equation}
\bEs[g(\sigma) \check Z_{[a,b]}^{\gs}] \leq \eta \bEs[\check Z_{[a,b]}^{\gs}] + \bE_{\sigma} \big[ \ind_{J_L^c}\, \check Z_{[a,b]}^{\sigma}\big].
\end{equation}
Then, provided that we can show $\bE_{\sigma} \big[ \ind_{J_L^c}\, \check Z_{[a,b]}^{\sigma}\big] \leq \eta \bP_{\nu}(b-a\in\nu)$ whenever $b-a \geq \gep L$, we conclude
\begin{equation}
\label{secondholder}
\bEs[g(\sigma) \check Z_{[a,b]}^{\gs}] \leq (2\eta + \ind_{\{b-a<\gep L \}})\bP_\nu(b-a\in\nu).
\end{equation}
The following lemma fills in the missing pieces of the preceding, so we can conclude that \eqref{firstholder} and \eqref{secondholder} hold in the case $\ga^* >1/2$. The proof is in section \ref{sec:eventJ}.

\begin{lemma}
\label{lem:eventJ}
Suppose \eqref{eq:alphas} holds, with $\tga>0$ and $\ga+\tga<1$.
If $\ga^*>1/2$, then for any fixed $\eta>0$ and $\gep>0$, there exist events $J_L$ determined by $\{0,\sigma_1,\ldots,\sigma_L\}$, and $L_0>0$, such that, if $L\geq L_0$,
\begin{equation}
\label{eq:cond1}
\bPs(J_L)\leq \eta^{\zeta/(1-\zeta)}
\end{equation} 
and moreover, for $0\leq a<b\leq L$ with $b-a\geq \gep L$,
\begin{equation}
\label{eq:cond2}
\bE_{\sigma} \big[ \ind_{J_L^c}\, \check Z_{[a,b],0}^{\sigma}\big] \leq \eta\, \bP_{\nu}(b-a\in\nu)\, .
\end{equation}
Additionally, for all $L$, $\bE_{\sigma} \big[ \ind_{J_L}\, \check Z_{[a,b],0}^{\sigma}\big] \leq 
 \bP_{\nu}(b-a\in\nu)$.
\end{lemma}

In Section \ref{sec:example}, we show that without the restriction $\ga^*>1/2$, for arbitrary $\ga,\tga$ with $\ga+\tga <1$, there exist distributions of $\tau$ and $\sigma$ with these tail exponents, and suitable events $I_L$ with the same properties as $J_L$, see Lemma \ref{lem:eventI}. The same conclusion \eqref{secondholder} follows similarly.

\medskip

To bound $\bE_{\sigma}\big[  g_{\mI} \mZ_{\mI} \big]$ we need the following extension of Lemma \ref{lem:eventJ}, which concerns a single block, to cover all blocks.  The proof is in section \ref{sec:eventJ}.

\begin{lemma}
\label{keylemma}
Suppose \eqref{eq:alphas} holds, with $\tga>0$ and $\ga+\tga<1$. If $\ga^*>1/2$ and $\zeta<1$, $\eta>0$, $\gep>0$ are fixed, then there exist an event $J_L$ depending on $\{0,\sigma_1,\ldots,\sigma_L\}$, and $L_0>0$, such that if $L \geq L_0$ then $\bPs(J_L)\leq \eta^{\zeta/(1-\zeta)}$, and for every $\mI=\{1\leq i_1<\dots<i_m=n\}$ one has
\begin{equation}\label{gIZI}
\bE_{\sigma}\Big[ g_{\mI} \mZ_{\mI}\Big]  
\leq \sumtwo{d_{i_1}\leq f_{i_1}}{d_{i_1},f_{i_1}\in \mB_{i_1}} \cdots \sumtwo{d_{i_m}\leq f_{i_m}=nL}{d_{i_m}\in\mB_{i_m}}  \prod_{k=1}^m K^*(d_{i_k}-f_{i_{k-1}}) (2\eta + \ind_{\{f_{i_k}-d_{i_k}<\gep L\}}) \bP_{\nu}(f_{i_k} -d_{i_k}\in\nu).
\end{equation}
\end{lemma}

For fixed $\gd>0$, we claim that, by taking $\gep,\eta$ small enough in Lemma \ref{keylemma}, 
if $L$ is large enough, 
\begin{equation}
\label{eq:endchangemeasure}
\bE_{\sigma}\big[  g_{\mI} \mZ_{\mI}\big] 
\leq  \prod_{k=1}^{|\mI|} \frac{\gd}{(i_k-i_{k-1})^{1+\ga^*/2}} \, .
\end{equation}
This, together with \eqref{eq:holder} and \eqref{firstholder}, enables us to conclude that \eqref{coarsepartition} holds, with 
$\gd$ replaced by $\gd'=2 \gd^{2\zeta}$, since $\zeta(1+\ga^*/2) = 1+\ga^*/4$.

\smallskip

\begin{proof}[Proof of \eqref{eq:endchangemeasure}]
For every $\mI=\{1\leq i_1<i_2<\cdots <i_{m}=n\}$ we have by Lemma \ref{keylemma}
\begin{equation}
\bE_{\sigma}\Big[ g_{\mI} \mZ_{\mI}\Big]  
\leq \bE_{\nu} \left[ \ind_{E_{\mI}}   \prod_{k=1}^m  (2\eta + \ind_{\{F_{i_k}-D_{i_k}<\gep L\}}) \right] \, ,
\end{equation}
where we set
\[E_{\mI}:=\Big\{\nu:  \{i: \nu\cap \mB_i \neq\emptyset\} = \mI \Big\}, \quad D_{i} := \min \{\nu \cap \mB_{i}\},\ \ F_{i}:= \max \{ \nu \cap\mB_i\} \, .\]

We next show that given $\gd>0$, for $\gep,\eta$ sufficiently small, for any given $\mI$ and $1\leq k\leq m$
\begin{multline}
\label{blockbyblock}
\sup_{f_{i_{k-1}}\in\mB_{i_{k-1}}} \bE_{\nu} \left[   (2\eta + \ind_{\{F_{i_k}-D_{i_k}<\gep L\}}) \ind_{\{\nu \cap\mB_i =\emptyset \ \forall\, i_{k-1}<i<i_{k},\,\nu \cap \mB_{i_{k}} \neq\emptyset\}} \, \big|\,  F_{i_{k-1}} = f_{i_{k-1}}\right] \\
\leq \gd (i_{k}-i_{k-1})^{-(1+\ga^*/2)}\, ,
\end{multline}
where we used the convention that $i_0=0,f_0=0$ and $\mB_{i_0}=\{0\}$.
Then, we easily get  \eqref{eq:endchangemeasure} by iteration.

\smallskip

First, we see that for every $f_{i_{k-1}}\in \mB_{i_{k-1}}$, we have
\begin{equation} \label{firstbound}
\bE_{\nu} \left[ \ind_{\{\nu \cap\mB_i =\emptyset \ \forall\, i_{k-1}<i<i_{k},\,\nu \cap \mB_{i_{k}} \neq\emptyset\}} \, \big|\, F_{i_{k-1}} = f_{i_{k-1}} \right] = \bP_{\nu} \big( f_{i_{k-1}} + \nu_1 \in \mB_{i_k} \mid \nu_1\geq i_{k-1}L-f_{i_{k-1}} \big).
\end{equation}
If $i_{k}-i_{k-1}=1$, we bound this by $1$.
If $i_k-i_{k-1}\geq 2$, writing for simplicity $j_k:= i_k-i_{k-1} -1 \geq 1$, the right side of \eqref{firstbound} is at most
\begin{align}
\sup_{0\leq m\leq L}\bP_{\nu} & \big( \nu_1 \in (m+j_k L, m+(j_k+1)L] \,  | \,  \nu_1>m\big)
\leq \frac{L}{\bP_\nu(\nu_1\geq L)} \sup_{ j_k L \leq x \leq (j_k+2) L} \bP_\nu(\nu_1=x) \notag\\
 & \leq c_1L \gp^*(L)^{-1} L^{\ga^*} \gp^*(j_k L) (j_k L)^{-(1+\ga^*)} \leq c_2 j_k^{-(1+\ga^*/2)}\, ,
 \label{longjump}
\end{align}
where we used the existence of a constant $c_3$ such that, for all $a\geq 1$ and $L$ large, $\gp^*(aL) / \gp^*(L) \leq c_3 a^{\ga^*/2}$.
In the end, we have 
\begin{equation}
\label{bbb1}
2\eta \bE_{\nu} \left[ \ind_{\{\tau\cap\mB_i =\emptyset \ \forall\, i_{k-1}<i<i_{k},\,\tau\cap \mB_{i_{k}} \neq\emptyset\}} \, \big|\,  F_{i_{k-1}} = f_{i_{k-1}}\right] \leq c_4 \eta  (i_{k}-i_{k-1})^{-(1+\ga^*/2)} \, .
\end{equation}

\medskip
It remains to bound the rest of \eqref{blockbyblock}. We decompose the expectation according to whether the interval $[D_{i_k},F_{i_k}]$ is far from the upper end of the block $\mathcal{B}_{i_k}$, that is,
$D_{i_k}\in ((i_k-1) L , (i_k-\gamma)L]$ or not, where $\gamma := \gep^{(\ga^*\wedge 1) /(1+\ga^*)} \gg\gep $.

Note that if $D_{i_k} \leq (i_k-\gamma)L $ and $F_{i_k} -D_{i_k} \leq \gep L$ then $i_kL - F_{i_k} \geq \frac12 \gamma L$.
We therefore have
\begin{align*}
\bE_{\nu} &\left[ \ind_{\{D_{i_k}\in ((i_k-1) L , (i_k-\gamma)L]\}} \ind_{\{F_{i_k}-D_{i_k}<\gep L\}}\ind_{\{\tau\cap\mB_i =\emptyset \ \forall\, i_{k-1}<i<i_{k},\,\tau\cap \mB_{i_{k}} \neq\emptyset\}} \, \big|\,  F_{i_{k-1}} = f_{i_{k-1}}\right] \\
&\leq \sup_{0\leq m\leq L}\bP_{\nu}  \big( \nu_1 \in [m+j_k L, m+(j_k+1-\gamma)L] \,  | \,  \nu_1>m\big) \times \sum_{\ell=0}^{\gep L} \bP_{\nu}(\ell\in\nu) \bP_{\nu}(\nu_1 \geq  \tfrac12 \gamma L)\, .
\end{align*}
Then, using \eqref{longjump} if $i_k-i_{k-1}\geq 2$, or bounding the probability by $1$ if $i_k-i_{k-1}=1$, we bound this from above for large $L$ by
\begin{equation}\label{smallgap}
c_2 (i_{k}- i_{k-1})^{-(1+\ga^*/2)}\,  \frac{\gp^*(\gamma L)}{ (\gamma L)^{\ga^*}}\sum_{\ell=0}^{\gep L} \bP_{\nu}(\ell\in\nu)
\leq c_5 \gep^{\ga^* \wedge 1} \gamma^{-\ga^*} \,  (i_{k}- i_{k-1})^{-(1+\ga^*/2)}\, .
\end{equation}
For the last inequality, we used \eqref{Doney} to get that $\bP_{\nu}(\ell \in\nu) \leq c_6 \min\{ 1, \gp^*(\ell)^{-1} \ell^{\ga^*-1}  \}$ for all $\ell\geq 1$.

We now deal with the term when $D_{i_k} \in ((i_k-\gamma)L, i_kL]$. We have 
\begin{multline}
\bE_{\nu} \left[ \ind_{\{D_{i_k}\in ((i_k-\gamma)L, (i_k+1)L]\}}\ind_{\{\tau\cap\mB_i =\emptyset \ \forall\, i_{k-1}<i<i_{k},\,\tau\cap \mB_{i_{k}} \neq\emptyset\}} \, \big|\,  F_{i_{k-1}} = f_{i_{k-1}}\right] \\
\leq \sup_{0\leq m\leq L}\bP_{\nu}  \big( \nu_1 \in (m+(j_k+1-\gamma) L, m+(j_k+1)L] \,  | \,  \nu_1>m\big) \, .
\end{multline}
We can now repeat the argument of \eqref{blockbyblock}.
If $i_k-i_{k-1}=1$, we get
\begin{equation} \label{smallgapA}
\sup_{0\leq m\leq L}\bP_{\nu}  \big( \nu_1 \in (m+(1-\gamma) L, m+L] \,  | \,  \nu_1>m\big)
\leq \frac{\gamma L}{\bP(\nu_1\geq L)} \sup_{ (1-\gamma) L \leq x \leq 2 L} \bP_\nu(\nu_1=x) \leq c_7 \gamma \, .
\end{equation}
 If $i_{k}-i_{k-1}\geq 2$, we end up similarly with the bound $c_8 \gamma (i_k-i_{k-1})^{-(1+\ga^*/2)}$.

Combining the bounds \eqref{smallgap} and \eqref{smallgapA}, one gets that 
\begin{equation}
\label{bbb2}
\bE_{\nu} \left[ \ind_{\{F_{i_k}-D_{i_k}<\gep L\}} \ind_{\{\tau\cap\mB_i =\emptyset \ \forall\, i_{k-1}<i<i_{k},\,\tau\cap \mB_{i_{k}} \neq\emptyset\}} \, \big|\,  F_{i_{k-1}} = f_{i_{k-1}}\right] \leq c_9 \gamma (i_{k}-i_{k-1})^{-(1+\ga^*/2)} \, ,
\end{equation}
where we used the fact that $\gep^{\ga^* \wedge 1} \gamma^{-\ga^*} =\gamma$.

Combining \eqref{bbb1} and \eqref{bbb2}, we obtain \eqref{blockbyblock} with $\gd=c_4 \eta+ c_9 \gamma$, completing the proof of \eqref{eq:endchangemeasure} and thus of \eqref{coarsepartition}.

\end{proof}

\subsection{Proof of Lemma \ref{lem:eventJ}: choice of the change of measure, }
\label{sec:eventJ}

We rewrite the partition function in \eqref{eq:cond2} as follows:
\begin{align}
\label{lemma32:1}
\frac{1}{\bP_{\nu}(b-a\in\nu)} \bE_{\sigma} \big[ \ind_{J_L^c} \check Z_{[a,b],0}^{\sigma}\big]
  &= \bE_{\nu}\bE_{\sigma}\big[ \ind_{J_L^c}\,   W(\sigma,\nu,[a,b])  \big| \, a,b\in\nu\big] \notag\\
&= \bE_{\nu}\bE_{\sigma,\tau} \bigg[ \ind_{J_L^c} \, \frac{\prod_{k: \nu_k\in [a,b]} \ind_{\{\sigma_{\nu_k} \in \tau \}}}{\bE_{\tau,\sigma}[\prod_{k: \nu_k  
  \in [a,b]} \ind_{\{\sigma_{\nu_k} \in \tau \}}]} \ \bigg| \  a,b\in\nu \bigg]\notag\\
&= \bE_{\nu}\Big[ \bP_{\sigma,\tau} \big( J_L^c \ |\  A_L(\nu)\big)  \ \Big| \ \nu_0=a\, ;\, b\in\nu \, ;\, \nu \cap (b,L]=\emptyset \Big],
\end{align}
where, for any renewal trajectory $\nu$, we defined the event
\begin{equation}
\label{defAL}
 A_L(\nu) = \left\{ (\sigma,\tau): \sigma_{\nu_k} \in \tau \text{ for all } k \text{ such that } \nu_k \in  [0,L] \right\}.
\end{equation}

Observe that if we find a set $\mG$ of (good) trajectories of $\nu$ satisfying 
\[
  \bP_{\nu} \Big( \mG \, \big| \, \nu_0=a\, ;\, b\in\nu \, ;\, \nu \cap (b,L]=\emptyset \Big)\geq 1-\eta/2
  \]
whenever $L\geq L_0$ and $b-a\geq \gep L$, then
\[\bE_{\nu}\Big[ \ind_{\mG^c}\bP_{\sigma,\tau} \big( J_L^c \ |\  A_L(\nu)\big)  \ \Big| \ \nu_0=a\, ;\, b\in\nu \, ;\, \nu \cap (b,L]=\emptyset \Big] \leq \eta/2.\]
To get \eqref{eq:cond2}, it is therefore enough to show that provided $L$ is large and $b-a \geq \gep L$, we have that for all $\nu\in\mG$
\begin{equation}\label{JLcbound}
\bP_{\sigma,\tau} \big( J_L^c \ |\  A_L(\nu)\big)  \leq \eta/2.
\end{equation}

\medskip

We now construct the event $J_L$ and the set $\mG$ (see \eqref{def:JL} and \eqref{def:G}), and  show \eqref{eq:cond1}-\eqref{eq:cond2}.
Write $L=\gamma_L k_L$ and decompose $[0,L]$ into $\gamma_L$ subblocks of length $k_L$.  
We take
\begin{equation}\label{kLtype}
  k_L \equiv 1\quad  \text{if } \bE_\nu[\nu_1]<\infty, \qquad k_L,\gamma_L \to\infty\quad  \text{to be specified, if } \bE_\nu[\nu_1]=\infty.
\end{equation}
Define
\begin{equation}
 S_{L,j} = \max_{(j-1)k_L<i\leq j k_L}\{\sigma_i - \sigma_{i-1}\}.
\end{equation}
It is standard that $S_{L,1}/B_{k_L}$ has a nontrivial limiting distribution if $k_L\to\infty$ (see Section \ref{sec:notations}), and since $0<\tga<1$ we also have $b_n=B_n$. Define 
\[ Y_L = \big| \{j\leq \gamma_L : S_{L,j} > r_L\,  B_{k_L}\} \big| = \sum_{j=1}^{\gamma_L} \ind_{\{S_{L,j} > r_L \,B_{k_L}\}}, \]
where $r_L\stackrel{L\to\infty}{\to}\infty$ slowly, to be specified.
 Here our use of block size $k_L\gg 1$ when $\bE_\nu[\nu_1]=\infty$ is purely a technical device to deal with the fact we only know $n\mapsto \bP(n\in\tau)$ is approximately monotone, not exactly; the spirit of our argument is captured in the $k_L\equiv 1$ case.  See Remark \ref{rem:monotone} for further comments.

The heuristic is that large gaps $\sigma_i-\sigma_{i-1}$ lower the probability of $A_L(\nu)$, since they make $\tau$ less likely to ```find'' the locations $\sigma_{\nu_k}$; equivalently, $A_L(\nu)$ reduces the occurrence of large gaps, and in particular reduces the probability that $S_{L,j}$ is a large multiple of its typical size $B_{k_L}$.  More precisely,
our goal is to show that the conditioning by $A_L(\nu)$ (for $\nu\in\mG$ with appropriate $\mG$) induces a reduction in $\bEs[Y_L]$ of a size $y_L$ 
which is much larger than $\sqrt{\Var(Y_L)}$ and $\sqrt{\Var(Y_L | A_L(\nu))}$; then, one can take $J_L$ of the form $\big\{ Y_L - \bEs[Y_L] \leq -y_L/2 \big\}$, and obtain $\bPs(J_L)\to 0$, and $\bPst(J_L|A_L(\nu)) \to 1$ as $L\to\infty$.

\medskip
{\bf Step 1.} We first estimate $\bE[Y_L]$ and $\Var(Y_L)$, without the conditioning on $A_L(\nu)$. Suppose first that $k_L\to\infty$. Since $\bPs(\sigma_1 > n) \stackrel{n\to\infty}{\sim} \frac{1}{\tga} \gp(n) n^{-\tga}$, one easily has that, for fixed $r>0$, as $L\to\infty$,
\begin{align}\label{Sjbound}
 \bPs(S_{L,1}> r\, B_{k_L}) & = 1-\Big( 1-\bPs(\sigma_1 > r\, B_{k_L}) \Big)^{k_L} 
   \sim 1 - \exp\left( -k_L\bPs\big(\sigma_1 > r B_{k_L} \big) \right) \notag\\
 &\sim  1 - \exp\left( -\frac{1}{\tga} r^{-\tga} k_L \tilde \gp( B_{k_L}) B_{k_L}^{-\tga} \right) \sim  1 - \exp\left( -\frac{1}{\tga} r^{-\tga} \right),
\end{align}
where we used \eqref{defanbn} for the last equivalence.
 Therefore when $r_L\stackrel{L\to\infty}{\to}\infty$ slowly enough we have
\begin{equation} \label{Sjbound2}
  \bPs(S_{L,1} > r_L B_{k_L} ) \stackrel{L\to\infty}{\sim} k_L\bPs(\sigma_1 > r_L\, B_{k_L} ) \stackrel{L\to\infty}{\sim} \frac{1}{\tga} r_L^{-\tga},
\end{equation}
and hence
\begin{equation}
\label{eq:EandVarYL}
\begin{split}
\bE[Y_L]&= \gamma_L \bP(S_{L,1} > r_L\, B_{k_L}) \stackrel{L\to\infty}{\sim}  \frac{1}{\tga} \gamma_L r_L^{-\tga}\, ,  \\
 \Var(Y_L)&=\gamma_L \bP(S_{L,1} > r_L\, B_{k_L}) \big(1-\bP(S_{L,1} > r_L\, B_{k_L})  \big) \stackrel{L\to\infty}{\sim} \frac{1}{\tga} \gamma_L r_L^{-\tga}\, . 
\end{split}
\end{equation}

In the alternative case $k_L\equiv 1$, we have $B_{k_L}\equiv d_\sigma$ (see \eqref{ddef}, \eqref{defAnBn}) so whenever $r_L\to\infty$ we have
\[
  \bPs(S_{L,1}> r_L\, B_{k_L}) = \bPs(\sigma_1>d_\sigma r_L) \stackrel{L\to\infty}{\sim} \frac{1}{\tga d_\sigma^{\tga}} r_L^{-\tga}\tilde\gp(r_L),
\]
so similarly to \eqref{eq:EandVarYL},
\begin{equation}\label{eq:EandVarYL2}
  \bE[Y_L] \stackrel{L\to\infty}{\sim} \Var(Y_L) \stackrel{L\to\infty}{\sim} \frac{1}{\tga d_\sigma^{\tga}} r_L^{-\tga}\tilde\gp(r_L)L.
\end{equation}

\medskip

{\bf Step 2.} We now study the influence of  the conditioning by $A_L(\nu)$ on the events $\{\sigma_i - \sigma_{i-1} > r\, B_{k_L}\}$ and $\{S_{L,j} > r_L\, B_{k_L}\}$.  As we have noted, heuristically one expects the probabilities to decrease, and this is readily shown to be true if $\bPt(n\in\tau)$ is decreasing in $n$, but in general such monotonicity only holds asymptotically.  So instead 
we show that, for $L$ large enough,
\begin{equation}
\label{epeffect}
  \bPst\left( S_{L,j} > r_L \, B_{k_L}\ \big|\ A_L(\nu), \{S_{L,\ell}, \ell \neq j\} \right) \leq (1+2\tilde\epsilon_L)\, \bPs(S_{L,j} > r_L  B_{k_L}).
\end{equation}
where $\tilde\epsilon_L \to 0$ is defined below, see \eqref{defepsilonk}.

\smallskip
To prove this, we first show that for $n\geq j \geq 1$ and $t\geq 0$, $\bPst(\sigma_n\in \tau \mid \sigma_j - \sigma_{j-1} >  r_L \, B_{k_L}\, ,\, \sigma_{j-1} = t)$ is not much more than $\bPst(\sigma_n\in\tau \mid \sigma_{j-1} = t)$.
To this end, we can form a coupling $(\rho,\rho')$ such that $\rho$ has the distribution $\bPs(\sigma_1 \in \cdot)$, $\rho'$ has the distribution $\bPs(\sigma_1\in \cdot \mid \sigma_1 \geq r_L \, B_{k_L})$, and  $\rho\leq\rho'$. Then, thanks to the exangeability of the $(\sigma_i-\sigma_{i-1})_{i\geq 1}$, we have for any $n\geq j \geq 1$ and $t\geq 0$,
\begin{equation}\label{rhoprime}
  \bPst(\sigma_n\in \tau \mid \sigma_j - \sigma_{j-1} > r_L\, B_{k_L}, \sigma_{j-1} = t)
    = \bP_{\sigma,\tau,\rho'}(t + \rho' + \sigma_{n-j} \in \tau),
\end{equation}
and 
\begin{equation}\label{rho}
  \bPst(\sigma_n\in\tau \mid \sigma_{j-1} = t) = \bP_{\sigma,\tau,\rho}(t + \rho + \sigma_{n-j} \in \tau).
\end{equation}
To bound \eqref{rhoprime} relative to \eqref{rho}, we proceed as follows.
Let $R\geq 1$ and observe that
\begin{align} \label{Rlowerbd}
  \bP_{\sigma,\tau,\rho}\big(t + \rho\ +\ &\sigma_{n-j} \in \tau \mid t + \sigma_{n-j} \leq R \big) \notag\\
  &\geq  \bP_\rho(R \leq \rho < 2R) \ \bP_{\sigma,\tau,\rho}\big(t + \rho + \sigma_{n-j} \in \tau \mid R \leq \rho < 2R, t + \sigma_{n-j} \leq R \big) \notag\\
  &\geq \bP_\rho(R \leq \rho < 2R) \times \min_{R\leq m < 3R} \bPt(m\in\tau) \, ,
\end{align}
which depends only on $R$.
On the other hand, since $\rho'\geq r_L\, B_{k_L}$,
\begin{multline}
\label{belowR}
\bP_{\sigma,\tau,\rho'}(t + \rho' + \sigma_{n-j} \in \tau \mid t + \sigma_{n-j} \leq R) \leq \max_{m\geq r_L \, B_{k_L}} \bPt(m\in\tau) \\
\leq   \bP_{\sigma,\tau,\rho}\big(t + \rho + \sigma_{n-j} \in \tau \mid t + \sigma_{n-j} \leq R \big),
\end{multline}
where by \eqref{Rlowerbd} the last inequality holds for any fixed sufficiently large $R$, provided that $L$ is large enough.

Now from \eqref{Doney}, since $\ga<1$, if $R$ and $L$ are large enough, we have that
\[
  \bPt(m\in\tau) \geq \bPt(n\in\tau) \quad \text{for all } R \leq m \leq \tfrac{1}{2} r_L\, B_{k_L} \text{ and } n \geq r_L\, B_{k_L},
\]
which yields
\begin{align}\label{aboveR}
  \bP_{\sigma,\tau,\rho}&\left(t + \rho + \sigma_{n-j} \in \tau\ \Big|\ t + \sigma_{n-j} > R, t + \rho + \sigma_{n-j} \leq \tfrac12 r_L\, B_{k_L}\right) \notag\\
  &\geq \bP_{\sigma,\tau,\rho'}\left(t + \rho' + \sigma_{n-j} \in \tau\ \Big|\ t + \sigma_{n-j} > R, t + \rho + \sigma_{n-j} \leq \tfrac12 r_L\, B_{k_L}\right).
\end{align}
To control the contribution when $t + \rho + \sigma_{n-j} \geq \tfrac12  r_L\, B_{k_L}$, we define the quantity
\begin{equation}
\label{defepsilonk}
  \epsilon_k = \sup\left\{ \frac{ \bPt(u\in\tau) }{ \bPt(v\in\tau) }-1: u\geq v \geq k \right\},
\end{equation}
which satisfies $\lim_{k\to\infty} \epsilon_k = 0$. Then, writing $\tilde\epsilon_L$ for $\epsilon_{  r_L\, B_{k_L}/2}$, we have
\begin{align}\label{aboveR2}
  (1+\tilde\epsilon_L)&\bP_{\sigma,\tau,\rho}\left(t + \rho + \sigma_{n-j} \in \tau\ \Big|\ t + \sigma_{n-j} > R, 
    t + \rho + \sigma_{n-j} > \tfrac12 r_L\, B_{k_L}\right) \notag\\
  &\geq \bP_{\sigma,\tau,\rho'}\left(t + \rho' + \sigma_{n-j} \in \tau\ \Big|\ t + \sigma_{n-j} > R, t + \rho + \sigma_{n-j} > \tfrac12 r_L\, B_{k_L}\right).
\end{align}
Combining \eqref{belowR}, \eqref{aboveR} and \eqref{aboveR2}, we obtain
\[
  \bP_{\sigma,\tau,\rho'}\big(t + \rho' + \sigma_{n-j} \in \tau \big) \leq (1+\tilde\epsilon_L)\ \bP_{\sigma,\tau,\rho} \big(t + \rho + \sigma_{n-j} \in \tau \big)
\]
By \eqref{rho} and \eqref{rhoprime} this is the same as
\[
  \frac{ \bPst\big(\sigma_n\in \tau \mid \sigma_j - \sigma_{j-1} > r_L\, B_{k_L}, \sigma_{j-1} = t \big) }
    { \bPst(\sigma_n\in\tau \mid \sigma_{j-1} = t) } \leq 1+\tilde\epsilon_L,
\]
which is equivalent to
\[
  \frac{ \bPst(\sigma_j - \sigma_{j-1} > r_L\, B_{k_L} \mid \sigma_n\in \tau, \sigma_{j-1} = t) }
    { \bPst(\sigma_j - \sigma_{j-1} > r_L\, B_{k_L} \mid \sigma_{j-1} = t) } \leq 1+\tilde\epsilon_L,
\]
or by independence,
\begin{equation}\label{rewrite}
  \frac{ \bPst(\sigma_j - \sigma_{j-1} > r_LB_{k_L} \mid \sigma_n\in \tau, \sigma_{j-1} = t) }
    { \bPst(\sigma_j - \sigma_{j-1} > r_LB_{k_L}) } \leq 1+\tilde\epsilon_L.
\end{equation}
Since $t$ and $n$ are arbitrary, this shows that for arbitrary $\nu$, conditionally on $A_L(\nu)$, the 
variables $(\ind_{\{\sigma_i - \sigma_{i-1} > r_L\, B_{k_L}\}})_{ i \leq L}$ are (jointly) stochastically dominated by a Bernoulli sequence of parameter $(1+\tilde\epsilon_L)\bPs(\sigma_1 > r_L\, B_{k_L})$.  In fact, by exchangeability this domination holds for any given $\sigma_\ell - \sigma_{\ell-1}$ conditionally on any information about the other variables $\sigma_i - \sigma_{i-1}$.

\smallskip
We need to quantify what this conclusion says about the variables $S_{L,j}$.  The following is easily established:  there exists $\delta_0>0$ such that for all $p\in(0,1),\epsilon \in(0,1)$ and $k\geq 1$ with $kp<\delta_0$, we have
\[
  \frac{ 1 - (1 - (1+\epsilon)p)^k }{ 1 - (1-p)^k } \leq 1+2\epsilon.
\]
Taking $p = \bPs(\sigma_1 > r_L\, B_{k_L}), k=k_L$ so that $pk \leq \gd_0$ provided $r_L$ is large (see \eqref{Sjbound2}), taking $\epsilon = \tilde\epsilon_L$ and using the stochastic domination, we obtain that for $L$ large, for all $\nu$, \eqref{epeffect} holds.

\medskip

{\bf Step 3.} We next want to show that for certain $\nu$ and $j$ we can make a much stronger statement than \eqref{epeffect}. Specifically, with $L$ fixed, we say an interval $\mI$ is \emph{visited} (in $\nu$) if $\mI\cap \nu \neq \phi$, and (for $3\leq j<\gamma_L-2$) we say 
the sub-block $\mQ_j=((j-1)k_L,jk_L]$ is 
\emph{capped} (in $\nu$) if $\mQ_{j-2}\cup \mQ_{j-1}$ and $\mQ_{j+1}\cup\mQ_{j+2}$ are both visited.  

We now prove that, if $j$ is such that $\mQ_j$ is capped in $\nu$, we have
\begin{equation}
\label{whencapped}
\bPst\left( S_{L,j}> r_L B_{k_L} \, \big| \, A_L(\nu) \right) \leq 2(r_L)^{-(1-\ga)/2} \, \bPst \big( S_{L,j} > r_L B_{k_L}\big)\, .
\end{equation}

\smallskip

Suppose that $\mQ_j$ is capped in $\nu$  , and that $s$ is an index such that $(\nu_{s-1},\nu_s] \cap \mQ_j  \neq \emptyset$: we  write $s\prec \mQ_j$.
Note that the events
\[
  H_{s,j} = \left\{ \gs: \max_{i\in (\nu_{s-1},\nu_s] \cap \mQ_j} (\gs_i-\gs_{i-1}) > r_L\, B_{k_L} \right\}
\]
are conditionally independent given $A_L(\nu)$, and (by exchangeability) satisfy
\begin{equation}\label{HgivenA}
  \bPst\left( H_{s,j}\ \big|\ A_L(\nu) \right) = \bPst\left( \max_{1\leq i\leq \ell} (\gs_i-\gs_{i-1}) > r_L\, B_{k_L}\ \big|\ \gs_k\in \tau \right),
\end{equation}
where $k:=\nu_s-\nu_{s-1}\leq 5 k_L$ (since $\mQ_j$ is capped), and $\ell:= |(\nu_{s-1},\nu_s] \cap \mQ_j| \leq k$.

\smallskip
Furthermore, since $\ga\in[0,1)$, by \eqref{Doney} we have $\bP(m\in\tau)= \bar \gp(m) m^{-(1-\ga)}$ for some slowly varying $\bar\gp(\cdot)$. Therefore
\begin{align}\label{condeffect}
\bPst\left( \sigma_k\in\tau\, \Big|\, \max_{1\leq i\leq \ell} (\gs_i-\gs_{i-1}) > r_L \, B_{k_L} \right) &\leq \max_{m\geq r_L\, B_{k_L}} \bPt(m\in\tau)
   \notag\\ &\leq c_{10} (r_L\, B_{k_L})^{-(1-\alpha)} \bar\varphi(r_L\, B_{k_L}) \, .
\end{align}
On the other hand, since $0<\tga<1$,  $\sigma_n/B_n$ has a non-degenerate limiting distribution with positive density on $(0,\infty)$ (see Section \ref{sec:notations}). Using again that $\bP(m\in\tau)=\bar \gp(m) m^{-(1-\ga)}$, we therefore find that there exist a constant $c_{11}$ such that, for any $k\leq 5k_L$,
\begin{equation}
  \bPst(\sigma_k\in\tau) \geq \bPs(\sigma_k \in [\bar d_\tau B_k , 2\bar d_\tau B_k ]) \ \min_{\bar d_\tau\leq m\leq 2\bar d_\tau B_{k}} \bPt(m\in\tau) 
    \geq c_{11} B_{k_L}^{-(1-\ga)} \bar \gp(B_{k_L}) \, .
\end{equation}

With \eqref{condeffect} this  shows
there exists some $c_{12}$ such that for large $L$, provided $r_L\to\infty$ slowly enough,
\begin{equation}\label{rLslow}
  \frac{ \bPst\left( \sigma_k\in\tau\, \big|\, \max_{1\leq i\leq \ell} (\gs_i-\gs_{i-1}) > r_L\, B_{k_L} \right) }{ \bPst(\sigma_k\in\tau) } \leq (c_{12}\vee \bar\gp(r_L))\, r_L^{-(1-\alpha)} \leq r_L^{-(1-\alpha)/2},
\end{equation}
or equivalently, using \eqref{HgivenA},
\begin{equation}\label{effect2}
  \frac{ \bPst\left( H_{s,j}\ \big|\ A_L(\nu) \right) }{ \bPst(H_{s,j}) } =
  \frac{ \bPst\left( \max\limits_{1\leq i\leq \ell} (\gs_i-\gs_{i-1})> r_L\, B_{k_L}\ \big|\ \sigma_k\in\tau \right) }{ \bPst \left(\max\limits_{1\leq i\leq \ell} (\gs_i-\gs_{i-1}) > r_L\, B_{k_L} \right) } \leq (r_L)^{-(1-\alpha)/2}.
\end{equation} 
Since the $\{H_{s,j}: s\prec \mQ_j\}$ are independent for fixed $j$ (even when conditioned on $A_L(\nu)$), with $\bigcup_{s\prec\mQ_j} H_{s,j} = \{ S_{L,j} > r_L\, B_{k_L} \}$ and $\bPst\left( S_{L,j} > r_L\, B_{k_L} \right)\to 0$ as $L\to\infty$, we have from \eqref{effect2} that for large $L$
\begin{align} \label{condindep}
  \bPst\left( S_{L,j} > r_L\, B_{k_L}\ \big|\ A_L(\nu) \right) 
  &\leq \sum_{s\prec\mQ_j} \bPst\left( H_{s,j}\ \big|\ A_L(\nu) \right) \notag\\
  & \leq (r_L)^{-(1-\alpha)/2} \sum_{s\prec\mQ_j} \bPst(H_{s,j}) \notag\\
  &\leq 2(r_L)^{-(1-\alpha)/2} \bPst\left( S_{L,j} > r_L\, B_{k_L} \right)\, .
\end{align}

\medskip

{\bf Step 4.} We now control the number of capped blocks $\mQ_j$.
Let 
\[
  \mC(\nu) = \{j\leq \gamma_L: \mQ_j \text{ is capped in } \nu\}, \qquad m^*(n):=\bE_{\nu}[\nu_1 \wedge n]\ ,
  \]
and define what is heuristically a lower bound for the typical size of $|\mC(\nu)|$ (see \eqref{probagood} and its proof):
\begin{equation}
\label{def:D}
D_L:= 
\begin{cases}
\gamma_L^{\ga^*}\, \gp^*(k_L) \gp^*(L)^{-1}  & \text{ if } 0<\ga^*<1 \, ,\\
\gamma_L\, m^*(k_L) m^*(L)^{-1} & \text{ if } \ga^*=1,\, \bE_{\nu}[\nu_1]=+\infty\, ,\\
L \gp^*(k_L) k_L^{-\ga^*}& \text{ if } \bE_{\nu}[\nu_1]<+\infty\, .
\end{cases}
\end{equation}
We now restrict our choice of  $\gamma_L, k_L$ as follows (with further restriction to come):
\begin{itemize}
\item[(i)] if $\bE_{\nu}[\nu_1]<+\infty$, we choose $k_L\equiv 1,\gamma_L=L$, and get $D_L=L\gp^*(1)$ ;
\item[(ii)] if $\bE_{\nu}[\nu_1]=+\infty$ we choose $1\ll k_L\ll L$ such that $\gamma_L=L/k_L$ is slowly varying, and as $L\to\infty$ we have $\gamma_L\to \infty$ slowly enough so $\gp^*(k_L) \sim\gp^*(L)$ if $\ga^*<1$ or $m^*(k_L)\sim m^*(L)$ if $\ga^*=1$, which is possible since for $\ga^*=1$, $m^*(n)$ is slowly varying. We obtain $D_L\sim \gamma_L^{\ga^*}\gg 1$.
\end{itemize}
We then fix $\kappa > 0$ and define the set $\mG$ of \emph{good} trajectories by
\begin{equation}
\label{def:G}
 \mG = \{ \nu:  |\mC(\nu)| \geq \kappa D_L \}.
\end{equation}
We wish to choose $\kappa = \kappa(\eta,\gep)$ sufficiently small so that, provided that $L$ is large enough and $b-a \geq \gep L$, one has
\begin{equation}
\label{probagood}
\bP_{\nu}\big( \mG\, |\, \nu_0=a\, ;\, b\in\nu \, ;\, \nu \cap (b,L]=\emptyset \big) \geq 1-\eta/2 \, .
\end{equation}
For a fixed $\gep$, for large $L$ we have $D_M\geq \gep D_L$ for all $M\geq \gep L$, so it suffices to prove this for $\gep=1$, that is, to consider only $a=0,b=L$. We now consider two cases.
 
\smallskip
\textbullet\ \textit{Case of $\ga^*<1$ or $\bE_{\nu}[\nu_1]<+\infty$.}
In that case, it is standard that there exists a constant $c_{13}(\eta)$ such that, defining $n_L := c_{13}\, \min(\gp^*(L)^{-1} L^{\ga^*},L)$,  we have for $L$ large enough
\begin{equation}
\label{countnu}
\bP_{\nu}\big( |\nu\cap[0,L/2]|\leq n_L \big) \leq \eta \, .
\end{equation}
Observe that each gap $(\nu_{s-1},\nu_s]$ of $\nu$ of length $\nu_s-\nu_{s-1} \in [2 k_L, 3 k_L)$ contains a capped sub-block. We therefore simply need to count the number of such gaps: denote
\[
  V_n:= \sum_{s=1}^n \ind_{\{\nu_s - \nu_{s-1} \in[2k_L,3k_L)  \}}.
\]
There exists a constant $c_{14}=c_{14}(\eta)$ such that, provided that $n\bP_{\nu}(\nu_1\in[2k_L,3k_L))$ is large enough, one has
\begin{equation}
\label{pVn}
\bP_{\nu}\big(V_n \leq c_{14}\, n\, \bP_{\nu}(\nu_1\in[2k_L,3k_L)) \big) \leq \eta \, .
\end{equation}
Using that $\bP_{\nu}(\nu_1\in[2k_L,3k_L) )$ is of order $\gp^*(k_L) k_L^{-\ga^*}$ we get that $ D_L \leq c_{15}^{-1} n_L \bP_{\nu}(\nu_1\in[2k_L,3k_L) )$ for some $c_{15}(\eta)$, and we obtain
\begin{align*}
\bP_{\nu}&\big( \mG^c\, |\, L\in\nu \big) \leq \bP_{\nu} \Big( \big| \big\{ s \leq |\nu\cap[0,L]|:\, \nu_s-\nu_{s-1} \in[2k_L, 3k_L) \big\} \big| < \kappa D_L\, \Big|\, L\in\nu \Big) \\
 &\leq \bP_{\nu}\Big( |\nu\cap[0,L/2]|\leq n_L \big| L\in\nu \Big)  + \bP_{\nu}\Big( |\nu\cap[0,L/2]|\geq n_L ; V_{n_L}< \gk D_L \big| L\in\nu \Big)\\
&\leq c_{16}\eta + c_{16} \bP_{\nu}\Big( V_{n_L}< \frac{\gk}{c'_{\eta}} n_L \bP_{\nu}(\nu_1\in[2k_L,3k_L) ) \Big) \\
&\leq 2c_{16} \eta\, .
\end{align*}
In the third inequality, we used Lemma \ref{lem:condtioning} (Lemma A.2 in \cite{GLT09}, that we recall in the Appendix since we use it several times) to remove the conditioning at the expense of the constant $c_{16}$. In the last inequality, we used \eqref{pVn}, taking $\kappa$ sufficiently small.  


\smallskip
\textbullet\ \textit{Case of $\ga^*=1$, $\bE_{\nu}[\nu_1]=+\infty$.}
First, we claim that with our choice of $k_L$, we have 
$
 \bP_\nu\left( \mQ_j \text{ is capped}\right) \stackrel{L\to\infty}{\to} 1 .
$
Indeed, summing over possible locations of a last visit to $\mQ_j$ gives that for $1\leq j\leq\gamma_L$ and $i\leq (j-1)k_L$,
\begin{align} \label{locations}
 \bP_\nu \Big( \mQ_j \text{ is visited } \big|\ i\in \nu \Big) &\geq \sum_{m=0}^{k_L-1} \bP_\nu( jk_L-m-i \in \nu)\bP_\nu(\nu_1>m) \notag\\
 &\geq \inf_{\ell\leq L} \bP_\nu(\ell\in \nu) m^*(k_L) \notag\\
 &\sim \frac{ m^*(k_L) }{ m^*(L) }\sim 1\quad \text{as } L\to\infty,
\end{align}
where we used $\sum_{m=0}^{k_L-1} \bP(\nu_1>m) = m^*(k_L)$ in the second inequality, and \eqref{Doney} in the last line.  Note that the lower bound in \eqref{locations} is uniform in the specified $i,j$.
Therefore, applying \eqref{locations} three times, we obtain uniformly in $j\leq \gamma_L$
\begin{equation}\label{cappedQj}
  \bP_\nu\left( \mQ_j \text{ is capped}\right) \geq \bP_\nu\left( \mQ_{j-1},\mQ_j,\mQ_{j+1} \text{ are visited}\right) 
    \stackrel{L\to\infty}{\to} 1\, .
\end{equation}
Recall $D_L\sim\gamma_L$, and 
denote $W_L:= |\{j\leq \gamma_L/2 \, ;\, \mQ_j \text{ is not capped}\}|$, so that $\bE_\nu[W_L] = o(\gamma_L)$, by \eqref{cappedQj}.
Taking $\kappa=1/4$, thanks to Markov's inequality we then have for large $L$
\begin{equation}\label{GcgivenL}
\bP_{\nu}\big( \mG^c\, |\, L\in\nu \big) \leq  \bP_{\nu}\left( W_L\geq  \tfrac18 \gamma_L \mid L\in\nu \right)
  \leq c_{16} \bP_{\nu}\left( W_L\geq  \tfrac18 \gamma_L \right)
  \to 0 \quad \text{  as } L\to\infty,
\end{equation}
where in the second inequality, we used Lemma \ref{lem:condtioning} to remove the conditioning at the expense of the constant $c_{16}$.

\medskip

{\bf Step 5.} We now have the ingredients to control how the conditioning by $A_L(\nu)$ shifts $\bEs[Y_L]=\gamma_L \bPs(S_{L,1} > r_L B_{k_L})$, when $\nu\in \mG$.

We wish to choose $\gamma_L,k_L$ so that, for $L$ sufficiently large,
\begin{equation}
\label{cond:a} 
\gamma_L  \tilde \epsilon_{L} \, \leq   \frac18 \kappa D_L \le \frac18 |\mC(\nu)| \quad\text{for all } \nu\in\mG.
\end{equation}
The second inequality is just the definition of $\mG$.  When $\ga^*\geq 1$, $\gamma_L$ and $D_L$ are of the same order, so the first inequality follows from $\tilde\gep_L\to 0$. So consider the case $\ga^*<1$, where $D_L\sim\gamma_L^{\alpha^*}$ by (ii) after \eqref{def:D}. Here the first inequality in \eqref{cond:a} follows if we have $\gamma_L^{1-\ga^*} \ll \tilde \epsilon_L^{-1}$.  Since $\gep_n\searrow 0$ and $L/k_L$ is slowly varying, we have $B_{k_L} \gg L$ so $\tilde\gep_L \leq \gep_L$ so a sufficient condition is $\gamma_L^{1-\ga^*} \ll \gep_L^{-1}$. But our only restriction so far is from (ii) after \eqref{def:D}, that $\gamma_L\to\infty$ slowly enough, so we may choose $\gamma_L$ to satisfy this sufficient condition also.

Thanks to \eqref{epeffect}, \eqref{effect2} and \eqref{cond:a},  we have for $\nu\in\mG$ and large $L$
\begin{align}\label{rightsized}
  \bE_{\sigma,\tau}&[Y_L\mid A_L(\nu)] \notag\\
  &= \sum_{j=1}^{ \gamma_L} \bP_{\sigma,\tau}(S_{L,j} > r_L\, b_{k_L} \mid A_L(\nu)) \notag\\
  &\leq 2|\mC(\nu)| \, (r_L)^{-(1-\alpha)/2}\bPs(S_{L,1} > r_L \, B_{k_L})
    + (\gamma_L - |\mC(\nu)|)(1+2\tilde \epsilon_L) \bPs(S_{L,1} > r_L\, B_{k_L}) \notag\\
  &\leq \Big[ \gamma_L - |\mC(\nu)| + 2|\mC(\nu)| \, (r_L)^{-(1-\alpha)/2} + 2\gamma_L \tilde \epsilon_L \Big]
    \bPs(S_{L,1} > r_L B_{k_L}) \notag\\
  &\leq \left[ \gamma_L - \tfrac12 \kappa D_L \right] \bPs(S_{L,1} > r_L B_{k_L}).
\end{align}

\medskip
{\bf Step 6.} We now define
\begin{align}
  J_L &:= \left\{ \sigma: \sum_{j\leq \gamma_L} \ind_{\{S_{L,j} > r_L\, B_{k_L}\}} \leq \left[ \gamma_L - \tfrac{1}{4}\gk D_L \right] \bPs(S_{L,1} > r_L\, b_{B_L}) \right\} \notag\\
  &=\left\{ Y_L -\bE[Y_L] \leq - \frac{\gk}{4} D_L \bPs(S_{L,1} > r_L\, B_{k_L}) \right\}.
  \label{def:JL}
\end{align}
Let us compare $\bP(J_L)$ and $\bP(J_L^c | A_L(\nu))$ for $\nu\in\mG$.
Since $\ga^*>1/2$, $\gamma_LD_L^{-2}$ is of order $L^{-(2\ga^*-1)\wedge 1}\to 0$, so we can choose $r_L$ to satisfy $\gamma_LD_L^{-2}  r_L^{\tga}\to 0$,
which is compatible with our previous requirement on $r_L$ involving \eqref{rLslow}.
Using \eqref{eq:EandVarYL} and Chebyshev's inequality we then get that
\begin{equation}
\label{chebichevJ}
\bP(J_L)\leq \frac{16\gamma_L}{\gk^2 D_L^2 \bPs(S_{L,1} >r_L\, b_{k_L})} \stackrel{n\to\infty}{\sim}  \frac{16 \, \tga}{\gk^2} \, \gamma_LD_L^{-2}  r_L^{\tga}
  \to 0\quad\text{as } L\to\infty.
\end{equation}
On the other hand, by \eqref{epeffect}, conditionally on $A(\nu)$ with $\nu\in \mG$,  the variables $\ind_{\{S_{L,j} > r_L b_{k_L}\}}, 1\leq j\leq \gamma_L$, are (jointly) stochastically dominated by a collection of independent Bernoulli variables with parameter  $(1+2\tilde \epsilon_{L}) \bPs(S_{L,1}>r_L B_{k_L})$. Hence, $\Var(Y_L | A_L(\nu)) \leq 2 \gamma_L \bPs(S_{L,1}>r_L B_{k_L})$ and as in \eqref{chebichevJ} we obtain
\begin{equation}
\label{probaJc}
  \bPst(J_L^c \mid A_L(\nu))\leq \bP\Big(Y_L - \bEst[Y_L | A_L(\nu)] \geq  \frac{\kappa}{4} D_L \bPs(S_{L,1} > r_L\, b_{k_L}) \Big)\stackrel{L\to\infty}{\to} 0 \, .
\end{equation}
We have thus proved \eqref{eq:cond1} and \eqref{JLcbound}, and hence also \eqref{eq:cond2}.
\qed

\begin{remark}\rm
\label{rem:monotone}
If $n\mapsto \bP(n\in\tau)$ is non-increasing, then we have that $\epsilon_k=0$ for all $k$, and we can replace \eqref{epeffect} with
\[
  \bPst\left( S_{L,j} > r_L \, B_{k_L}\ \big|\ A_L(\nu), \{S_{L,\ell}, \ell \neq j\} \right) \leq \bPs(S_{L,j} > r_L \, B_{k_L})\, .
\]
The term $\gamma_L \tilde\epsilon_L$ does not appear in the computation in \eqref{rightsized}, and we can drop the condition \eqref{cond:a}. We can therefore choose $\gamma_L=L,k_L=1$ in all cases, not just when $\bE_\nu[\nu_1]<\infty$. 

Then for $\ga^*<1$, we have $ D_L$ of order $L^{\ga^*}  \varphi^*(L)^{-1}$, and the condition $\gamma_LD_L^{-2}\to 0$ in Step 6 becomes
\[
L^{2\ga^*-1} \gp^*(L)^{-2} \gg 1.
\]
We therefore need $\ga^*>1/2$, or $\ga^*=1/2$ and $\varphi^*(L) \stackrel{L\to\infty}{\to} 0$, going slightly beyond the condition $\ga^*>1/2$.
In Appendix \ref{app:K*}, we show that $\gp^*(n) \sim c \tilde \varphi(b_n)^{-(1-\ga)/\tilde{\ga}} \varphi_0(b_n)$, with $\varphi_0(k) = \gp(k)^{-1}$ if $\ga\in(0,1)$, and $\bar\gp(k) = \gp(k)/\bPt(\tau_1>k)^2$ if $\ga=0$, cf.~\eqref{Doney}. We have that $\lim_{n\to\infty}\varphi^*(n)  = 0$ if and only if $\lim_{k\to\infty} \tilde \varphi(k)^{(1-\ga)/\tilde{\ga}} \varphi_0(k) = +\infty$.
\end{remark}

\begin{proof}[Proof of Lemma \ref{keylemma}]
The proof of Lemma \ref{lem:eventJ} is for a single interval $[0,L]$, and we now adapt it to the whole system: we take the same definition for $J_L$.
Then, recalling $g_\mI$ from \eqref{defg}, we have
similarly to \eqref{lemma32:1}
\begin{multline}
\label{lemma33:1}
\bEs\Big[ g_\mI \mZ_{\mI}\Big]
 = \sumtwo{d_{i_1}\leq f_{i_1}}{d_{i_1},f_{i_1}\in \mB_{i_1}} \cdots \sumtwo{d_{i_m}\leq 
   f_{i_m}=nL}{d_{i_m}\in\mB_{i_m}}  \prod_{k=1}^m K^*(d_{i_k}-f_{i_{k-1}}) \bP(f_{i_k} -d_{i_k}\in\nu) \\
\times   \bE_{\nu}\Big[ \bEst\Big( \prod_{k=1}^m (\eta+ \ind_{\{\sigma_{\mB_{i_k}}\in J_L^c\}})\ \big|\  A(\nu) \Big)  \Big|  
  \mE_{(d_{i_k},f_{i_k})_{1\le k\leq m}} \Big],
\end{multline}
where, setting $\nu_{\mB_i}=\nu\cap \mB_i$,
\[
\mE_{(d_{i_k},f_{i_k})_{1\le k\leq m}}  :=\big\{ \nu:   \forall 1\leq  k \leq m \, , \, \min(\nu_{\mB_{i_k}})=d_{i_k}, \max(\nu_{\mB_{i_k}})=f_{i_k}; \nu_{\mB_{i}}=\emptyset \, \text{ if } i\notin\mI \big\} .
\]

Now, for any $1\leq k\leq m$ such that $f_{i_k}-d_{i_k}\leq \gep L$ or $\nu_{\mB_{i_k}}\notin\mG$, we bound $\ind_{\{\sigma_{\mB_{i_k}}\in J_L^c\}}$ by $1$. We get the bound
\begin{equation}\label{indicators}
  \eta+ \ind_{\{\sigma_{\mB_{i_k}}\in J_L^c\}} \leq 
  \eta+ \ind_{\{f_{i_k}-d_{i_k}\leq \gep L\}} +  \ind_{\{f_{i_k}-d_{i_k}>\gep L \}} \ind_{\{\nu_{\mB_{i_k}} \notin\mG\}} 
  + \ind_{\{\sigma_{\mB_{i_k}}\in J_L^c\}}  \ind_{\{\nu_{\mB_{i_k}} \in\mG\}}.
\end{equation}
By expanding the product over $k\in\{1,\ldots, m\}$, we obtain
\begin{align}
\label{bigexpansion}
\bE_{\nu}&\Big[ \bEst\Big( \prod_{k=1}^m (\eta+ \ind_{\{\sigma_{\mB_{i_k}}\in J_L^c\}})  \mid  A(\nu) \Big)  \Big|  \mE_{(d_{i_k},f_{i_k})_{1\le k\leq m}} \Big] \notag\\
&\leq \sumtwo{K_1,K_2,K_3,K_4 \ \text{disjoint}}{\bigcup_{j=1}^4 K_j = \{1,\ldots,m\}} \eta^{|K_1|} \prod_{k\in K_2} \ind_{\{f_{i_k}-d_{i_k}\leq \gep L\}} \notag \\
 & \hspace{1.5cm} \bE_\nu \Big[   \prod_{k\in K_3} \ind_{\{f_{i_k}-d_{i_k}>\gep L \}} \ind_{\{\nu_{\mB_{i_k}} \notin\mG\}}   \bEst\Big( \prod_{k\in K_4} \ind_{\{\sigma_{\mB_{i_k}}\in J_L^c\}}  \ind_{\{\nu_{\mB_{i_k}} \in\mG\}}\ \Big|\ A(\nu)\Big)  \, \Big| \, \mE_{(d_{i_k},f_{i_k})_{1\le k\leq m}} \Big].
\end{align}

The argument in Step 6 of the proof of Lemma \ref{lem:eventJ}, using domination of the variables $(\ind_{S_{L,j} > r_L B_{k_L}})$ in some $\mB_{i_k}$ by independent Bernoullis, remains valid if we also condition on any information about the other $\mB_i,i\neq i_k$.  This means we can ignore the dependencies between different $i_k \in \mI$, and if $L$ is sufficiently large we get that, since $\nu_{\mB_{i_k}}\in\mG$ for all $k\in K_4$
\begin{equation}\label{expansionterm4}
 \bEst\Big[ \prod_{k\in K_4} \ind_{\{\sigma_{\mB_{i_k}}\in J_L^c\}}  \ind_{\{\nu_{\mB_{i_k}} \in\mG\}} \mid A(\nu)\Big]   \leq (\eta/2)^{|K_4|}\, ,
\end{equation}
where the bound comes from~\eqref{probaJc}.

Then,  by independence of the $\nu_{\mB_{i_k}} \cap [d_{i_k},f_{i_k}]$ conditionally on $\mE_{(d_{i_k},f_{i_k})_{1\le k\leq m}} $, we get that
\begin{align}
\label{expansionterm3}
\bE_\nu &\Big[   \prod_{k\in K_3} \ind_{\{f_{i_k}-d_{i_k}>\gep L \}} \ind_{\{\nu_{\mB_{i_k}} \notin\mG\}}   \, \Big| \, \mE_{(d_{i_k},f_{i_k})_{1\le k\leq m}} \Big] \notag\\
& =  \prod_{k\in K_3} \ind_{\{f_{i_k}-d_{i_k}>\gep L \}} \bP_{\nu}\big( \nu \notin  \mG\, |\, \nu_0=0\, ;\, f_{i_k}-d_{i_k}\in\nu \, ;\, \nu \cap (f_{i_k}-d_{i_k},L]=\emptyset \big)
\leq (\eta/2)^{|K_3|}\, ,
\end{align}
where we used \eqref{probagood} in the last inequality.

Therefore, plugging \eqref{expansionterm4}-\eqref{expansionterm3} in \eqref{bigexpansion}, we get that
\begin{align*}
\bE_{\nu}\Big[ &\bEst\Big( \prod_{k=1}^m (\eta+ \ind_{\{\sigma_{\mB_{i_k}}\in J_L^c\}})\ \Big|\ A(\nu) \Big)  \Big|  \mE_{(d_{i_k},f_{i_k})_{1\le k\leq m}} \Big] \\
&\leq \sumtwo{K_1,K_2,K_3,K_4 \ \text{disjoint}}{\bigcup_{j=1}^4 K_j = \{1,\ldots,m\}} \eta^{|K_1|} (\eta/2)^{|K_3|} (\eta/2)^{|K_4|}\prod_{k\in K_2} \ind_{\{f_{i_k}-d_{i_k}\leq \gep L\}} =\prod_{k=1}^m (2\eta + \ind_{\{f_{i_k}-d_{i_k}\leq \gep L\}})\,.
\end{align*}
\end{proof}

\section{Proof of Theorem \ref{thm:exist}: examples with unequal critical points}
\label{sec:example}

We use the representation \eqref{zcheck} which yields that for $\zeta\in (0,1)$,
\begin{align}\label{zcheck2}
\bE_\sigma\left[ \left(\check Z_{N,u}^{\sigma}\right)^\zeta \right] &\leq \sum_{m=1}^{N} 
  e^{u\zeta m} \sum_{i_1<\cdots<i_m=N} \prod_{k=1}^m 
  K^*(i_k-i_{k-1})^\zeta   \bE_\sigma\left[ w(\sigma,i_{k-1},i_k)^\zeta \right] \notag\\
&= \sum_{m=1}^{N} e^{u\zeta m} \sum_{i_1<\cdots<i_m=N} \prod_{k=1}^m 
  \frac{ \bEs\left[ \bPt( \sigma_{i_k-i_{k-1}} \in \tau )^\zeta \right] }{ \bEst\left[ |\sigma\cap\tau| \right]^\zeta } \notag\\
&\leq \sum_{m=1}^N \left( e^{u\zeta} \sum_{j=1}^\infty \frac{ \bEs\left[ \bPt( \sigma_j \in \tau )^\zeta \right] }
  { \bEst\left[ |\sigma\cap\tau| \right]^\zeta } \right)^m.
\end{align}
Hence if 
\begin{equation} \label{less1}
  \sum_{j=1}^\infty \frac{ \bEs\left[ \bPt( \sigma_j \in \tau )^\zeta \right] }
  { \bEst\left[ |\sigma\cap\tau| \right]^\zeta } < 1\ ,
\end{equation}
then for sufficiently small $u>0$ we have \eqref{fracmomentfinite}, which establishes inequality of critical points since $u_c^{\a}=0$.

Fix $\ga,\tga,\gp,\tphi$ as in the theorem, and for $\gep\in (0,1/2)$ and even $N\geq 3$ define
\[
  K_{N,\gep}(n) = \begin{cases} \gep &\text{if } n=1,\\  
    0 &\text{if } 2 \leq N < N/2, \\ \frac{2p_{N,\gep}}{N} &\text{if } N/2\leq n \leq N-1, \\ n^{-(1+\ga)}\gp(n) &\text{if } n\geq N,\\
    \end{cases}
\]
where $p_{N,\gep}$ is chosen so that $\sum_{n\geq 1} K_{N,\gep}(n) = 1$. 
Define $\tK_{\tN,\gep}(n)$ and $\tp_{\tN,\gep}$ similarly with $\tga,\tphi,\tN$ in place of $\ga,\gp.N$, and suppose $\tau,\sigma$ have distributions $K_{N,\gep},\tK_{\tN,\gep}$ respectively. Note that $p_{N,\gep},\tp_{\tN,\gep}$ are in $(1/2,1)$ provided $N$ is sufficiently large. 
It suffices to show that \eqref{less1} holds, for some $N,\tN,\gep,\zeta$.

We use superscripts $N,\tN,\gep$ to designate these parameters in the corresponding probability measures and expectations, for example $\bPst^{N,\tN,\gep}(\cdot)$ and $\bEs^{N,\gep}[\cdot]$. We always take $N\geq \tN$.
For the denominator in $\tau$, we observe that
\begin{align}\label{meeting1}
  \bEst^{N,\tN,\gep}\left[ |\sigma\cap\tau| \right] \geq \bPst^{N,\tN,\gep}(\tau_1=\sigma_1=1) = \gep^2\ .
\end{align}
For the numerator in \eqref{less1}, we write
\begin{align*}
\sum_{j=1}^\infty \bEs^{\tN,\gep}\left[ \bPt^{N,\gep}( \sigma_j \in \tau )^\zeta \right]  &= \sum_{j=1}^{+\infty} \sum_{m=1}^{+\infty} \bPt^{N,\gep}(m\in\tau)^\zeta \bPs^{\tN,\gep}(\sigma_j=m) \\ 
&= \sum_{m=1}^{+\infty} \bPt^{N,\gep}(m\in\tau)^\zeta \bPs^{\tN,\gep}(m \in \sigma) \, .
\end{align*}
For $m<  N/2$ we have $\bPt^{\tN,\gep}(m \in \tau) =\gep^m$ (and similarly for $\sigma$ for $m<\tilde N/2$),  so it follows that if $\gep$ is small enough (and $\tilde N$ large)
\begin{align}
\label{numer1}
\sum_{j=1}^\infty \bEs^{\tN,\gep}\left[ \bPt^{N,\gep}( \sigma_j \in \tau )^\zeta \right]  &\le  \sum_{m=1}^{\tilde N/2-1} \gep^{(1+\zeta) m} + \sum_{m=\tilde N/2}^{N/2-1} \gep^{\zeta m} + \sum_{m=N/2}^\infty \bPt^{N,\gep}(m\in\tau)^\zeta \bPs^{\tN,\gep}(m\in\sigma)  \notag \\
& \le 2 \gep^{1+\zeta} + \sum_{m=N/2}^\infty \bPt^{N,\gep}(m\in\tau)^\zeta \bPs^{\tN,\gep}(m\in\sigma) \, .
\end{align}
We now would like to show that if $N$ is large, the last sum  is smaller than the term $2\gep^{1+\zeta}$.

\smallskip
We make two claims. First, 
there exists $c_{17}$ such that provided $N$ is large, 
\begin{equation}\label{mtauclaim1}
  \bPt^{N,\gep}(m\in\tau) \leq \frac{c_{17}}{N} \text{  for all } m 
    \geq \frac N2\ .
\end{equation}
Second, there exists $c_{18}(\gep)$ and slowly varying $\psi_0,\tps_0,\psi,\tps$ such that for all sufficiently large $N,\tN$,
\begin{equation}\label{mtauclaim2t}
  \bPt^{N,\gep}(m\in\tau) \leq c_{18}m^{-(1-\ga)}\psi_0(m) \quad \text{for all } m\geq N^{1/(1-\ga)}\psi(N)
\end{equation}
and
\begin{equation}\label{mtauclaim2s}
  \bPs^{\tN,\gep}(m\in\sigma) \leq c_{18}m^{-(1-\tga)}\tps_0(m)\quad \text{for all } m\geq \tN^{1/(1-\tga)}\tps(\tN)\ .
\end{equation}
Note that by \eqref{Doney} the inequalities in \eqref{mtauclaim2t} and \eqref{mtauclaim2s} hold for sufficiently large $m$, so the aspect to be proved is that the given ranges of $m$ are sufficient in the case of distributions $K_{N,\gep},\tK_{\tN,\gep}$.
We henceforth take $N\geq 2\tN^{1/(1-\ga)}\tps(\tN)$ and $\zeta\in (\tga/(1-\ga),1)$.
Applying \eqref{mtauclaim1} and \eqref{mtauclaim2s} to part of the  sum on the right  in \eqref{numer1} then yields that for some slowly varying $\psi_1$, provided $N$ is large,
\begin{align}\label{lowerm}
  \sum_{N/2\leq m\leq N^{1/(1-\ga)}\psi(N)} \bPt^{N,\gep}(m\in\tau)^\zeta &\bPs^{\tN,\gep}(m\in\sigma) \leq \frac{c_{17}^\zeta}{N^\zeta}
    \sum_{m\leq N^{1/(1-\ga)}\psi(N)} c_{18}m^{-(1-\tga)} \tilde\psi_0(m) \notag\\
  &\leq c_{19} N^{-\zeta+\tga/(1-\ga)} \psi_1(N) \leq \gep^2\ .
\end{align}
Then taking $N$ large and applying \eqref{mtauclaim2t} and \eqref{mtauclaim2s} to the rest of that same sum yields that for some slowly varying $\psi_2$,
\begin{align}\label{upperm}
  \sum_{m>N^{1/(1-\ga)}\psi(N)} \bPt^{N,\gep}(m\in\tau)^\zeta \bPs^{\tN,\gep}(m\in\sigma) &\leq c_{19}N^{-\zeta+\tga/(1-\ga)}\psi_2(N) < \gep^2\ .
\end{align}
Provided $\gep$ is small, \eqref{meeting1}, \eqref{numer1}, \eqref{lowerm} and \eqref{upperm} yield
\begin{equation}\label{combined}
  \sum_{j=1}^\infty 
    \frac{ \bEs^{\tN,\gep}\left[ \bPt^{N,\gep}( \sigma_j \in \tau )^\zeta \right] }{ \bEst^{N,\tN,\gep}\left[ |\sigma\cap\tau| \right]^\zeta }
    \leq \frac{2\gep^{1+\zeta} + 2\gep^2}{(2\gep^2)^\zeta} < 1\ ,
\end{equation}
proving \eqref{less1} and completing the proof of Theorem \ref{thm:exist}.

\medskip
We now prove \eqref{mtauclaim1}. 
Fix $m\geq N/2$ and let 
\[
  F_m = \max \tau\cap [0,m-N], \quad G_m = \min \tau \cap (m-N,m],
  \]
with $G_m=\infty$ when the corresponding intersection is empty, and $F_m$ undefined, $G_m=0$ when $m<N$.  
From the definition of $K_{N,\gep}$ we have for $N/2\leq k<N$
\begin{align}\label{firsthalf}
  \bPt^{N,\gep}(m\in\tau \mid G_m=m-k) &= \gep^k + \sum_{j=0}^{k-N/2} (j+1) \gep^j \frac{2p_{N,\gep}}{N} \leq \gep^{N/2} + \frac 3N \leq \frac4N\ ,
\end{align}
and for $0\leq k <N/2$
\begin{equation}\label{secondhalf}
  \bPt^{N,\gep}(m\in\tau \mid G_m=m-k) = \gep^k.
\end{equation}
Now \eqref{mtauclaim1} follows from \eqref{firsthalf} when $N/2\leq m<N$, so we consider $m\geq N$. It is enough to show that for all $m\geq N$ and $N\leq j\leq m$,
\begin{equation} \label{jumppoint}
  \bPt^{N,\gep} \big( m\in \tau \mid F_m = m-j \big) \leq \frac{c_{17}}{N}\ .
\end{equation}
For any $N\leq j\leq m$, we have that 
\begin{align}\label{givenFm}
\bPt^{N,\gep} \big( m\in \tau \big|  F_m = m-j \big) &=\sum_{k=0}^{N-1} \bPt^{N,\gep}\big( m \in \tau \big| G_m = m-k) \bPt^{N,\gep}\big( G_m = m-k \big| F_m = m-j\big) \notag \\
&\le  \sum_{k=0}^{N/2 -1} \gep^k \bPt^{N,\gep}\big(  G_m = m-k \big| F_m = m-j \big)  + \frac 4N \, ,
\end{align}
where we used \eqref{secondhalf} for the terms $k<N/2$, and \eqref{firsthalf} for the terms $N/2\le k<N$.

We now show that there is a constant $c_{18}$ such that uniformly for $k<N/2,j\ge N$,
\begin{equation}\label{FmGm}
  \bPt^{N,\gep}\big(  G_m = m-k \mid F_m = m-j \big)\le c_{18}/N\, .
\end{equation}
The probability on the left is equal to $\bPt(\tau_1=j-k)/\bPt(\tau_1>j-N)$.  Therefore
for each $j\geq 3N/2$, the probabilities $\bPt(G_m = m-k | F_m=m-j), k<N/2$ are all of the same order so are bounded by $c_{20}/N$ for some $c_{20}$. 

For the other $j$ values, that is, $N\leq j<3N/2$, the probabilities $\bPt(G_m = m-k \mid F_m=m-j), j-N<k\leq j-N/2$, are all equal so have common value at most $2/N$.  For $k\leq j-N$ the probabilities $\bPt(G_m = m-k | F_m=m-j)$ are uniformly smaller than this common value provided $N$ is large, because $n^{-(1+\ga)}\gp(n) \leq 2p_{N,\gep}/N$ for all $n\geq N$. 
Therefore the bound of $2/N$ applies for all $k\leq j-N/2$, which includes all $k<N/2$.

This proves \eqref{FmGm}, which with \eqref{givenFm} proves \eqref{jumppoint} and thus \eqref{mtauclaim1}.

\medskip
It remains to prove \eqref{mtauclaim2t}; \eqref{mtauclaim2s} is equivalent.  Let $\delta\in (0,1)$ to be specified, and define 
\[
   \Delta_j=\tau_j-\tau_{j-1}, \quad M_n=\max_{j\leq n} \Delta_j, \quad \ell_m = \frac{\delta m}{\log m}, 
\]
and
\[
   \beta=\frac{8(1-\ga)}{3\ga},\quad n_0(m) = \max\{n: a_n\leq m(\log m)^\beta\}\ ,
\]
where $a_n$ is as in \eqref{defanbn}.  If $\tau_n=m$ for some $n>n_0(m)$, it means that $\tau$ is ``very compressed'' in the sense that $\tau_n \ll a_n$.
It is easily checked that there exist slowly varying $\psi_3,\psi_4$ such that
\begin{equation}\label{n0size}
  n_0(m) \sim m^\ga \psi_3(m) \quad\text{as } m\to\infty
  \end{equation}
and
\begin{equation}\label{mediums}
  n \geq k_N := N^{\ga/(1-\ga)}\psi_4(N) \quad\implies\quad 2nN \leq \frac{a_n}{(\log a_n)^\beta}\ .
\end{equation}
It follows from \eqref{n0size} that we can choose the slowly varying function $\psi$ such that for large $N$,
\begin{equation}\label{n0vskN}
  n_0\Big(N^{1/(1-\ga)}\psi(N)\Big) \geq k_N\ .
\end{equation}
We choose $m \ge N^{1/(1-\ga)} \psi(N)$. Note that $\ell_m \ge N$ provided that $N$ is large.

We first handle the ``very compressed'' case, that is $n>n_0(m)$: using Lemma \ref{smallk}, provided $N$ is large,
\begin{align}\label{compress}
  \bPt^{N,\gep}(\tau_n = m &\text{ for some } n>n_0(m)) \leq \bPt^{N,\gep}(\tau_{n_0(m)}<m) \notag\\
  &\leq \bPt^{N,\gep}\left(\tau_{n_0(m)} < \frac{2a_{n_0(m)}}{(\log n_0(m))^\beta} \right) \leq \exp\left( -c_{26}(\log m)^2  \right) \leq m^{-(1-\ga)}\ .
\end{align}

Next, we consider the ``not very compressed'' case,  with a large gap, that is, $n\leq n_0(m)$ and $M_n\geq \ell_m$.
Let $m\geq N^{1/(1-\ga)}\psi(N),1\leq n\leq n_0(m)$ and $1\leq j<m$; note $\ell_m\geq N$ provided $N$ is large.  
For vectors $(t_1,\dots,t_n)\in \ZZ^n$, consider the mapping which adds $j$ to the first coordinate $t_i$ satsifying $t_i>\ell_m$, when one exists.  
This map is one-to-one, and provided $m$ is large, it decreases the corresponding probability $\bPt^{N,\gep}(\Delta_1=t_1,\dots,\Delta_n=t_n)$ by at most a factor of $(\delta/\log m)^{2(1+\ga)}$.  It follows that
\begin{align}\label{addj}
  \bPt^{N,\gep}(\tau_n = m+j, M_n\geq \ell_m) \geq \left( \frac{\delta}{\log m} \right)^{2(1+\ga)}\bPt^{N,\gep}(\tau_n = m, M_n> \ell_m)\ .
\end{align}
Summing over $1\leq j<m$ we obtain
\begin{equation}\label{sumj}
  \bPt^{N,\gep}(\tau_n = m, M_n> \ell_m) \leq \frac 1m \left( \frac{\log m}{\delta} \right)^{2(1+\ga)} \bPt^{N,\gep}(\tau_n \in [m,2m) )\ .
\end{equation}
We let
\[
   \tau_n^+ = \sum_{i=1}^n \Delta_i \ind_{\Delta_i\geq N}\ ,
\]
so $\tau_n-\tau_n^+\leq nN$.  By \eqref{n0vskN} we have $k_N\leq n_0(m)$.  Provided $N$ (and hence $m$) is large, when $k_N\leq n\leq n_0(m)$ using \eqref{mediums} we get $2nN \leq a_n/(\log a_n)^\beta \leq m$. Hence $2nN\leq m$ also for $n\le k_N$. Therefore, for all $n \le n_0(m)$,
\[
  \bPt^{N,\gep}(\tau_n \in [m,2m) ) \leq \bPt^{N,\gep}\left( \tau_n^+ \geq \frac m2 \right) \leq \bPt\left( \htau_n \geq \frac m2 \right) \le  c_{24} nm^{-\ga}\gp(m) ,
\]
where we let $\htau$ be a renewal process with distribution $\bPt(\htau_1=k) = k^{-(1+\ga)}\gp(k)$ for all $k\geq 1$ and then used Lemma~\ref{lem:uniform}. Summing over $n$ we obtain using \eqref{sumj} that for some slowly varying $\psi_5$,
\begin{align}\label{nocompress}
  \sum_{n=1}^{n_0(m)} \bPt^{N,\gep}(\tau_n = m, M_n> \ell_m)
  &  \leq c_{25} \left( \frac{\log m}{\delta} \right)^{2(1+\ga)}
    n_0(m)^2 m^{-1-\ga}\gp(m) \notag \\
    &\leq m^{-(1-\ga)}\psi_5(m)\ .
\end{align}

Finally we consider the case with no large gap: $m\geq N^{1/(1-\ga)}\psi(N)$ and $M_n<\ell_m$.  Let $\Delta_i^{\rm tr}, i\geq 1$, be iid with distribution $\bPt(\Delta_1 \in \cdot \mid \Delta_1\leq\ell_m)$, and let $\tau_n^{\rm tr} = \sum_{i=1}^n \Delta_i^{\rm tr}$.  Then
\begin{align}\label{trunc1}
  \bPt^{N,\gep}(\tau_n = m, M_n< \ell_m) = (1 - \bPt(\tau_1>\ell_m))^n \bPt(\tau_n^{\rm tr} =m)\ .
\end{align}
Let
\[
  n_1(m) = \frac{1-\ga}{12} m\ell_m^{-(1-\ga)}\gp(\ell_m)^{-1}\, .
\]
For $n\geq n_1(m)$ we bound the first factor on the right side of \eqref{trunc1}.  As $N$ (and hence $m$) grows we have
\begin{align}\label{largen}
  n_1(m)\bPt \big(\tau_1>\ell_m \big) \sim \frac{1}{2\ga}m \ell_m^{-1} = \frac{1-\ga}{12\ga\delta}\log m,
\end{align}
so provided $\gd$ is small, \eqref{trunc1} yields
\begin{align}\label{largen2}
  \sum_{n=n_1(m)}^\infty \bPt^{N,\gep}(\tau_n = m, M_n\leq \ell_m) &\leq \sum_{n=n_1(m)}^\infty (1 - \bPt^{N,\gep}(\tau_1>\ell_m))^n \notag\\
  &\leq \frac{1}{\bPt^{N,\gep}(\tau_1>\ell_m)} \exp\left( -n_1(m)\bPt^{N,\gep} \big(\tau_1>\ell_m \big) \right) \notag \\
  &\le m^{-(1-\ga)} \, .
\end{align}
For $n< n_1(m)$ we bound the second factor on the right side of \eqref{trunc1}.  We have
\begin{align}\label{expbd}
  \bPt^{N,\gep} \big(\tau_n^{\rm tr} \geq m \text{ for some } n<n_1(m) \big) & \leq \bPt^{N,\gep}(\tau_{n_1(m)}^{\rm tr} \geq m) \notag
\\
&    \leq e^{-m/\ell_m} \bEt^{N,\gep}\big[ e^{\tau_1^{\rm tr}/\ell_m} \big] ^{n_1(m)}\ .
\end{align}
It is easy to see that, in addition to \eqref{n0vskN}, we can choose the slowly varying $\psi$ so that 
$(1-\ga) N \le \ell_m^{1-\ga} \gp(\ell_m) $ whenever $N$ is large and $m\ge N^{1/(1-\ga)} \psi(N)$.  
Using  that $e^x\leq 1+2x$ for $x \in [0,1]$, we get that $ \bEt^{N,\gep}\big[ e^{\tau_1^{\rm tr}/\ell_m} \big] \le 1 + 2\ell_m^{-1} \bEt^{N,\gep}  \big[ \tau_1^{\rm tr} \big]$, 
and we have
\begin{align*}
\bEt^{N,\gep}  \big[ \tau_1^{\rm tr} \big] &\le N + \frac{1}{\bPt^{N,\gep}(\tau_1^{\rm tr} \le \ell_m)}  \sum_{j=N}^{\ell_m} j^{-\ga} \gp(j) \\
& \le N +  \frac{2}{1-\ga} \ell_m^{1-\ga} \gp(\ell_m) \leq \frac{3}{1-\ga} \ell_m^{1-\ga} \gp(\ell_m)\, ,
\end{align*}
so that 
\[\bEt^{N,\gep}\big[ e^{\tau_1^{\rm tr}/\ell_m} \big] \le 1+ \frac{6}{1-\ga} \ell_m^{-\ga}   \gp(\ell_m) \le \exp\left( \frac{6}{1-\ga} \ell_m^{-\ga} \gp(\ell_m)\right)\, .
\]
Therefore by the definition of $n_1(m)$,
\begin{align}\label{expbd2}
  e^{-m/\ell_m} \bEt^{N,\gep}\Big[ e^{\tau_1^{\rm tr}/\ell_m} \Big]^{n_1(m)} \leq e^{-m/2\ell_m} = m^{-1/\gd} \leq m^{-(1-\ga)}\ .
\end{align}
Combining \eqref{compress}, \eqref{nocompress}, \eqref{largen2}, and \eqref{expbd}-\eqref{expbd2},  we obtain for $m\geq N^{1/(1-\ga)}\psi(N)$,
\[
  \bPt^{N,\gep}(m\in\tau) \leq m^{-(1-\ga)}(3+\psi_5(m)),
\]
proving \eqref{mtauclaim2t}.

\section{Variations of the model: proofs of Proposition \ref{prop:compareFtilde} and Proposition \ref{prop:compareFbar}}
\label{sec:othermodels}

We start with the proof of Proposition \ref{prop:compareFbar}, since we adapt it for the proof of Proposition~\ref{prop:compareFtilde}, in which we need to control additionally $|\tau|_{\sigma_N}$.

\subsection{Proof of Proposition \ref{prop:compareFbar}(i)}
\label{sec:compareFbar}

We fix $\gb>\gb_c$, and find a lower bound on $\bar Z_{N,\gb}^{\sigma} =\bE_{\tau}\big[\exp(\gb |\tau\cap\sigma|_{\tau_N}) \big]$ by restricting $\tau$ to follow a particular strategy.

As in Section \ref{sec:proof1}, we divide the system into blocks $\mB_i:=\{0,\sigma_{(i-1)L +1} - \sigma_{(i-1)L} ,\ldots, \sigma_{i L} - \sigma_{(i-1)L}  \}$, with $L$ to be specified. For $b_N$ from \eqref{defanbn}, there exists $v_0>0$ such that $\bPs(\sigma_N > v_0 b_N) \leq 1/4$ for $N$ large.
Define the event of being \emph{good} by
\begin{equation}
\label{eq:defAtilde}
G_L:= \left\{ (\sigma_1,\ldots,\sigma_L): \frac{1}{L} \log Z_{L,\gb}^{\sigma} \geq \frac12 \mF(\gb) \quad \text{and} \quad \sigma_L\leq v_0 b_L \right\}.
\end{equation}
Since $\gb>\gb_c$, there exists $L_0$ such that, for $L\geq L_0$
\begin{equation}
\label{eq:cond1L}
\bP_{\sigma}\Big( \frac{1}{L} \log Z_{L,\gb}^{\sigma} \geq \frac12 \mF(\gb) \Big)\geq \frac 34\, ,
\end{equation}
so that
\begin{equation}
\label{eq:condL}
\bP_{\sigma}(G_L)\geq \frac 12\, .
\end{equation}

We now set $\mI=\mI(\sigma)=\{i: \mB_i\in G_L\} =\{i_1,i_2,\ldots\}$, the set of indices of the \emph{good} blocks, and set $i_0=0$. There can be at most $v_0 b_L $ $\tau$-renewals per block, so restricting trajectories to visit only blocks with index in $\mI$, we get that for all $m\in\NN$
\begin{equation}
\bar Z_{m v_0 b_{L},\gb}^{\sigma} \geq \prod_{k=1}^m \bP_{\tau}(\tau_1 = \sigma_{(i_k-1)L} - \sigma_{i_{k-1}L}) e^{\mF(\gb) L/2} \, ,
\end{equation} 
with the convention that $\bPt(\tau_1=0)$ means 1.
Then, letting $m\to\infty$, we get
\begin{equation}
\liminf_{N\to\infty} \frac1N \log \bar Z_{N,\gb}^{\sigma} \geq \frac{1}{ v_0 b_{L}}\, \frac12 \mF(\gb) L+ \frac{1}{v_0 b_{L}} \bEs[ \log \bPt(\tau_1 = \sigma_{(i_1-1) L}) ]\, \quad \bPs-a.s.
\end{equation}

Let us estimate the last term. Thanks to \eqref{eq:alphas}, we have that
\[\log \bPt(\tau_1 = \sigma_{(i_1-1) L}) \geq - (2+\ga) \ind_{\{i_1>1\}} \log \sigma_{(i_1-1) L},\]
provided $L$ is large. Then, Lemma~\ref{logsigma} applies: for $L$ large, since $\bPs(G_L)\geq 1/2$ we have that 
\[
\bEs[\ind_{\{i_1>1\}} \log \sigma_{(i_1-1) L})] 
\leq \frac{2}{\tga \wedge 1} \left(\log L + \log \frac{1}{\bPs(G_L)} \right) \leq \frac{4}{\tga \wedge 1} \log L .
\]
Hence for some $L_0$, for $L\geq L_0$,
\begin{equation}
\label{lowboundZbar}
\liminf_{N\to\infty} \frac1N \log \bar Z_{N,\gb}^{\sigma} \geq \frac{1}{v_0 b_{L}} \Big( \frac12 \mF(\gb) L -  \frac{4(2+\ga)}{\tga \wedge 1}  \log L\Big) \, ,
\end{equation}
and provided that $L$ is large enough, we have that
$\bar \mF(\gb) >0 $ for any $\gb>\gb_c$, meaning $\gb_c \geq \bar\gb_c$.

%
%
%

\subsection{Proof of Proposition \ref{prop:compareFbar}(ii)}
\label{sec:compareFbar2}
Observe that the annealed systems for the original and elastic polymers have the same critical point, by Remark \ref{rem:constrainedvsfree}.  Therefore $\beta_c^{\a}\leq \bar\beta_c$.   When $\ga=0,\tga\geq 1$, it then follows from Theorem \ref{thm:main} that $\beta_c \leq \bar\beta_c$, so equality holds.

Hence it remains to prove $\beta_c \leq \bar\beta_c$ assuming $\ga>0$, by showing that pinning in the elastic polymer (length-$\tau_N$ system) implies pinning in the original polymer (length-$\sigma_N$ system.)


In the recurrent case, $\beta_c=0$ so there is nothing to prove.  So we assume transience, which here implies $\ga,\tga\in(0,1)$.  Let $v_N = e^{2\tilde\alpha\beta N/\alpha}$ and $\tau_{[0,N]} = \{\tau_1,\dots,\tau_N\}$. Define
\[
  \quad X_{N,j} = X_{N,j}(\sigma) := 
   \bEt\left( e^{\beta |\sigma\cap\tau_{[0,N]}|_{\sigma_j}}\ind_{\{\sigma_j\in\tau_{[0,N]}\}} \right),
\]
and recall
\begin{align}\label{decomp}
  \bar Z_{N,\gb}^{\sigma} &= \bE_{\tau}\big[\exp(\gb |\tau\cap\sigma|_{\tau_N}) \big] \notag\\
  &= \sum_{j=1}^\infty X_{N,j}(\sigma)\bPt\bigg( \tau_{[0,N]}\cap \{\sigma_{j+1},\sigma_{j+2},\dots\} = \phi\ \big|\ \sigma_j\in \tau_{[0,N]} \bigg).
\end{align}
Let $\bar\mF(\beta)=\liminf_{N\to\infty} \frac1N \log \bar Z_{N,\gb}^{\sigma}$ (which is a tail random variable of $\{\sigma_i-\sigma_{i-1}, i\geq 1\}$, so is nonrandom, up to a null set) and assume that $\bar\mF(\beta)>0$. 
We define the truncated sums
\begin{align}\label{trunc}
  \bar Z_{N,\gb}^{\sigma,T(v)} &= \sum_{j=1}^{v} 
    X_{N,j}(\sigma) \bPt\bigg( \tau_{[0,N]}\cap \{\sigma_{j+1},\sigma_{j+2},\dots\} = \phi\ \big|\ \sigma_j\in \tau_{[0,N]} \bigg),\notag\\
  \bar Z_{N,\gb,0}^{\sigma,T(v)} &= \sum_{j=1}^{v} X_{N,j}(\sigma).
\end{align}
Then it is not hard to show that $\bPs$-a.s., for large $N$,
\[
   \bar Z_{N,\gb}^{\sigma}  -  \bar Z_{N,\gb}^{\sigma,T(v_N-1)} \leq e^{\beta N}\bPt(\tau_N>\sigma_{v_N-1}) \leq 1,
\]
so 
\begin{equation}\label{Fbarineq}
  \bar\mF(\beta)=\liminf_{N\to\infty} \frac1N \log \bar Z_{N,\gb}^{\sigma,T(v_N-1)} \leq \liminf_{N\to\infty} \frac1N \log \bar Z_{N,\gb,0}^{\sigma,T(v_N-1)}.
\end{equation}

Below, we show the following
\begin{lemma}
\label{lem:Ztrunc}
For every $\gb>\bar \gb_c$, there exists some $N$ large enough, so that 
\[
\liminf_{m\to\infty} \frac1m \log \bar Z_{mN,\gb,0}^{\sigma,T(m v_N)}
\geq \frac{1}{12} N \bar \mF(\gb)\, .
\]
\end{lemma}

It is essential here that the truncation in the partition function be at $m v_N$, not at the much larger value $v_{mN}$, as we want the allowed length of trajectories to grow essentially only linearly in $m$.  But we need to know that with this length restriction, the log of the partition function is still of order $mN$.

With this lemma in hand, we easily have that
\begin{align*}
\bar Z_{mN,\gb,0}^{\sigma,T(mv_N)} \leq \sum_{j=1}^{m(v_N+1)} \bEt \left[ e^{\gb |\sigma \cap \tau |_{\sigma_{j}}}  \right] 
\leq mv_N  Z_{m v_N,\gb}^{\sigma, {\rm free}}\, .
\end{align*}
Then thanks to Remark \ref{rem:constrainedvsfree} (and because $\tfrac1m \log (m v_N) \to 0$), we have
\begin{align*}
\mF(\gb) &= \lim_{m\to\infty} \frac{1}{m v_N} \log Z_{m v_N,\gb}^{\sigma, {\rm free}}\\
& \geq \liminf_{m\to +\infty} \frac{1}{m v_N}  \log \bar Z_{mN,\gb,0}^{\sigma,T(mv_N)}  \geq \frac{N}{12  v_N} \bar \mF(\gb),
\end{align*}
which gives that  $\mF(\gb)>0$ for any $\gb>\bar \gb_c$, that is $\gb_c \leq \bar \gb_c$. This concludes the proof of Proposition \ref{prop:compareFbar}(i).

\smallskip
\begin{proof}[Proof of Lemma \ref{lem:Ztrunc}]
Let $U=(U_i)_{i\geq 1}$ be an i.i.d.~sequence of variables independent of $\sigma$, and let
$(\mathcal{G}_n)_{n\geq 1}$ be the $\sigma$-field generated by $\{\sigma_1, U_1,\ldots,\sigma_n,U_n\}$.
We will define below a stopping time $T_1=T_1(\sigma,U)$ for the filtration $(\mathcal{G}_n)_{n\geq 1}$ satisfying $T_1\leq v_N$, and denote $\hat Q_{N,j}(\sigma) :=\bP_T(T_1=j \mid \sigma)$.  The stopping time property means that $\hat Q_{N,j}(\sigma)$ depends only on $\sigma_1,\dots,\sigma_j$.

Write $\sigma^{(j)}$ for the shifted sequence $(\sigma_k-\sigma_j,U_k), k\geq j$.
We then define iteratively the stopping times $(J_i)_{i\geq 1}$ by
\[J_i = J_1(  \sigma^{(J_{i-1})} ). \]

\smallskip
For any fixed $j_1<\cdots < j_m$ with $|j_i -j_{i-1}| \leq v_N$ (representing possible values of the $J_i$), we can decompose a product of $X$ variables as follows:
\begin{align*}
\prod_{i=1}^m X_{N,j_i-j_{i-1}}(\sigma^{(j_{i-1})}) = \sumtwo{0=\ell_0<\ell_1<\cdots <\ell_n}{|\ell_i -\ell_{i-1} | \leq N \ \forall \, i\leq m}\bE \Big[ e^{\gb |\sigma\cap \tau_{[0,mN]}| } \prod_{i=1}^m \ind_{\{\tau_{\ell_i} = \sigma_{j_i}\}} \Big].
\end{align*}
Hence, summing over all such $j_1<\cdots < j_m$, we get the bound
\begin{align*}
 \sumtwo{0=j_0<j_1<\cdots<j_m\leq mv_N}{|j_i-j_{i-1}|\leq v_N} \prod_{i=1}^m X_{N,j_i-j_{i-1}}(\sigma^{(j_{i-1})}) \leq \sumtwo{0=\ell_0<\ell_1<\cdots <\ell_n}{|\ell_i -\ell_{i-1} | \leq N \ \forall \, i\leq m} \bar Z_{mN,\gb,0}^{\sigma, T(mv_N)} \leq N^m \bar Z_{mN,\gb,0}^{\sigma, T(mv_N)} .
\end{align*}
which can be rewritten
\begin{align}\label{ZvsEJ}
\bar Z_{mN,\gb,0}^{\sigma, T(mv_N)}  &\geq N^{-m} \sumtwo{0=j_0<j_1<\cdots<j_m}{|j_i-j_{i-1}|\leq v_N} \prod_{i=1}^m \hat Q_{N,j_i-j_{i-1}}(\sigma^{(j_{i-1})}) \exp\Big( \sum_{i=1}^m \log \frac{X_{N,j_i-j_{i-1}}(\sigma^{(j_{i-1})})}{\hat Q_{N,j_i-j_{i-1}}(\sigma^{(j_{i-1})})}  \Big)\notag\\
&= N^{-m} \bE_J\bigg[  \exp\Big( \sum_{i=1}^m \log \frac{X_{N,J_i-J_{i-1}}(\sigma^{(J_{i-1})})}{\hat Q_{N,J_i-J_{i-1}}(\sigma^{(J_{i-1})})}  \Big) \ \Big | \ \sigma \bigg] .
\end{align}
Here in a mild abuse of notation we write $\bE_J[\cdot\mid\sigma]$ for the expectation over $U$, and we will write $\bE_{\sigma,J}$ for the expectation over $(\sigma,U)$.  Since $J_1$ is a stopping time, the summands on the right side of \eqref{ZvsEJ} are i.i.d.~functions of $(\sigma,U)$.
By \eqref{ZvsEJ} and Jensen's inequality, we have 
\begin{equation}
\log \bar Z_{mN,\gb,0}^{\sigma, T(mv_N)} \geq -m \log N + \sum_{i=1}^m \bE_J \Big[ \log  \frac{X_{N,J_i-J_{i-1}}(\sigma^{(J_{i-1})})}{\hat Q_{N,J_i-J_{i-1}}(\sigma^{(J_{i-1})})} \ \Big | \ \sigma \Big] .
\end{equation}

Since $X_{N,J_1}(\sigma) \geq \bPt(\tau_1=\sigma_{J_1})$ and $J_1\leq v_N$, we have for some $c_i$
\begin{equation}\label{sigmavN}
  \log \frac{X_{N,J_1}(\sigma)}{\hat Q_{N,J_1}(\sigma)} \geq \log \bPt(\tau_1=\sigma_{J_1}) \geq - c_{36} \log \sigma_{v_N}
  \end{equation}
and hence, writing $x_{-}$ for the negative part of $x$ and using \eqref{sigman},
\begin{equation}\label{sigmavN2}
  \bE_{\sigma,J}\left[ \left( \log \frac{X_{N,J_1}(\sigma)}{\hat Q_{N,J_1}(\sigma)} \right)_{-} \right]  \leq c_{36} \bEs[\log \sigma_{v_N}] \leq c_{37} v_N <\infty.
\end{equation}
It then follows from \eqref{ZvsEJ} that 
\begin{align}
\label{LLN}
\liminf_{m\to\infty} \frac1m \log  \bar Z_{mN, \gb,0}^{\sigma, T(mv_N)} \geq \bE_{\sigma,J} \Big[ \log \frac{X_{N,J_1}(\sigma)}{\hat Q_{N,J_1}(\sigma)} \Big] - \log N, \quad \bP_{\sigma,J}-\text{a.s.}
\end{align}
We will show that there is a choice of $N,J_1$, and $\hat Q_{N,J_1}(\sigma)$ satisfying
\begin{equation}\label{left}
   \bE_{\sigma,J} \Big[ \log  \frac{X_{N,J_1}(\sigma)}{\hat Q_{N,J_i}(\sigma)}  \Big] \geq \frac{1}{6}\bar \mF(\gb) N, \qquad \log N \leq \frac{1}{12} \bar\mF(\gb) N.
   \end{equation}
The lemma will then follow from \eqref{LLN} and \eqref{left}.
   
Fix $K$ to be specified and define
\[ R_N = R_N(\sigma) := 
    \min\left\{v: \bar Z_{N,\gb,0}^{\sigma,T(v)} = \sum_{j=1}^v X_{N,j}(\sigma) \geq  K e^{N \bar \mF(\gb)/2} \right\} \wedge v_N. 
    \]
From \eqref{Fbarineq}, we then have
\begin{equation}\label{normal}
  \bPs(R_N=v_N) = \bPs \Big(\bar Z_{N,\gb,0}^{\sigma,T(v_N-1)}<
  K \, e^{ N \bar \mF(\gb)/2}\Big) \to 0 \text{ as } N \to\infty.
\end{equation}
We define the marking probabilities 
\[
  Q_{N,j} = Q_{N,j}(\sigma) := 1 - \exp\left( -    X_{N,j} \,  e^{ -N \bar \mF(\gb)/2}\right),
\]
which only depends on $\sigma_1,\dots,\sigma_j$, and
\[
 \hat Q_{N,j}(\sigma) = \begin{cases}   \prod_{k=1}^{j-1} (1-Q_{N,k})  \times Q_{N,j} &\text{if } 1\leq j <  R_N(\sigma), \\
 \prod_{k=1}^{R_N} (1- Q_{N,k}) &\text{if } j=R_N, \\
   0 & \text{if } j>R_N ,\end{cases}
\]
so that $ \sum_{j=1}^{v_N} \hat Q_{N,j}(\sigma) = 1.$  Then set
\[
  J_1 = \min\{j\geq 1: U_j \leq Q_{N,j} \} \wedge R_N.
\]
We may view this as follows: for each $j<R_N$ we mark $\sigma_j$ with probability $Q_{N,j}$, independently, and we mark $\sigma_{R_N}$ with probability 1;  $\sigma_{J_1}$ is the first $\sigma_j$ to be marked. As a result we have $\bP_{J}(J_1 =j \mid \sigma) = \hat Q_{N,j}(\sigma)$, and $J_1\leq v_N$. Note that this weights $J_1$ toward values $j$ for which $X_{N,j}$ is large, which, heuristically, occurs when $\sigma_j$ follows a favorable stretch of the disorder $\sigma$.

\medskip
We now consider \eqref{left}, and write 
\begin{equation}
\label{EJ}
\bE_J \Big[ \log \frac{X_{N,J_1}}{\hat Q_{N,J_1}}\ \big | \ \sigma\Big]  = \sum_{j=1}^{R_N-1} \hat Q_{N,j} \log \frac{X_{N,j}}{\hat Q_{N,J_1}} +  \hat Q_{N,R_N} \log \frac{X_{N,R_{N}}}{\hat Q_{N,R_{N}}} .
\end{equation}

\smallskip
For the sum in \eqref{EJ}, using $1-e^{-x}\leq x$ we get $X_{N,j} \geq e^{N \bar \mF(\gb)/2} Q_{N,j} \geq e^{N \bar \mF(\gb)/2}\hat Q_{N,j}$ for all $j< R_N$, and therefore
\[
  \sum_{j=1}^{R_N-1} \hat Q_{N,j} \log \frac{X_{N,j}}{\hat Q_{N,J_1}} \geq  \frac12 \bP_J(J_1< R_N  \mid \sigma)\  \bar \mF(\gb) N .
  \]
For the last term in \eqref{EJ}, from \eqref{sigmavN} we have
\[ 
 \log X_{N,R_{N}} \geq  \log \bP(\tau_1 = \sigma_{R_{N}}) \geq -c_{36} \log  \sigma_{v_N}
 \]
and therefore
\[
  \hat Q_{N,R_N} \log \frac{X_{N,R_{N}}}{\hat Q_{N,R_{N}}}  \geq - c_{36} \ \hat Q_{N,R_{N}} \log \sigma_{v_N}\, .
  \]
Combining these bounds and averaging over $\sigma$, we get from \eqref{EJ} that
\begin{equation}
\label{lasttwoterms}
\bE_{\sigma,J} \Big[ \log \frac{X_{N,J_1}}{\hat Q_{N,J_1}} \Big]  \geq \frac12 \bP_{\sigma,J} (J_1< R_N)\ \bar\mF(\gb) N  - c_{36}\,  \bE_{\sigma,J}\big[  \bP_{J}(J_1=R_{N} \mid \sigma)\log \sigma_{v_N} \big]\, .
\end{equation}

\medskip
For the first term on the right side of \eqref{lasttwoterms}, when $R_N(\sigma)<  v_N$ we have from the definition of $R_N$ that
\begin{equation}\label{highprob}
\bP_{J} (J_1< R_N \mid \sigma)  = 1- \prod_{k=1}^{R_{N}} (1-Q_{N,k})  = 1- \exp\Big( - e^{-\bar\mF(\gb) N/2} \sum_{j=1}^{R_{N} } X_{N,j}\big) \geq 1- e^{-K}.
\end{equation}
We therefore get using \eqref{normal} that
\begin{equation}
\label{term1}
\bP_{\sigma,J} (J_1< R_N) \geq \bP_{\sigma} (R_N < v_N) (1-e^{-K}) \geq \frac23,
\end{equation}
provided that $K$ and then $N$ are chosen large enough.

\smallskip
For the last term in \eqref{lasttwoterms}, we have from \eqref{highprob} and \eqref{sigmavN2} that
\begin{equation}\label{smallR}
\bE_{\sigma,J}\big[ \ind_{\{R_N< v_N\}}\bP_{J}(J_1=R_{N} \mid \sigma)\log \sigma_{v_N} \big] \leq e^{-K} \bEs[ \log \sigma_{v_N}] \leq c_{38} e^{-K} \log v_N .
\end{equation}
In addition, it is routine that there exists a constant $c_{39}$ such that
\[
  \bEs\left[\left( \frac{\log \sigma_n}{\log n} \right)^2\right] \leq c_{39} \quad\text{for all } n\geq 2.
\] 
Hence using \eqref{normal} and the Cauchy-Schwarz inequality
\begin{align}\label{bigR}
\bEs\big[ \ind_{\{R_N=v_N\}}\log \sigma_{v_N} \big] 
& \leq (\log v_N) \bPs(R_N=v_N)^{1/2} \bEs\left[\left( \frac{\log \sigma_{v_N}}{\log v_N} \right)^2\right] \notag\\
&= o(\log v_N) \quad \text{as } N\to\infty.
\end{align}

Combining \eqref{smallR} and \eqref{bigR} we get that if we take $K$ and then $N$ is large enough, 
\begin{equation}
\label{term2}
c_{36}\bE_{\sigma,J}\big[ \bP_{J}(J_1=R_N \mid \sigma)\log \sigma_{v_N} \big] \leq 2c_{36}c_{38} e^{-K} \log v_N \leq  \frac16 \bar\mF(\gb) N.
\end{equation}

Plugging \eqref{term1} and \eqref{term2} into \eqref{lasttwoterms}, we finally get \eqref{left}.

\end{proof}


\subsection{Proof of Proposition \ref{prop:compareFtilde}}

As noted in Section \ref{variants}, we only need prove that $\hat \gb_c \leq \gb_c$, so let us fix $\gb>\gb_c$, and show that $\hat\mF(\gb)>0$.
We write $Z_{N,\gb}^{\sigma,f}(H)$ for $\bEt[e^{\gb|\tau\cap\sigma|_{\sigma_N}}\, \ind_H  ]$, for an event $H$. (Note this partition function may involve trajectories not tied down, that is, with $\sigma_N\notin\tau$.) 
A first observation is that for fixed $q\geq 1$,
\begin{equation}\label{NqvsN}
\log \hZ_{N+q,\gb}^{\sigma} \geq  Z_{N,\gb}^{\sigma,f}(\tau_N\leq\sigma_N ) \min_{d_\sigma m\leq k\leq \sigma_{N+q}}  \bPt(\tau_q = k). 
\end{equation}
By \eqref{nozero}, if we fix $q$ large enough and then take $N$ large, the minimum here is achieved by $k\geq \sigma_N/2$, so by \eqref{locallargedev}, it is at least $\tfrac12 q\bPt(\tau_1=\sigma_{N+q}) \geq \sigma_{N+q}^{-(2+\ga)}$.
Using \eqref{sigman} we therefore get for large $N$
\begin{align}
\label{tildettilde}
\bEs \log \hZ_{N+q,\gb}^{\sigma} &\geq  \bEs \log Z_{N,\gb}^{\sigma,f}(\tau_N\leq\sigma_N) -(2+\ga)\bEs[\log \sigma_{N+m} ] \notag\\
  &\geq   \bEs \log Z_{N,\gb}^{\sigma,f}(\tau_N\leq\sigma_N ) - \frac{2(2+\ga)}{\tga} \log N\, .
\end{align}
Since $Z_{N,\gb}^{\sigma,f}(\tau_N\leq\sigma_N) \geq \hZ_{N,\gb}^{\sigma}$, \eqref{tildettilde} proves that
\[\hat\mF(\gb) = \lim_{N\to\infty} \frac1N \log \hZ_{N,\gb}^{\sigma} = \lim_{N\to\infty} \frac1N \bEs \log Z_{N,\gb}^{\sigma,f}(\tau_N\leq\sigma_N) \, \qquad \bPs-a.s.,\]
and we therefore work with $Z_{N,\gb}^{\sigma,f}(\tau_N\leq\sigma_N)$ instead of $\hZ_{N,\gb}^{\sigma}$.

\smallskip
We continue notations from Section \ref{sec:compareFbar}: we take $v_0$ such that $\bPs(\sigma_n\geq v_0 b_n)\leq 1/4$, use $G_L$ from \eqref{eq:defAtilde}, take $L\geq L_0$ large so that \eqref{eq:condL} holds, and consider the set of good blocks $\mI=\{i:\,  \mB_i \in G_L\}$.
The idea of the proof is similar to that of Proposition \ref{prop:compareFbar}, but in addition, we need to control the size of $\tau_{L}$ on good blocks.  
\medskip

{\bf Case 1: $\bEs[\sigma_1]=+\infty$.} Here $b_n = \tilde\psi(n) n^{1/\tga}$, see \eqref{defanbn2}.
We need the following technical lemma, whose proof is postponed to Section~\ref{sec:lemmas}.

\begin{lemma}
\label{lem:logtausigma}
Assume \eqref{eq:alphas} and \eqref{nozero} with $\tga\neq 0$.
If $\bE[\sigma_1]=+\infty$, then, uniformly for $x\geq 1/10\ $,
\[\lim_{n\to\infty} \frac{1}{n} \bEs[\ind_{\{\sigma_{ xn} \geq  d_\tau n\}} \log \bPt(\tau_{n} \leq \sigma_{x n}) ]= 0 \, .\]
\end{lemma}

We consider $m$ blocks of length $L$, so $N=mL$, and define the events
\[
E_1:=\{\sigma:|\mI\cap [1,\tfrac78 m]| \geq \tfrac14 m\}, \quad E_2:= \{\sigma:\sigma_{mL}-\sigma_{7mL/8} \geq  d_\tau mL\},
\]
satisfying $\bPs(E_1)\to 1$ as $m\to\infty$, by \eqref{eq:condL}. 
Using that $\log Z_{mL,\gb}^{\sigma,f}(\tau_{mL}\leq\sigma_{mL}) \geq  \log \bPt(\tau_{mL} \leq \sigma_{mL})$, we get
\begin{equation} \label{e1e2bound}
\bEs \big[\log Z_{mL,\gb}^{\sigma,f}(\tau_{mL}\leq\sigma_{mL}) \big] \geq \bEs\big[  \ind_{E_1\cap E_2} \log Z_{mL,\gb}^{\sigma,f}(\tau_{mL}\leq\sigma_{mL})\big]
+ \bE_{\sigma}[\log \bPt(\tau_{mL} \leq \sigma_{mL})] \, .
\end{equation}
According to Lemma \ref{lem:logtausigma}, and because of \eqref{nozero}, we have 
\[
\lim_{m\to\infty} \frac{1}{mL} \bE_{\sigma}[\log \bPt(\tau_{mL} \leq \sigma_{mL})]=0,
\]
so we consider the first term on the right in \eqref{e1e2bound}.

\smallskip

On the event $E_1\cap E_2$, as in the proof of Proposition \ref{prop:compareFbar} we restrict to $\tau$ visiting only good blocks $\mB_i$, including visits to the endpoints $\sigma_{(i-1)L},\sigma_{iL}$.
Since (by definition of $G_L$) on each good block $ \mB_i$ there are at most $v_0 b_L$ $\tau$-renewals, up to $\tau_m$ such $\tau$ visit at least $\ell_m := \lfloor m /(v_0 b_{L}) \rfloor$ good blocks. We also choose $L$ large so that $\ell_m \leq \frac18 m$: on the event $E_1$, it ensures that $i_{\ell_m} \leq 7m/8$.

We denote $k_0\leq mL$ the index such that $\tau_{k_0}=\sigma_{i_{\ell_m} L}$: since $\tau_{mL}-\tau_{k_0}\leq \tau_{mL}$, we have on $E_1$, restricting to $\tau$ visiting the first $\ell_m$ good blocks,
\begin{align*}
\log &Z_{N,\gb}^{\sigma,f}(\tau_N\leq\sigma_N) \\
&\geq \sum_{k=1}^{\ell_m} \log \bP_{\tau}(\tau_1 = \sigma_{(i_k-1)L} - \sigma_{i_{k-1}L})  +  \ell_m\, \frac12 \mF(\gb) L + \log \bPt(\tau_{mL} \leq \sigma_{mL} - \sigma_{i_{\ell_m}L}).
\end{align*}
On the event $E_1\cap E_2$ we have $\sigma_{mL} - \sigma_{i_{\ell_m}L} \geq \sigma_{mL} - \sigma_{7mL/8} \geq  d_\tau mL$, so we then get
\begin{align}\label{ttildelast}
\bEs\big[  &\ind_{ E_1\cap E_2 }  \log Z_{N,\gb}^{\sigma,f}(\tau_N\leq\sigma_N) \big] \notag\\
&\geq \sum_{k=1}^{\ell_m} \bEs \big[ \log \bP_{\tau}(\tau_1 = \sigma_{(i_k-1)L} - \sigma_{i_{k-1}L})  \big]  \notag\\
&\hspace{2cm} + \frac12 \ell_m\mF(\gb) L \bPs(E_1\cap E_2) + \bEs\big[\ind_{E_2 } \log\bPt(\tau_{mL} \leq \sigma_{mL} - \sigma_{i_{\ell_m}L})  \big] 
 \notag \\
&\geq  \ell_m  \left( \frac14 \mF(\gb) L+ \bEs \log \bPt(\tau_1 = \sigma_{(i_1-1)L})  \right)
 +  \bEs \left[ \ind_{\{\sigma_{mL/8} \geq d_\tau mL\}}\log \bPt( \tau_{mL} \leq \sigma_{mL/8} ) \right]  ,
\end{align}
where we used that for $L$ large enough, $\bPs(E_1\cap E_2)\geq 1/2$.
Therefore dividing by $mL$ and letting $m\to\infty$ gives
\begin{equation}
L\hat\mF(\gb) \geq \frac{1}{4 v_0 b_{L}}\, \mF(\gb) L + \frac{1}{v_0 b_{L}} \bEs[ \log \bPt(\tau_1 = \sigma_{(i_1-1) L}) ]\, ,
\end{equation}
where we used Lemma \ref{lem:logtausigma} to get that 
\[
\lim_{m\to\infty}\frac{1}{mL}  \bEs \Big[ \ind_{\{\sigma_{mL/8} \geq d_\tau mL\}}\log \bPt( \tau_{mL} \leq \sigma_{mL/8} ) \Big]=0.
\]
Then we finish as in the proof of Proposition \ref{prop:compareFbar}: there exists some $L_0$ such that, for $L\geq L_0$, one has as in \eqref{lowboundZbar}
\begin{equation}
L\hat\mF(\gb)\geq \frac{1}{v_0 b_{L}} \Big( \frac14 L \mF(\gb) -  \frac{4(2+\ga)}{\tga \wedge 1}  \log L\Big) \, ,
\end{equation}
and then taking $L$ sufficiently large shows $\hat\mF(\beta)>0$.

\medskip
{\bf Case 2.} We now deal with the case when $\mu_{\sigma}:=\bEs[\sigma_1]<+\infty$ and $\bEt[\tau_1]\leq \bEs[\sigma_1]$. Here $b_n=\mu_{\sigma} n$, see \eqref{defanbn}.
Let us fix $m\in\NN$ large, and consider a system consisting in $m$ blocks of length $L$.
Decomposing according to whether the first block is good or not we have, recalling $G_L$ from \eqref{eq:defAtilde},
\begin{align}
\label{eq:tildeFcase1}
\bEs[\log Z_{mL,\gb}^{\sigma, f} (\tau_{mL} \leq \sigma_{mL}) ] \geq  \bEs \big[ \ind_{\{ \mB_1\in G_L\}} &\ind_{\{\sigma_{mL} -\sigma_L \geq  d_{\tau} m L\}} \log  Z_{mL,\gb}^{\sigma, f}(\tau_{mL} \leq \sigma_{mL}) \big] \notag \\
& + \bE_{\sigma}[\log \bPt(\tau_{mL} \leq \sigma_{mL})].
\end{align}

Recalling the indicator $\ind_{\{\sigma_L \in \tau\}}$ in the definition of $Z_{L,\gb}^{\sigma}$, if $\tau_{k_0}=\sigma_L$ for some $k_0$, then the relation $\tau_{mL}-\tau_{k_0}\leq \tau_{mL} $ guarantees that 
\begin{equation} \label{oneblock}
  Z_{mL,\gb}^{\sigma,f} (\tau_{mL} \leq \sigma_{mL}) \geq Z_{L,\gb}^{\sigma} \times \bPt ( \tau_{mL} \leq \sigma_{mL} - \sigma_{L}) .
  \end{equation}
Provided $m$ exceeds some $m_0$ we have $d_{\tau} m/(m-1) <\bE[\tau_1]\leq \mu_{\sigma}$, and therefore for sufficiently large $L$ we have $\bPs(\sigma_{(m-1)L} \geq d_{\tau} m L ) \geq 1/2$.
Since $\mB_1$ is independent of all other blocks and $\bPs(G_L)\geq 1/2$, it then follows from \eqref{oneblock} that for $\sigma\in G_L$,
\begin{multline}
\label{eq:tildeFcase2}
\bEs \big[ \ind_{\{ \mB_1\in G_L\}} \ind_{\{\sigma_{mL} -\sigma_L \geq  d_{\tau} m L\}} \log Z_{mL,\gb}^{\sigma,f} (\tau_{mL} \leq \sigma_{mL})\big]\\
 \geq \frac18  \mF(\gb) L + \frac 12 \bEs[\ind_{\{\sigma_{(m-1)L}\geq d_{\tau} mL\}} \log \bPt( \tau_{mL}\leq \sigma_{(m-1)L}) ]\, .
\end{multline}

It remains only to control $\bEs[\ind_{\{\sigma_{(m-1)L}\geq d_{\tau} mL\}}\log \bPt(\tau_{mL}\leq \sigma_{(m-1)L})]$. To that end we have the following lemma which we prove in Section \ref{sec:lemmas}.
\begin{lemma}
\label{lem:logtausigma1}
If $\bE[\tau_1]\leq\bE[\sigma_1]<+\infty$, then, for any $\gep>0$, for sufficiently large $m$ we have
\begin{equation}
\label{eq:lem:logtausigma1}
\begin{split}
\lim_{L\to\infty}\frac{1}{L}\bEs[\ind_{\{\sigma_{(m-1)L} \geq d_{\tau} mL\}}\log \bPt( \tau_{mL}\leq \sigma_{(m-1)L}) ] \geq - \gep\, ,\\
\text{ and } \quad \lim_{L\to\infty}\frac{1}{L}\bEs[\log \bPt( \tau_{mL}\leq \sigma_{mL}) ] \geq - \gep\, .
\end{split}
\end{equation}
\end{lemma}

We choose $\gep=\frac{1}{40} \mF(\gb)$, and fix $m_1\geq m_0$ such that one can apply Lemma \ref{lem:logtausigma1}. Then, we take $L$ large enough  such that $\frac{1}{L}\bEs[\ind_{\{\sigma_{(m_1-1)L} \geq d_{\tau} m_1L\}}\log \bPt( \tau_{m_1L}\leq \sigma_{(m_1-1)L}) ]\geq - \frac{1}{32} \mF(\gb) $ and also $\frac{1}{L}\bEs[\log \bPt( \tau_{m_1 L}\leq \sigma_{m_1 L}) ] \geq - \frac{1}{32} \mF(\gb)$. Combining this with \eqref{eq:tildeFcase1} and \eqref{eq:tildeFcase2}, we obtain for sufficiently large $L$
\begin{equation}
\label{eq:boundlogZ}
\bEs \big[ \log Z_{m_1 L,\gb}^{\sigma,f}(\tau_{m_1 L} \leq \sigma_{m_1 L}) \big] \geq  \frac{1}{16} \mF(\gb) L \, .
\end{equation}

To conclude the proof of Proposition \ref{prop:compareFtilde}, we use that $(\log\hZ_{N,\gb}^{\sigma})_{N\in\NN}$ is an ergodic super-additive sequence, so that $\hat\mF(\gb) \geq \sup_{N\in\NN} \frac{1}{N} \bEs[\log\hZ_{N,\gb}^{\sigma}] $. In view of \eqref{tildettilde}, we therefore have that, if $L$ is large enough, for $q$ as specified after \eqref{NqvsN},
\begin{equation}
\hat\mF(\gb) \geq \frac{1}{m_1 L+ q} \bEs \big [\log Z_{m_1 L,\gb}^{\sigma,f}(\tau_{m_1 L} \leq \sigma_{m_1 L}) \big] - \frac{3(2+\ga)}{\tga} \frac{\log( m_1L)}{m_1 L+q}\, .
\end{equation}
Then we can use \eqref{eq:boundlogZ} to obtain, by taking $L$ large enough,
\begin{equation}
\hat\mF(\gb) \geq \frac{1}{m_1 L+q}\ \frac{1}{20} \mF(\gb) L \geq \frac{1}{30 m_1} \mF(\gb) > 0\, .
\end{equation}

\subsection{Proof of Lemmas \ref{lem:logtausigma}-\ref{lem:logtausigma1}}
\label{sec:lemmas}

\begin{proof}[Proof of Lemma \ref{lem:logtausigma}]
We have the following crude bound: there exists a constant $c_{40}>0$ such that, for every $k\geq d_\tau n$, 
\begin{equation}
\label{tausmall}
\bPt(\tau_n\leq k) \geq  \bPt\big( \max_{1\leq i\leq n} \{\tau_{i}- \tau_{i-1}\}  \leq k/n \big) =\Big(1- \bPt\big(\tau_1> \tfrac{k}{n} \big) \Big)^{n}  \geq e^{-c_{40}\,  \bPt(\tau_1> \frac{k}{n})\, n }\, ,
\end{equation}
where the last inequality follows from the fact that $\bPt\big(\tau_1> \tfrac{k}{n} \big)$ is bounded away from 1, for all $k\geq d_\tau n$.

Since $\bEs[\sigma_1]=+\infty$, we may choose a sequence $\ga_n$ with $\ga_n/n\to +\infty$, and uniformly in $x\geq 1/10$, $\bPs(\sigma_{xn} \leq \ga_n) \stackrel{n\to\infty}{\to} 0$. We get
\begin{align}
0&\geq \frac1n \bEs\big[\ind_{\{\sigma_{xn} \geq d_\tau n\}}\log \bPt(\tau_n \leq \sigma_{xn}) \big]  \notag\\
& \geq \frac1n \bPs(\sigma_{xn} > \ga_n) \log \bPt(\tau_n \leq \ga_n) + \frac1n \bPs(\sigma_{xn} \leq \ga_n) \log \bPt(\tau_1=d_\tau)^n \notag\\
&\geq   - c_{41}\,  \bP(\tau_1 >\ga_n/n) +  \bPs(\sigma_{xn} \leq \ga_n) \,  \log \bPt(\tau_1=d_\tau) ,
\end{align}
where we used \eqref{tausmall} in the second inequality. Letting $n\to\infty$, we see that the limit is $0$ thanks to our choice of $\ga_n$.
\end{proof}

\begin{proof}[Proof of Lemma \ref{lem:logtausigma1}]
From the standard large deviation principle for i.i.d.~sums, we  can define the rate function
\begin{equation}
\label{ldp}
\bar J_\tau(t):=J_\tau(\mu_\tau-t) = \lim_{n\to\infty} -\frac1n \log \bPt(\tau_n \leq \mu_{\tau} n  - t n) , 
\end{equation} 
with $X:=\mu_{\tau}-\tau_1$, and $J_\tau$ is defined after \eqref{ratefunc}.  We define $\bar J_\sigma$ analogously.

It is standard that, since $E[X]=0$, we have $\bar J_\tau(t) = o(t)$ as $t\searrow0$.
Therefore for fixed $\gep>0$, for large $m$ we have $\bar J_\tau(2\mu_\tau/m)\leq \gep/m$.
We also have, using that $\mu_{\tau}\leq \mu_{\sigma}$, that for any $m\geq 2$
\[
    \bPs\left(\sigma_{(m-1)L} \leq (m-2)L\mu_{\tau} \right) \to 0\quad \text{as } L\to\infty .
   \]
Hence, observing that $(m-2) \mu_{\tau} \geq  d_{\tau} m$ provided that $m$ has been fixed large enough, we can get
\begin{align*}
\bEs &\big[\ind_{\{\sigma_{(m-1)L} \geq  d_{\tau} m L\}}  \log \bPt(\tau_{m L} \leq \sigma_{(m-1)L}) \big]   \\
& \geq  \log \bPt \big( \tau_{mL} \leq (m-2)L\mu_{\tau} \big) + \bPs \big( \sigma_{(m-1)L} \leq (m-2)L\mu_{\tau} \big) \times m L \log \bPt(\tau_1=d_\tau) \, .
\end{align*}
Hence we obtain, using \eqref{ldp}
\begin{equation}
 \liminf_{n\to\infty} \frac1L \bEs[\ind_{\{\sigma_{(m-1)L} \geq m L\}}\log \bPt(\tau_{mL} \leq \sigma_{(m-1)L})] \geq - m\, \bar J_\tau(\frac{2\mu_{\tau}}{m}) \geq -\gep\, ,
\end{equation}
which gives the first line in \eqref{eq:lem:logtausigma1}.

Similarly, decomposing according to whether $\sigma_{mL}\leq (m-2)L \mu_{\tau}$ or not, we have the lower bound
\begin{align*}
\bEs &\big[  \log \bPt(\tau_{m L} \leq \sigma_{mL}) \big]   \\
& \geq  \log \bPt \big( \tau_{mL} \leq (m-2)L\mu_{\tau} \big) + \bPs \big( \sigma_{mL} \leq (m-2)L\mu_{\tau} \big) \times m L \log \bPt(\tau_1=d_\tau) \, ,
\end{align*}
which in turn also gives that
$ \liminf_{n\to\infty} \frac1L \bEs[\log \bPt(\tau_{mL} \leq \sigma_{mL})] \geq -\gep $.
\end{proof}

\begin{appendix}

\section{Technical results on renewal processes}

\subsection{Some estimates on renewal processes}
\label{sec:renewestimates}

First of all, we state a result that we use throughout the paper, which is Lemma A.2 in \cite{GLT09} (that was slightly generalized in \cite{ABinter} to cover the case $\ga=0$).
\begin{lemma}
\label{lem:condtioning}
Assume that $\bP(\tau_1=k) = \gp(k) k^{-(1+\ga)}$ for some $\ga\geq 0$ and some slowly varying function $\gp(\cdot)$.
Then, there exists a constant $C_0>0$ such that, for all sufficiently large $n$, for any non-negative function $f_n(\tau)$ depending only on $\tau\cap\{0,\ldots,n\}$, we have
\[\bE[f_n(\tau) \mid 2n \in\tau] \leq C_0 \bE[f_n(\tau)]\, .\]
\end{lemma}

\medskip
In the rest of this section, we consider a renewal $\tau$ satisfying \eqref{eq:alphas} with $\ga\in (0,1)$.
Similar results exist for $\ga=0$ (see \cite{ABzero}) and for $\ga>1$ (see Appendix~A in \cite{BLpin}), but we do not need them here.

\smallskip
As noted in Section \ref{sec:notations}, for $a_n$ as in \eqref{defanbn2}, $\tau_n/a_n$ converges to an $\ga$-stable distribution with some density $h$, which is bounded and satisfies
\begin{equation} \label{stabletail}
  h(x) \sim C_1 \, x^{-(1+\ga)} \quad \text{as } x \to \infty, \quad h(x) \leq C_2 x^{-(1+\ga)} \quad \text{for all } x>0\, .
\end{equation}
Further, by the local limit theorem for such convergence, see \cite[\textsection 50]{GnedenkoKolmo}, and the fact that $a_k^\ga\sim k\gp(a_k)$, for any given $0<\theta<K<\infty$ we have as $k\to\infty$
\begin{equation} \label{locallim}
\begin{split}
  \bPt(\tau_k=m) \sim \frac{1}{a_k} h \left( \frac{m}{a_k} \right) \leq C_2 a_k^\ga m^{-(1+\ga)}  \leq 2C_2 km^{-(1+\ga)}\gp(m) \\  
   \text{ uniformly over } m \in [\theta a_k,Ka_k].
  \end{split}
\end{equation}
Moreover, Doney \cite[Thm. A]{Doney} gives that,
uniformly in $m\gg a_k$,
\begin{equation}
\label{locallargedev}
\bPt(\tau_k = m) = (1+o(1))\,  k\, \bPt(\tau_1=m) \quad \text{as $k\to\infty$} .
\end{equation}
Together with \eqref{locallim} and Lemma \ref{smallk} below, which deals with the case $m\ll a_k$, we thereby obtain the following uniform bound.

\begin{lemma}
\label{lem:uniform}
Assume $\ga\in(0,1)$.
There exists a constant $c_{42}>0$ such that, for $k$ large enough and all $m\geq k$,
\[\bPt(\tau_k = m ) \leq  c_{42} \, k \, \bPt (\tau_1 = m) \, ,\]
and 
\[\bPt(\tau_k \geq m) \leq c_{42}\, k\, \bPt(\tau_1 \geq m) \, .\]
\end{lemma}

For the lower tail we have the following.

\begin{lemma} \label{smallk}
Assume $\ga \in (0,1)$.  
There exist a constant $c_{43}$ such that, for all $\epsilon <1/2$ and $n\geq 1$
\begin{equation}
\label{eq:smallk}
\bPt(\tau_n \leq \epsilon\, a_n) \leq \exp\left( - c_{43}\, \epsilon^{- \frac34 \frac{\ga}{1-\ga}} \right).
\end{equation}
In particular, for any $\gamma\in(0,1)$, for all $n\geq 1$,
\[\bPt(\tau_n \leq n^{(1-\gamma)/\ga}) \leq \exp \left( - c_{44}\,  n^{\frac{\gamma }{2 (1-\ga)}}\right) \, .\]
\end{lemma}

Note that this lemma implies that the density $h$ also satisfies
\begin{equation}
\label{hforsmall}
h(x)\leq \exp \Big( -c_{45}\, x^{-\frac{\ga}{2(1-\ga)}} \Big)\, ,
\end{equation}
for some constant $c_{45}>0$ and sufficiently small $x$.
%

\begin{proof}[Proof of Lemma \ref{smallk}]
Let $\Lambda(t)=\log \bEt[e^{ t \tau_1}]$, so that for all $t>0$,
\begin{equation}
\label{expobound1}
\bPt(\tau_n \leq \epsilon\, a_n) \leq \exp( \epsilon\, a_n t + \Lambda(-t)\, n).
\end{equation}
By \eqref{eq:alphas} and standard properties of the Laplace transform, there is a constant $c_{\ga}$ such that
\begin{equation}\label{gLapprox}
  \Lambda(-t) \sim - c_{\ga} \, t^{\ga}\, \gp\left( \tfrac{1}{t} \right) \quad \text{as } t\searrow 0.
\end{equation}
We have $\Lambda(-t)< - d_\tau t$ for all $t\in(0,+\infty)$, and $\Lambda(-t)\sim -d_\tau t$ as $t\to+\infty$, for $d_\tau$ from \eqref{ddef}.
In the rest of the proof, we assume $\epsilon a_n \geq d_\tau n$, since otherwise the probability is $0$.

\smallskip
We can approximately optimize \eqref{expobound1} by taking $t=t_n=t_n(\epsilon)$ given by
\begin{equation}\label{tndef}
\frac{\gL(-t_n)}{t_n} = \frac{- 2 \epsilon a_n}{n}.
\end{equation}
Such a solution exists since  $t^{-1}\gL(-t)\to -\infty$ as $t\to 0$, $t^{-1}\gL(-t)\to -d_\tau$ as $t\to \infty$, and $2\epsilon a_n/n \geq 2d_\tau$. 
We therefore end up with
\begin{equation}
\label{expobound2}
\bPt(\tau_n \leq \epsilon\, a_n) \leq \exp( - \epsilon\, a_n t_n),
\end{equation}
so we need a lower bound on $a_nt_n$.

\smallskip
Let $c_{45}$ be given by $c_{45}^{-1}\gL(-c_{45})=-2d_\tau$; then $t_n \in(0, c_{45}]$, since $\epsilon a_n \geq d_\tau n$.  Letting $c_{46}$ large enough so that $\varphi(x)>0$ for all $x\geq c_{46}/c_{45}$, we then have that $\gp(c_{46}/t_n)>0$.  Let $\lambda=(1-\ga/4)/(1-\ga)$.  By \eqref{defanbn}, since $a_n\to\infty$ and $\epsilon a_n \geq d_\tau n$, there exists $n_0$ such that for $n\geq n_0$,
\begin{equation}\label{an1}
 4\epsilon a_n^{1-\ga}\gp(a_n) \geq \frac{2\epsilon a_n}{n} \geq 2d_\tau \quad \text{and} \quad
   \frac{d_\tau}{2a_n^{1-\ga}\gp(a_n)} \geq \left( \frac{\bar d_\tau}{a_n} \right)^{1/\lambda},
\end{equation}
which together show that $\epsilon^\lambda a_n \geq \bar d_\tau$. Therefore there exists $c_{47}$ such that 
\[
  \frac{\gp(\epsilon^\lambda a_n)}{\gp(a_n)} \geq c_{47}\epsilon^{\frac14 \ga},
\]
which with \eqref{gLapprox}, \eqref{tndef}, and the first inequality in \eqref{an1} shows that for some $c_{48}$,
\[
 (\epsilon^\lambda a_n)^{1-\ga} \gp(\epsilon^\lambda a_n) \geq c_{47}\epsilon a_n^{1-\ga}\gp(a_n) \geq \frac{c_{47}\epsilon a_n}{2n}
   \geq c_{48}\left( \frac{c_{46}}{t_n} \right)^{1-\ga} \gp\left( \frac{c_{46}}{t_n} \right).
\]
We emphasize that the constants $c_i$ and $n_0$ do not depend on $\epsilon$. It follows that $\epsilon^\lambda a_n \geq c_{49}/t_n$, or equivalently $\epsilon a_nt_n \geq c_{49}\epsilon^{-(\lambda-1)}$, which with \eqref{expobound2} completes the proof for $n\geq n_0$.

We finish by observing that for $q=\max_{n<n_0} a_n$, for all $n<n_0$ and $\epsilon\leq 1/2$ we have
\[
  \bPt(\tau_n\leq \epsilon a_n) \leq \bPt(\tau_1 \leq q) < 1,
\]
and $\bPt(\tau_n\leq \epsilon a_n)=0$ if $\epsilon<1/q$.  Therefore after reducing $c_{43}$ if necessary, \eqref{eq:smallk} also holds for $n<n_0$.
\end{proof}

Let us finally prove some analogue of Lemma \ref{smallk} in the case $\ga=1$ with $\bE[\tau_1]=+\infty$, that will be needed below. Recall in that case the definition \eqref{defanbn} of $a_n$, and that $\tau_n/a_n \to 1$ in probability.
\begin{lemma}
\label{smallk2}
Assume that $\ga=1$ and that $\bar m(x) = \bEt[\tau_1\wedge x] \to+\infty$.
For every $\gd>0$, there exists a constant $c_{\gd}>0$ such that, for any $\epsilon \in(0,1/2)$ and $n\ge 1$
\[\bPt\big( \tau_n \le \epsilon a_n \big)\leq \exp \big(- c_{\gd} \epsilon^{ - 1/\gd}\big) \, .\]
\end{lemma}
Of course, this result is not optimal when $\epsilon$ is not close to $0$, but it is sufficient for our purpose.

\begin{proof}
We use the same notations as for the proof of Lemma \ref{smallk}. The difference here is that, as $t\downarrow 0$, $\Lambda(-t) \sim  - t \bar m(1/t)$, where we recall that $\bar m(x) = \bEt[\tau_1\wedge x]$ diverges as a slowly varying function as $x\to\infty$. Hence, we still have that $t^{-1} \Lambda(-t) \to -\infty$ as $t\downarrow 0$, and as in \eqref{tndef} we may define $t_n = t_n(\epsilon)$ as the solution of 
\begin{equation}
\label{deftn2}
\Lambda(-t_n) t_n^{-1} = - 2\epsilon a_n /n \, .
\end{equation}
Then we estimate $t_n$: using that $a_n \sim n \bar m (a_n)$, and that $\Lambda(-t) \sim - t \mu(1/t)$, we get that there is a constant $c_{50}$ such that for any $n\geq 1$,
\[\bar m(1/t_n) \le c_{50} \epsilon \bar m(a_n) \, .\]
Using Potter's bound, we get that for any $\gd'>0$ there is a constant $c_{\gd'}$ such that for any $n \ge 1$ $\bar m (a_n)\le  c_{\gd'} ( t_n a_n)^{\gd'} \bar m(1/t_n) $. In the end, we obtain that there is a constant $\tilde c_{\gd'}$ (independent of $\epsilon$) such that for all $n\ge 1$
\[ c_{50} c_{\gd'} ( t_n a_n)^{\gd'} \ge \epsilon^{-1}  , \qquad i.e.\ \ t_n a_n \ge \tilde c_{\gd'} \epsilon^{-1/\gd'} .  \]
Therefore, we obtain as in \eqref{expobound2}
\[\bPt(\tau_n \le \epsilon a_n) \le \exp\big( - \epsilon a_n t_n \big) \le \exp\big(  - \tilde c_{\gd'} \epsilon^{1-1/\gd'} \big)\, ,\]
which ends the proof of Lemma \ref{smallk2}, by taking $\gd=\gd'/(1-\gd')$.
\end{proof}

\subsection{Proof of Lemma \ref{lem:Kstar} and of Proposition \ref{prop:annealed}}
\label{app:K*}

\begin{proof}[Proof of Lemma \ref{lem:Kstar}]
From \eqref{Doney}, we have that 
\begin{equation}
\label{renewaltheorem}
\bPt(k\in\tau) \stackrel{n\to\infty}{\sim} \varphi_0(k) \, k^{-(1-\ga\wedge 1)} ,\quad \text{ with } \varphi_0(x) =
\begin{cases}
\frac{\gp(x)}{\bP(\tau_1\geq x)^2} & \text{ if } \ga=0\, ,\\
 \frac{\ga \sin(\pi\ga)}{\ga} \gp(x)^{-1} & \text{ if } \ga\in(0,1)\, .
\end{cases}
\end{equation}

\medskip
Let us start with the case $\tga\in(0,1)$.
We fix $0<\theta<K<\infty$ and $\delta>0$, and (recalling the definition \eqref{defanbn} of $b_n$) we split $\bP_{\sigma,\tau}(\sigma_n\in\tau)$ into three sums:
\begin{align}\label{threesums}
  \bP_{\sigma,\tau}(\sigma_n\in\tau) &= \sum_{k=n}^\infty \bPs(\sigma_n=k)\bPt(k\in\tau) \notag\\
  &\leq \sum_{0 \leq j\leq\log(\theta b_n)} \bP_\sigma(\sigma_n \in (e^{-(j+1)}\theta b_n,e^{-j}\theta b_n])
    \max_{k \in (e^{-(j+1)}\theta b_n,e^{-j}\theta b_n]} \bP_\tau(k \in \tau) \notag\\
  &\qquad + \sum_{\theta b_n < k < Kb_n} \bPs(\sigma_n=k)\bPt(k\in\tau)\ \  + \ \ \bPs(\sigma_n \geq Kb_n) \max_{k \geq Kb_n} \bPt(k\in\tau).
\end{align}

Let us first focus on the second sum, which gives the main contribution (and is also a lower bound).
Using \eqref{locallim} and  \eqref{renewaltheorem}, it is asymptotic to 
\begin{equation} \label{second}
   \varphi_0(b_n)\, b_n^{-(1-\alpha\wedge 1)} \sum_{\theta b_n < k < Kb_n}  \frac{1}{b_n} h\left( \frac{k}{b_n} \right) \left( \frac{k}{b_n} \right)^{-(1-\ga\wedge 1)} 
    \stackrel{n\to\infty}{\sim} \varphi_0(b_n) \,  b_n^{-(1-\alpha \wedge 1)} \int_\theta^K x^{-(1-\ga \wedge 1)}h(x)\ dx.
\end{equation}

If $K$ is sufficiently large (depending on $\delta$), then thanks to \eqref{renewaltheorem}, the third sum on the right in \eqref{threesums} is bounded by
\[
 c_{51}\,  (K b_n)^{-(1-\ga\wedge 1)} \varphi_0(b_n) \leq \delta b_n^{-(1-\ga\wedge 1)} \varphi_0(b_n).
\]

If $\theta$ is sufficiently small (depending on $\delta$), then  $\phi(e^{j}/\theta)\geq (e^{j}/\theta)^{-\tga/2(1-\tga)}$, and thanks to Lemma \ref{smallk} and \eqref{renewaltheorem}, the first sum on the right in \eqref{threesums} can be bounded by
\begin{equation}\label{first}
  c_{52} \sum_{0 \leq j\leq\log(\theta b_n)} \exp\left( - c_{53} \Big( \frac{e^j}{\theta} \Big)^{\frac{\tga}{2(1-\tga)}}  \right) \times
    \left( \frac{\theta b_n}{e^j} \right)^{-(1-\ga\wedge 1)} \varphi_0\left( \frac{\theta b_n}{e^j} \right) \\
  \leq \delta  b_n^{-(1-\ga\wedge 1)} \varphi_0(b_n)\, .
\end{equation}

By \eqref{stabletail}, the integral in \eqref{second} remains bounded as $K\to\infty$, and it also remains bounded as $\theta\to 0$ thanks to \eqref{hforsmall}. Because $\delta$ is arbitrary and the second sum \eqref{second} is also a lower bound for $\bPst(\sigma_n\in\tau)$, it follows that 
\begin{equation}
\label{eq:asympPsigmaintau}
\bP_{\sigma,\tau}(\sigma_n\in\tau) \stackrel{n\to\infty}{\sim} \Big( \int_0^{+\infty} x^{-(1-\ga\wedge 1)} h(x)\ dx \Big) \times  b_n^{-(1-\alpha\wedge 1)}\varphi_0(b_n) .
\end{equation}

We therefore conclude, thanks to \eqref{defAnBn}, that
\[ \bP_{\sigma,\tau}(\sigma_n\in\tau) = \gp^*(n) n^{-(1+\ga^*)}, \]
for $\ga^*=(1-\ga\wedge 1 -\tga)/\tga$ and for some slowly varying $\varphi^*$, given asymptotically via \eqref{eq:asympPsigmaintau}:
\begin{equation}\label{phistar}
\varphi^*(n) \stackrel{n\to\infty}{\sim} c_{54}\, \tilde\gp(b_n)^{-(1-\ga\wedge 1)/\tga} \varphi_0(b_n) \quad \text{with}\quad  c_{54}=  \int_0^{+\infty} x^{-(1-\ga\wedge 1)} h(x)\ dx .
\end{equation}

\medskip

It remains to treat the case $\tga \ge1$. In that case, $\sigma_n /b_n$ converges in probability to $1$. We fix $\gd>0$ and $\gep>0$, and split the probability into
\begin{align}
\bPst\big(\sigma_n & \in \tau  \big) = \bPst\big( \sigma_n \le (1-\gep)b_n , \sigma_n \in\tau \big) \notag \\
&+ \bPst\big(  \sigma_n \in \big( (1-\gep)b_n ,(1+\gep) b_n \big) , \sigma_n\in\tau \big)  + \bPst\big(\sigma_n \ge (1+\gep)b_n , \sigma_n \in\tau \big)\, .
\label{threesums2}
\end{align}

The second term, which is the leading term, is bounded from below by
\begin{equation}
\bPst \big(  \sigma_n/b_n \in (1-\gep, 1+\gep) \big) \inf_{k\in \big( (1-\gep)b_n ,(1+\gep) b_n \big)} \bPt(k\in\tau) \ge (1-\gd) \gp_0(b_n) b_n^{-(1-\ga\wedge 1)} \, ,
\end{equation}
provided that $n$ is large enough, and that $\gep$ had been fixed small enough. To obtain this, we used \eqref{renewaltheorem}, and also that $\sigma_n/b_n$ converges to $1$ in probability.
Analogously, this is also bounded from above by $(1+\gd) \gp_0(b_n) b_n^{-(1-\ga\wedge 1)}$.

The last term in \eqref{threesums2} is bounded from above by
\begin{equation}
\bPst\big(\sigma_n \ge (1+\gep)b_n \big) \sup_{k\ge b_n} \bPt(k\in\tau) = o(1) \gp_0(b_n) b_n^{-(1-\ga\wedge 1)},
\end{equation}
where for the last identity holds we used \eqref{renewaltheorem}, together with the fact that $\sigma_n/b_n$ converges in probability to $1$.

It remains to bound the first term in \eqref{threesums2}. If $\bEs[\sigma_1]<+\infty$, then it is easy: there exists some $c_{\gep}>0$ such that
\begin{equation}
\bPs(\sigma_n \le (1-\gep) b_n ) \le e^{-c_{\gep} n}   = o( \gp_0(b_n) b_n^{-(1-\ga\wedge 1)} ) \, .
\end{equation}
Indeed, one simply uses that for any $\lambda>0$ $\bPs(\sigma_n \le (1-\gep) b_n ) \le e^{\lambda (1-\gep) b_n } \bE[e^{-\lambda \sigma_1}]^n$, and choose $\lambda = \lambda_{\gep}$ small enough so that $\bEs[e^{-\lambda \sigma_1}] \le 1- (1-\gep/2) \lambda \bEs[\sigma_1]$ (using that $1-\bEs[e^{-\lambda \sigma_1}] \sim \lambda \bEs[\sigma_1] $ as $\lambda\downarrow 0$).
We then get that, since $b_n=n \bEs[\sigma_1]$, $\bPs(\sigma_n \le (1-\gep) b_n ) \le \exp(- \tfrac12 \gep \lambda_{\gep} \bEs[\sigma_1] n)$.

When $\bEs[\sigma_1] =+\infty$, it is more delicate. We fix $\theta>0$ and bound the first term in \eqref{threesums2} by
\begin{align*}
\sum_{0\le j\le \log(\theta b_n)} &\bPs\big( \sigma_n \in (e^{-(j+1)} \theta b_n , e^{-j} \theta b_n] \big) \max_{k\in (e^{-(j+1)} \theta b_n , e^{-j} \theta b_n]} \bPt(k\in\tau) \\
&+ \bPs\big( \sigma_n \in (\theta b_n, (1-\gep) b_n]  \big) \max_{k \ge  \theta b_n} \bPt(k\in\tau)\, .
\end{align*}
The first term is treated as in \eqref{first}, using Lemma \ref{smallk2} in place of Lemma \ref{smallk}, and is bounded by $\gd\gp_0(b_n) b_n^{-(1-\ga\wedge 1)} $ provided that $\theta$ had been fixed small enough. For the second term, we use \eqref{renewaltheorem} to get that it is bounded by a constant times $ \bPs\big( \sigma_n \le(1-\gep) b_n  \big)  \theta^{-(1-\ga\wedge 1)} \gp_0(b_n) b_n^{-(1-\ga\wedge 1)}$, with $ \bPs\big( \sigma_n \le(1-\gep) b_n  \big)  \stackrel{n\to\infty}{\to} 0$.

In the end, and because $\gd$ is arbitrary, we proved that when $\tga\ge 1$
\[ \bPst\big(\sigma_n  \in \tau  \big) \stackrel{n\to\infty}{\sim} \gp_0(b_n) b_n^{-(1-\ga\wedge 1)}, \quad i.e.\ \ \bPst\big(\sigma_n  \in \tau  \big) =\gp^*(n) n^{-(1+\ga^*)},\]
with $\ga^* = (1- \ga\wedge 1 - \tga \wedge 1) /\tga\wedge 1$, and (recalling the definition \eqref{defanbn} of $b_n$) 
\[\gp^*(n) \stackrel{n\to\infty}{\sim} \gp_0(b_n)   \tolm(b_n)^{-(1-\ga \wedge 1)}\, .\]
\end{proof}

\begin{proof}[Proof of Proposition \ref{prop:annealed}]
We assume that $\gb>0$.
We use a Mayer expansion as in \eqref{expand1}, writing $e^{\gb} = 1+ (e^{\gb}-1)$, and expanding $(1+e^{\gb}-1)^{\sum_{n=1}^N \ind_{\sigma_n\in \tau}}$, to obtain
\begin{align*}
\bE_{\sigma} Z_{N,\gb}^{\sigma} = \sum_{m=1}^N (e^{\gb}-1)^m \sum_{1\le i_1 < \cdots< i_m =N} \bEst\Big[ \prod_{k=1}^{m} \ind_{\{\sigma_{i_k} \in\tau\}} \Big]\, .
\end{align*}
Then we may set $K^*_\gb(n) := (e^{\gb}-1) e^{-\mathtt{F} n} \bPst(\sigma_n \in\tau)$, where $\mathtt{F}$ is a solution to \eqref{eq:fann} if it exists, and set $\mathtt{F} = 0$ otherwise. Hence $(K^*_\gb(n))_{n\ge 1}$  is a sub-probability on $\NN$, and a probability if $\mathtt{F}>0$. We therefore get that
\[\bE_{\sigma} Z_{N,\gb}^{\sigma} = e^{\mathtt{F} N} \sum_{m=1}^N  \sum_{1\le i_1 < \cdots< i_m =N}  \prod_{k=1}^{m} K^*_{\gb}( i_k-i_{k-1}) = e^{\mathtt{F} N} \bP^*( N\in  \nu^* ) ,\]
where $\nu^*$ is a renewal process with inter-arrival distribution $(K^*_\gb(n))_{n\ge 1}$. Then because of Lemma~\ref{lem:Kstar}, we get that: either (i) $\mathtt{F}>0$ and then $\bP^*(N\in\nu^*)$ converges to the constant $\bE^*[\nu_1]^{-1}$, or (ii) $\mathtt{F}=0$ but $\sum_{n\ge 1} K^*_{\gb}(n) =1$ in which case $\ga^*\ge 0$ and $\bP^*(N\in \nu^*)$ is regularly varying with exponent $1-\ga^*\wedge 1$ (see \eqref{Doney} with $\gp,\ga$ replaced by $\gp^*,\ga^*$), or (iii) $\mathtt{F}=0$ and $\sum_{n\ge 1} K^*_{\gb}(n) <1$, in which case $\nu^*$ is transient (with $\ga^*\ge 0$), and  $\bP^*(N\in \nu^*)$ is asymptotically a constant multiple of $K^*_{\gb}(n)$, see for example \cite[App. A.5]{GBbook}. This shows that in any case $\lim_{N\to\infty} N^{-1} \log \bP^*(N\in\nu^*)=0$, and hence $\mF(\gb) =\mathtt{F}$.

\smallskip
From \eqref{eq:fann}, it is standard to obtain the critical point $\gb_c^{\a}= \log\big( 1+ \bEst[|\tau\cap\sigma|]^{-1}\big)$ and 
the behavior of $\mF(\gb)$ close to criticality, since the behavior of $\bPst(\sigma_n\in\tau)$ is known from Lemma \ref{lem:Kstar}. Let us recall briefly how this is done for the sake of completeness.

 In the recurrent case, \textit{i.e.}\ when $\sum_{n\ge 1} \bPst(\sigma_n\in\tau) = +\infty$ (and necessarily $\ga^* \in [-1,0]$), we use Theorem 8.7.3 in \cite{Bingham} to get that $\sum_{n\ge 1} e^{- \mathtt{F} n} \bPst(\sigma_n \in\tau)$ is asymptotically equivalent, as $\mathtt{F} \downarrow 0$, to an explicit constant times $\sum_{n=1}^{1/\mathtt{F}} \bPst(\sigma_n \in\tau) \sim L^*(1/\mathtt{F})\, (1/\mathtt{F})^{-\ga^*}$ for some slowly varying function $L^*(\cdot)$. Because of \eqref{eq:fann}, we therefore get that, as $\gb\downarrow 0$, $L^*(1/\mF(\gb))\, \mF(\gb)^{|\ga^*|} \sim \gb^{-1}$, and we obtain the critical behavior \eqref{fanncritic}.

In the transient case, \eqref{eq:fann} can be multiplied by $e^{\gb_c^{\a}} -1 = \bEst[|\tau\cap \sigma|]^{-1} >0$ (so that $K_{\gb_c^\a}^*(n)=(e^{\gb_c^\a}-1) \bPst(\sigma_n\in\tau)$ is a probability, equal to $K^*(n)$ of \eqref{eq:defK*}) to get 
\[\sum_{n\ge 1} e^{-\mathtt{F} n} K^*(n) = \frac{e^{\gb_c^{\a}} -1}{ e^{\gb}-1 }  \, .\]
This is the definition of the free energy of a standard homogeneous pinning model with inter-arrival distribution $K^*(n)$, and pinning parameter $\log((e^{\gb_c^\a}-1) /(e^{\gb}-1)) \stackrel{\gb\downarrow \gb_c^{\a}}{\sim} c(\gb-\gb_c^{\a})$ for some constant $c$, see \cite[Ch.~2]{GBbook}, in particular Equation (2.2) there. It is standard (as described in the previous paragraph) to obtain the critical behavior \eqref{fanncritic} from this using Lemma~\ref{lem:Kstar}.
\end{proof}

\end{appendix}

\bibliographystyle{plain}
\bibliography{biblio}

\end{document}